\documentclass[12pt]{amsart}
\usepackage{amsfonts}
\usepackage{amsmath}
\usepackage{amssymb}
\usepackage[margin=1.2in]{geometry}
\usepackage{bbm}
\usepackage{tikz}
\usepackage[hidelinks]{hyperref}
\usepackage{mathtools}
\allowdisplaybreaks[4]

\newtheorem{theorem}{Theorem}[section]
\newtheorem{proposition}[theorem]{Proposition}
\newtheorem{corollary}[theorem]{Corollary}
\newtheorem{lemma}[theorem]{Lemma}

\theoremstyle{definition}

\newtheorem{remark}[theorem]{Remark}
\theoremstyle{plain}
\newtheorem*{theoremrestated}{Theorem 1.1}
\theoremstyle{remark}

\numberwithin{equation}{section}
\numberwithin{theorem}{section}

\newcommand{\Z}{\mathbb{Z}}

\renewcommand{\Re}{\operatorname{Re}}
\renewcommand{\Im}{\operatorname{Im}}

\numberwithin{equation}{section}

\makeatletter
\@namedef{subjclassname@2020}{
	\textup{2020} Mathematics Subject Classification}
\makeatother

\makeatletter
\newcommand{\statement}[1]{\def\@statement{#1}}
\def\@statement{}

\let\old@setkeywords\@setkeywords
\def\@setkeywords{
	\old@setkeywords
	\ifx\@statement\@empty\else
	\par\medskip
	\@statement
	\par\medskip
	\fi
}
\makeatother

\begin{document}
	\title[Large value and zero density estimates for Dirichlet $L$-functions]{Large Value Estimates for Dirichlet Polynomials with Characters and Zero Density of Dirichlet $L$-Functions}
	
	\author{Bin Chen}
	\address{Bin Chen\\ Department of Mathematics: Analysis, Logic and Discrete Mathematics\\ Ghent University\\ Krijgslaan 281\\ B 9000 Ghent\\ Belgium}
	\email{bin.chen@UGent.be}
	\author{Vishal Gupta}
	\address{Vishal Gupta\\ Mathematical Institute, University of Oxford}
	\email{vishal.gupta@maths.ox.ac.uk}
	
	\author{Yung Chi Li}
	\address{Yung Chi Li\\ Department of Mathematics, Northwestern University}
	\email{yungchi928@gmail.com}
	\subjclass[2020]{Primary 11M06,  11M26; Secondary 11N05, 11N13.}
	\keywords{Large value estimates; The Guth-Maynard method; Zero-density estimate for Dirichlet $L$-functions; Distribution of primes.}
	\statement{After the preprint \cite{Chen2025} by the first author was released, the other authors informed BC that they had independently obtained the exponent $7/3$ in the zero density estimate. The present version is the result of the authors' subsequent collaboration on this topic.}
	\begin{abstract} 
		It is proved that
		\[
		\sum_{\chi \bmod q}N(\sigma,T,\chi)
		\ll_{\epsilon}
		(qT)^{7(1-\sigma)/3+\epsilon},
		\]
		where $N(\sigma,T,\chi)$ denotes the number of zeros
		$\rho=\beta+it$ of $L(s,\chi)$ in the rectangle
		$\sigma\leq \beta\leq 1$, $|t|\leq T$. The exponent $7/3$ improves upon
		Huxley's earlier exponent of $12/5$.
		
		The key innovation lies in deriving a sharp upper bound for sums over affine transformations of functions with a GCD twist, which arises from our adaptation of the Guth--Maynard method. As applications of the zero density estimates obtained in this paper, we derive a new upper bound for the least Goldbach number in arithmetic progressions modulo a prime and establish new results on primes in arithmetic progressions in short intervals, in particular for prime-power moduli.
	\end{abstract}

	\maketitle
	
	\section{Introduction}\label{intro}
	The primary objective of this paper is to extend the Guth-Maynard method to study the distribution of Dirichlet polynomials twisted by primitive Dirichlet characters. This yields improved large values estimates for such Dirichlet polynomials of length $N$, taking values of size close to $N^{3/4}$. Specifically, we prove
	\begin{theorem}[Large value estimate for Dirichlet polynomials with characters]\label{Partial LVE} Let
		\[
		D_{N}(t, \chi)=\sum_{N < n \leq 2N} a_{n}\chi(n)n^{it},
		\]
		where $\chi$ is a primitive character modulo $q$ and $|a_{n}| \leq1$. Suppose that $W$ is a finite set of pairs $(t, \chi)$,  where $|t| \leq T$ and for $(t, \chi)\neq (t', \chi')$ either $\chi \neq \chi'$ or $|t-t'|\geq 1$. Assume that
		\[ |D_{N}(t, \chi)| \geq V
		\]
		for all $(t, \chi) \in W$ and $N \geq(qT)^{\frac{2}{3}}$.
		Then for any divisor $q_1 \mid q$, we have 
		\[ |W| \leq (qT)^{o(1)} (N^{2}V^{-2}+ qq_1^{-\frac{1}{2}}T^{\frac{1}{2}}N^{3}V^{-4}+ q q_1^{\frac{1}{3}}TN^{2}V^{-4} + qTN^{\frac{12}{5}}V^{-4}).
		\]
		In all cases, we have
		\[ |W| \leq (qT)^{o(1)}(N^{2}V^{-2}+(qT)^{\frac{1}{2}}N^{3}V^{-4}+q^{\frac{4}{3}}TN^{2}V^{-4}+qTN^{\frac{12}{5}}V^{-4}).
		\]
	\end{theorem}
	When $q=1$, in the key range $N<T^{5/6}$, Theorem \ref{Partial LVE} recovers a recent bound of Guth-Maynard \cite{GM}, which this work is based on:
	\begin{align}
		|W| \leq T^{o(1)}(N^2V^{-2}+TN^{\frac{12}{5}}V^{-4}). \label{GM bound}
	\end{align}
	To properly contextualize the strength of Theorem \ref{Partial LVE}, we compare it with two standard large values estimates in the literature. First, the $qT$-analogue of the classical mean value theorem \cite[Thm 9.12]{iwanieckowalski} yields the bound
	\begin{equation}\label{qTmvt}
		|W| \leq (qT)^{o(1)}(N^{2}V^{-2}+qTNV^{-2}),
	\end{equation}
	while the Halász-Montgomery-Huxley method \cite{Huxley1975} gives 
	\begin{equation}\label{hmhlve}
		|W| \leq (qT)^{o(1)}(N^{2}V^{-2}+qTN^{4}V^{-6}).
	\end{equation}
	Zero-density estimates for the Riemann zeta function and Dirichlet $L$-functions occupy a central role in multiplicative number theory. These estimates yield important arithmetic consequences, particularly in prime number theory; we refer to \cite[Chapter 12]{Ivicbook} and \cite[Chapter 15]{MontgomeryBook} for comprehensive discussions of their applications. Large values estimates such as \eqref{qTmvt} and \eqref{hmhlve} are important tools in establishing zero density estimates for Dirichlet $L$-functions. Let $N(\sigma, T, \chi)$ denote the number of zeros $\rho = \beta + it$ of $L(s, \chi)$ in the rectangle $\sigma \leq \beta \leq 1$, $|t| \leq T$. Then \eqref{qTmvt} yields
	\begin{equation}\label{Inghaqt}
		\sum_{\chi \bmod q}N(\sigma , T, \chi) \leq (qT)^{\frac{3(1-\sigma)}{2-\sigma}+o(1)},
	\end{equation}
	whereas \eqref{hmhlve} gives the estimate 
	\begin{equation}\label{zeeshmh}
		\sum_{\chi \bmod q}N(\sigma , T, \chi) \leq (qT)^{\frac{3(1-\sigma)}{3\sigma-1}+o(1)} 	
	\end{equation}
	that performs better than \eqref{Inghaqt} when $\sigma > 3/4$. When $\sigma=3/4$,  the bounds \eqref{Inghaqt} and \eqref{zeeshmh} coincide. This threshold case corresponds to Dirichlet polynomials of length $N=(qT)^{4/5}$ attaining values of size $V=N^{3/4}$, for which both estimates yield $|W| \leq (qT)^{3/5+o(1)}$. In this critical situation, Theorem \ref{Partial LVE} gives us the estimate
	\begin{align*}
		|W| \leq (qT)^{o(1)}(qT(q_1T)^{-1/2}+q_1^{1/3}(qT)^{1/5}+(qT)^{13/25})    
	\end{align*}
	for any divisor $q_1 \mid q$. We may always choose $q_1 = q$ and get
	\begin{align*}
		|W| \leq (qT)^{o(1)}(q^{1/3}(qT)^{1/5}+(qT)^{13/25}) \leq (qT)^{\frac{8}{15}+o(1)},
	\end{align*}
	giving an improvement for all moduli $q$. Employing Theorem \ref{Partial LVE} within the zero density machinery leads to the following zero density estimate.
	\begin{theorem}[Zero-density estimate]\label{zero density estimate}
		For any divisor $q_1 \mid q$ and $1/2<\sigma<1$, we have
		\begin{align*}
			\sum_{\chi \bmod q} N(\sigma,T,\chi)
			\leq (qT)^{o(1)}
			(&
			(q_1^{\frac{1}{3}}qT)^{\frac{3(1-\sigma)}{1+\sigma}}
			+
			(qT(q_1T)^{-\frac{1}{2}})^{\frac{3(1-\sigma)}{\sigma}}+
			(q_1T)^{-\frac{1}{2}}(qT)^{\frac{21-20\sigma}{6}}
			+
			(qT)^{\frac{15(1-\sigma)}{3+5\sigma}}
			).
		\end{align*}
		If $q_1 \mid q$ and $q_1 \geq \sqrt{q}$, we also have
		\begin{align*}
			\sum_{\chi \bmod{q}} N(\sigma,T,\chi)  \leq (qT)^{o(1)}(&(q_1^{\frac{1}{3}}q^2T^2)^{1-\sigma}+(q^3T^{\frac{9}{4}}q_1^{-\frac{3}{4}})^{1-\sigma}+ (qT)^{B(1-\sigma)} + (qT)^{\frac{30}{13}(1-\sigma)}),
		\end{align*}
		where $B=\frac{37+3\beta-\sqrt{9\beta^2+222\beta-71}}{12}$ and $\beta = \frac{\log{q_1T}}{\log{qT}}$.
		\\
		In all cases, we have 
		\begin{align*}
			\sum_{\chi \bmod q} N(\sigma,T, \chi) \leq (qT)^{o(1)}(q^{\frac{7}{3}(1-\sigma)}T^{2(1-\sigma)}+(qT)^{\frac{30}{13}(1-\sigma)}).
		\end{align*}
		If $q$ is $T$-smooth, then 
		\begin{align*}
			\sum_{\chi \bmod q} N(\sigma,T,\chi) \leq (qT)^{\frac{30}{13}(1-\sigma)+o(1)}.
		\end{align*}
	\end{theorem}
	The best exponent $A$ in $\sum_{\chi \bmod q} N(\sigma,T,\chi) \leq (qT)^{A(1-\sigma)+o(1)}$ that we get from Theorem \ref{zero density estimate} uniformly in $q$ and $T$ is $A=7/3$. The exponent $7/3$ represents an improvement over previous results: Huxley \cite{Huxley1975} obtained $A=12/5$, Jutila \cite{Jutila1972} obtained $A=2.460\dots$, Forti and Viola \cite{F-V1973} obtained $A=2.463\dots$ and Montgomery \cite[Thm. 12.1]{MontgomeryBook} obtained $A=2.5$. This enhanced exponent leads to important consequences in prime number theory\footnote{Igor Shparlinski informed us that the current best value
		$L=2.1115$ of the Linnik constant for primes in arithmetic progressions
		modulo powers of a fixed prime \cite{BS2019} may be further improved using our new exponent $7/3$. We do not pursue this application here.}. One such application, building on the work of Jutila \cite{Jutila1972}, is the following corollary.
	
	\begin{corollary}\label{lestGbn}
		Let $p$ be an odd prime, $(p, k)=1$ and let $G(p, k)$ be the least Goldbach's number (a number of the form $p_{1}+p_{2}$ with $p_{1}$, $p_{2}$ primes) that is congruent to $k$ modulo $p$. Then
		\[G(p, k) \ll_{\epsilon} p^{7/6+\epsilon}. \]
	\end{corollary}
	Zero density estimates for the Riemann zeta function have important consequences
	for the distribution of primes. Similarly, zero density estimates for Dirichlet $L$-functions have important consequences for the distribution of primes in arithmetic
	progressions, particularly in short intervals. In contrast to the Riemann zeta function, a key difficulty for Dirichlet $L$-functions is the lack of a sufficiently strong zero-free region in the $q$-aspect. In particular, there is the possibility of an exceptional real and simple zero $\rho$ of $L(s,\chi)$ such that $\beta \geq 1-\frac{c}{\log{3q}}$ for an absolute constant $c$ and a non-trivial real primitive Dirichlet character modulo $q$ \cite[Theorem 5.26]{iwanieckowalski}. 
	
	Even if such exceptional zeroes do not exist, the Vinogradov-Korobov zero-free region \cite[p. 176]{Montgomery1994}, which states that $L(s,\chi) \neq 0$ if
	\begin{align*}
		\sigma > 1 -\frac{c}{\log{q}+(\log (|t|+2))^{\frac{2}{3}}(\log\log (|t|+2))^{\frac{1}{3}}},
	\end{align*}
	is not strong enough in the $q$-aspect. This restricts our ability to prove asymptotics for primes in arithmetic progressions for short intervals. However, improvements to the zero-free region are available for specific moduli $q$ such as powerful moduli, as shown by Postnikov \cite{Postnikov1956}. If the zero-free region can be made wide enough for the relevant heights, we have the following asymptotic for primes in arithmetic progressions:
	\begin{corollary}[Primes in arithmetic progressions under a good zero-free region]
		Let \(h\leq x\),
		\(q_1\mid q\) with \(q_1\geq \sqrt q\) and \((a,q)=1\). Suppose that
		$\prod_{\chi \bmod q} L(s,\chi)\neq 0$ for $\sigma\geq 1-\eta(q,T)$ when $|\text{Im}(s)| \le T$ and that $\eta(q,x^{\epsilon/3}T_0) \log x \to \infty$ as $x \to \infty$ uniformly for $q < h/x^{17/30-\epsilon}$, where $T_0=(x/h)^{1+o(1)}$. Then
		\begin{align}
			\psi(x+h;q,a)-\psi(x;q,a)
			=
			(1+o_{\eta,\epsilon}(1))\frac{h}{\phi(q)}, \label{PNT in APs in short intervals}
		\end{align}
		provided that
		\[
		\frac{h}{\phi(q)}
		\geq
		x^{\frac12+\epsilon}q_1^{\frac16}
		+
		x^{\frac59+\epsilon}q^{\frac13}q_1^{-\frac13}
		+
		x^{
			\frac23
			-
			\frac{1}{30}
			\sqrt{
				10+45\frac{\log(q_1/q)}{\log x}
			}
			+\epsilon
		}
		+
		x^{\frac{17}{30}+\epsilon}.
		\]
		If \(q\) is
		\( x^{13/30-\epsilon/2}q^{-1}\)-smooth, then we have \eqref{PNT in APs in short intervals} provided that
		\[
		\frac{h}{\phi(q)}\geq x^{\frac{17}{30}+\epsilon}.
		\]
		For all \(h\) and \(q\), we have \eqref{PNT in APs in short intervals} provided that
		\[
		\frac{h}{\phi(q)}
		\geq
		x^{\frac12+\epsilon}q^{\frac16}
		+
		x^{\frac{17}{30}+\epsilon}.
		\]
	\end{corollary}
	This can be deduced from Theorem \ref{zero density estimate} via a smoothed version of the explicit formula. Combining this with the zero-free region for powerful moduli by Iwaniec \cite{Iwaniec1974}, we have the following corollary:
	\begin{corollary}[Primes in arithmetic progressions for prime-power moduli]
		Let \(h\leq x\) and suppose that $q=q_1^n$
		for some prime \(q_1\) such that $q_1\leq (\log x)^A$ for some absolute constant \(A\). Let \((a,q)=1\). Then
		\[
		\psi(x+h;q,a)-\psi(x;q,a)
		=
		(1+o_{\epsilon,A}(1))\frac{h}{\phi(q)}
		\]
		provided that
		\[
		\frac{h}{\phi(q)}
		\geq x^{\frac{17}{30}+\epsilon}.
		\]
	\end{corollary}
	In this case, exceptional zeroes are ruled out except for characters of conductor $q_1$ as the modulus of real primitive characters is squarefree up to a factor of 4.

	\subsection*{Acknowledgements} BC is grateful to Igor Shparlinski for his comments. VG would like to thank James Maynard for many helpful discussions, comments and suggestions. YCL would like to thank Maksym Radziwi\l\l \ for many insightful comments and suggestions. BC is supported by Ghent University through a postdoctoral fellowship (grant number BOF24/PDO/020). VG is supported by the Martingale Scholarship.
	
	\subsection{Notation and Conventions}\label{not}
	We write $A \ll B$ or $A=O(B)$ to mean that $A \leq CB$ for some absolute constant $C$. If the constant $C$ depends on a parameter $z$, we write
	$A \ll_{z} B$ or $A=O_{z}(B)$. We say $A \asymp B$ if both $A \ll B$ and $B \ll A$ hold, meaning $A$ and $B$ are comparable up to absolute constants. The notation $A \sim B$ signifies the stronger condition $B<A \leq2B$. Similarly, we write $A \lessapprox B$ to indicate that for any $\epsilon > 0$, there exists a constant $C(\epsilon) > 0$, depending only on $\epsilon$, such that $A \leq C(\epsilon)(qT)^{\epsilon} B$ for all sufficiently large $qT$. If the constant $C$ also depends on a parameter $z$, we write $A \lessapprox_{z}B$. Asymptotic quantities such as $o(1)$ are interpreted as $qT \to \infty$.
	
	Throughout this paper, we denote by $\phi$ and $\mu$ the Euler totient function and the Möbius function respectively, write $\#\mathcal{A}$ for the cardinality of a finite set $\mathcal{A}$, use $(a,b)$ for the greatest common divisor of integers $a$ and $b$ and let $\mathbf{1}_{\mathcal{S}}$ represent the indicator function of a set $\mathcal{S}$. We write $e(x):=e^{2\pi i x}$ for $x \in \mathbb{R}$ to denote the complex exponential and write $e_d(n)=e(n/d)$ for $n \in \mathbb{Z}/d\Z$. We write $C_q(m)=\sum_{\substack{a \bmod q\\ (a,q)=1}}
	e_q(am)$ for Ramanujan's sum. We write $\chi_0$ for the principal Dirichlet character, with the modulus being clear from the context. The Fourier transform of an integrable function $f$ is defined as $\hat{f}(\xi):=\mathcal{F}(f)(\xi)= \int_{\mathbb{R}} f(x)e(-\xi x)\:\mathrm{d}x$.  
	
	\section{Sketch outline}
	
	Our argument follows that of Guth-Maynard \cite{GM}, with some differences. Here we outline the differences and new ideas that are involved. For the sake of this sketch, we consider the case where $\sigma = 3/4$, $N=(qT)^{1-\delta}$ for $0< \delta \leq 1/3$. We aim to beat the bound $|W| \lessapprox  qTN^{-\frac{1}{2}}$ coming from \eqref{qTmvt}. We consider a set $W$ consisting of pairs $(t,\chi)$ of real numbers $t \in [0,T]$ and primitive Dirichlet characters modulo $q$ such that if $(t_1,\chi_1) \neq (t_2,\chi_2)$ then if $\chi_1 \neq \chi_2$, we have that $|t_1-t_2| \geq 1$. We suppose that $|D_N(t,\chi)| \geq N^{\frac{3}{4}}$ for each $(t,\chi) \in W$, where 
	\begin{align*}
		D_N(t,\chi) = \sum_{n \sim N} a_n \chi(n) n^{it}
	\end{align*}
	and $|a_n| \leq 1$. We suppress many technical details, such as the presence of smoothing in most sums.
	
	Let $M$ be the $|W| \times N$ matrix with entries 
	\begin{align*}
		M_{(t, \chi),n} = \chi(n)n^{it}.
	\end{align*}
	Set $\textbf{a}=(a_n)_{n\sim N}$, then $D_N(t,\chi) = (M\textbf{a})_{(t, \chi)}$. Since $|D_N(t,\chi)| \geq N^{3/4}$ for each $(t,\chi) \in W$, we have that
	\begin{align*}
		|W| N^{3/2} \leq \sum_{(t,\chi) \in W} |D_N(t,\chi)|^2 =||M\textbf{a}||_2^2 \leq ||M||^2 ||\textbf{a}||_2^2  \leq N||M||^2,
	\end{align*}
	as $|a_n| \leq 1$. Therefore to beat the bound $|W| \lessapprox qTN^{-\frac{1}{2}}$, we aim to beat $||M|| \lessapprox (qT)^{\frac{1}{2}}$. We recall that $\|M\|=s_1(M)$, where $s_1(M)$ denotes the largest singular value of $M$. Equivalently, $s_1(M)$ is the square root of the largest eigenvalue of $MM^*$. A possible starting point is the simple bound
	\[
	s_1(M)\leq \operatorname{tr}\bigl((MM^*)^r\bigr)^{1/2r}.
	\]
	We make two observations. First, the case $r=1$ only gives a trivial
	estimate while the case $r=2$ is closely related to the Halász--Montgomery
	method. Second, this bound can be quite lossy, especially when the eigenvalues
	of $MM^*$ are all of comparable size. We therefore take $r=3$, but use instead
	the refined bound
	\[
	s_1(M)
	\ll
	\left(
	\operatorname{tr}\bigl((MM^*)^3\bigr)
	-
	\frac{\operatorname{tr}(MM^*)^3}{|W|^2}
	\right)^{1/6}
	+
	\left(
	\frac{\operatorname{tr}(MM^*)}{|W|}
	\right)^{1/2}.
	\]
	This inequality leads to a bound of the following shape:
	\begin{align*}
		s_1(M)^6 \lessapprox \bigg|\sum_{\substack{(t_i,\chi_i) \in W \\ 1 \leq i \leq 3 , \text{ not all equal}}} \sum_{\substack{n_1,n_2,n_3 \sim N}} n_1^{i(t_1-t_2)}n_2^{i(t_2-t_3)}n_3^{i(t_3-t_1)} {\chi_1 \overline{\chi_2}}(n_1) 
		{\chi_2 \overline{\chi_3}}(n_2) 
		{\chi_3 \overline{\chi_1}}(n_3)\bigg|.
	\end{align*}
	To beat $|W| \lessapprox qTN^{-1/{2}}$, we aim to show that the right hand side is smaller than $(qT)^3$. We apply Poisson summation on the $n_1,n_2,n_3$ sums without simplifying the Fourier integrals. This yields a bound of roughly
	\begin{align*}      \bigg({\frac{N}{q}}\bigg)^3\sum_{\substack{(t_i,\chi_i) \in W \\ 1 \leq i \leq 3 \\  \text{not all equal}}}\sum_{|m_1|,|m_2|,|m_3| \sim \frac{qT}{N}}&\int_{[1,2]^3} e\bigg(-\frac{N}{q} \overrightarrow{m}\cdot \overrightarrow{u}\bigg)  u_1^{i(t_1-t_2)}u_2^{i(t_2-t_3)}u_3^{i(t_3-t_1)} \mathrm{d}\overrightarrow{u} \\ &\sum_{\substack{c_1,c_2,c_3 \bmod{q} \\ (c_1c_2c_3,q)=1}} e_q(\overrightarrow{c} \cdot \overrightarrow{m}) \chi_1(c_1c_3^{-1}) \chi_2 (c_2c_1^{-1})\chi_3(c_3c_2^{-1}) .
	\end{align*}
	This can be written as 
	\begin{align}\label{R sum}
		\bigg({\frac{N}{q}}\bigg)^3\sum_{|\overrightarrow{m}| \sim \frac{qT}{N}}\int_{[1,2]^3}\sum_{\substack{c_1,c_2,c_3 \bmod{q} \\ (c_1c_2c_3,q)=1}} &R\bigg(\frac{u_1}{u_3},c_1c_3^{-1}\bigg) R\bigg(\frac{u_2}{u_1},c_2c_1^{-1}\bigg)R\bigg(\frac{u_3}{u_2} ,c_3c_2^{-1}\bigg) e\bigg(-\frac{N}{q} \overrightarrow{m}\cdot \overrightarrow{u}\bigg)\mathrm{d}\overrightarrow{u} \notag  \\
		&e_q(\overrightarrow{c} \cdot \overrightarrow{m}),  
	\end{align}
	where for $v \in \mathbb{R}_{>0}$ and $b \in \mathbb{Z}/q\mathbb{Z}$,
	\begin{align*}
		R(v,b) = \sum_{(t,\chi) \in W} v^{it} \chi(b).
	\end{align*}
	In the expression \eqref{R sum}, both of the arguments of the $R$ functions lie in a 2 dimensional subvariety of the form $z_1z_2z_3=1$ so a change of variables in both arguments gives us the expression 
	\newpage
	\begin{align*}
		\bigg(\frac{N}{q}\bigg)^3 &\sum_{|\overrightarrow{m}| \sim \frac{qT}{N}}\int_{[1/2,2]^2} \sum_{\substack{b_1,b_2 \bmod{q} \\ (b_1b_2,q)=1}} R(v_1,b_1)R\bigg(\frac{v_2}{v_1},b_2b_1^{-1}\bigg)R\bigg(\frac{1}{v_2},b_2^{-1}\bigg) \: dv_1 dv_2 \\
		&\bigg(\int_{[1,2]} e\bigg(-\frac{Nu_3(m_1v_1+m_2v_2+m_3)}{q} \bigg) \:\mathrm{d}u_3 \bigg) \cdot \bigg( \sum_{\substack{c_3 \bmod{q} \\ (c_3,q)=1}}e_q((m_1b_1+m_2b_2+m_3)c_3) \bigg).
	\end{align*}
	This is the key step in the argument. In the Fourier integral over $u_3$, there is a huge amount of cancellation unless $|m_1v_1+m_2v_2+m_3| \approx 0$. However, in the (Ramanujan) sum over $c_3$, we get varying amounts of cancellation as its typical size is roughly $(m_1b_1+m_2b_2+m_3,q)$, which is an upper bound. We take this upper bound and now aim to bound an expression of the form 
	\begin{align}\label{main Rexpression}
		\frac{N^3}{q^3T} \sum_{|\overrightarrow{m}| \sim \frac{qT}{N}} \int\limits_{\substack{v_1 \in [1/2,2]}} \sum_{\substack{b_1,b_2 \bmod q \\ (b_1b_2,q)=1}} &\bigg|R(v_1,b_1)R\bigg(\frac{m_1v_1+m_3}{m_2v_1},b_2b_1^{-1}\bigg)R\bigg(\frac{m_1v_1+m_3}{m_2},b_2\bigg)\bigg|  \: dv_1 \notag \\
		&(m_1b_1-m_2b_2+m_3,q)
	\end{align}
	In \eqref{main Rexpression}, we make use of the averaging over $\overrightarrow{m}$ by establishing a sharp upper bound for sums over affine transformations of functions with a GCD twist. This gives the bound
	\begin{align}\label{singular value sketch equation}
		s_1(M)^6 \lessapprox (qT)^2|W|^{3/2}+qTN|W|^{1/2}E(W)^{1/2},
	\end{align}
	where 
	\begin{align*}
		E(W) = \#\{(t_1,\chi_1),(t_2,\chi_2),(t_3,\chi_3),(t_4,\chi_4) \in W: |t_1+t_2-t_3-t_4| \leq 1, \chi_1\chi_2=\chi_3\chi_4\}
	\end{align*}
	is the additive energy of $W$.
	
	Next, we prove a bound for $E(W)$ using the following estimate due to Heath-Brown for $1$-bounded $b_n$:
	\begin{align*}
		\sum_{(t_1,\chi_1),(t_2,\chi_2) \in W} \bigg| \sum_{n \sim N} b_n n^{i(t_1-t_2)}\chi_1 \overline{\chi_2}(n) \bigg|^2 \lessapprox|W|^2N+|W|N^2 + (qT)^{1/2}|W|^{5/4}N. 
	\end{align*}
	After inserting the resulting bound for $E(W)$ into \eqref{singular value sketch equation}, we finally obtain
	\begin{align}\label{final sketch bound}
		|W| \lessapprox (qT)^{1/2}+(qT)^{4/3}N^{-1}. 
	\end{align}
	This improves on the target bound $|W|\lessapprox qTN^{-1/2}$ when
	$N=(qT)^{1-\delta}$ for $0< \delta \leq 1/3$, covering the critical case when $N=(qT)^{4/5}$.
	
	An analogous estimate is also produced by the Guth-Maynard argument,
	\begin{align}
		|W| \lessapprox T^{1/2}+T^{4/3}N^{-1}, \label{GM subdiv bound}
	\end{align}
	and here a major difference arises. When $T \geq N^{6/5}$, the dominating term is $T^{4/3}N^{-1}$. To refine the estimate, we can perform subdivision. That is, we let $T_0$ be the value of $T$ that balances the two terms $T^{1/2}$ and $T^{4/3}N^{-1}$ and then cut up our interval $[0,T]$ into $T/T_0$ intervals and apply the bound \eqref{GM subdiv bound} to each piece. This gives the bound
	\begin{align*}
		|W| \lessapprox T^{1/2}+ TN^{-3/5}
	\end{align*}
	which yields a further improvement to the resulting zero density estimate. Subdivision in the $q$-aspect involves splitting $q$ into its factors. For general moduli $q$, this is not available. However, when the factors of $q$ are of an ideal size, we can do some subdivision and get some further improvements to the bound \eqref{final sketch bound}.
	
	This is analogous to the Halász-Montgomery-Huxley large values estimate \cite{Huxley1972} where initially this was proved using subdivision and so, only directly applicable to the Riemann $\zeta$ function. Later, Huxley \cite{Huxley1975} developed a technique based on his reflection method that allowed for a $q$-analogue with the same strength bound to be proved, even when subdivision was not available. It appears to be difficult to find an analogous argument to Huxley's in this setting to improve the estimate when there is no subdivision available. However, as above the Guth-Maynard argument can be run without performing any subdivision at all and we still get a non-trivial improvement to the critical case $\sigma = 3/4$, giving us a zero density estimate of $\sum_{\chi \bmod q}N(\sigma,T,\chi) \leq (qT)^{\frac{7}{3}(1-\sigma)+o(1)}$ in all cases.
	
	\section{Reduction to main proposition}\label{remp}
	Throughout the paper we fix a smooth function $w: \mathbb{R} \rightarrow \mathbb{R}_{\geq 0}$, supported on $[1, 2]$, such that $\Vert w^{(j)} \Vert_{\infty}=O_{j}(1)$ for all $j \in \mathbb{N}$ and $w(t)=1$ for $t \in [6/5, 9/5]$. We reduce our main theorem to the following similar theorem which is technically more convenient as we have smoothing in the $n$ variable.
	\begin{proposition}\label{Auxiliary theorem}
		Let $\sigma \in [7/10, 4/5]$ and 
		\[S_{N}(t, \chi)=\sum_{n \in \mathbb{Z}} w\bigg(\frac{n}{N}\bigg)b_{n}\chi(n)n^{it},\]
		where $\chi$ is a primitive character modulo $q$ and $|b_{n}| \leq1$. Suppose that $W$ is a finite set of pairs $(t, \chi)$,  where $|t| \leq T$ and for $(t, \chi)\neq (t', \chi')$ either $\chi \neq \chi'$ or $|t-t'|\geq (qT)^{\epsilon}$. If
		\[ |S_{N}(t, \chi)| \geq N^{\sigma}/6
		\]
		for all $(t, \chi) \in W$ and $N \geq (qT)^{\frac{2}{3}}/2$,
		we have
		\[ |W| \lessapprox_{\epsilon} N^{2-2\sigma}+ (qT)^{\frac{1}{2}}N^{3-4\sigma}+(qT)^{\frac{4}{3}}N^{2-4\sigma}.
		\]
	\end{proposition} 
	
	We will use this proposition to prove Theorem \ref{Partial LVE}, whose statement we recall:
	\begin{theoremrestated}
		Let
		\[
		D_{N}(t, \chi)=\sum_{N < n \leq 2N} a_{n}\chi(n)n^{it},
		\]
		where $\chi$ is a primitive character modulo $q$ and $|a_{n}| \leq1$.
		Suppose that $W$ is a finite set of pairs $(t, \chi)$, where $|t| \leq T$
		and for $(t, \chi)\neq (t', \chi')$ either $\chi \neq \chi'$ or
		$|t-t'|\geq 1$. Assume that
		\[
		|D_{N}(t, \chi)| \geq V
		\]
		for all $(t, \chi) \in W$ and $N \geq(qT)^{\frac{2}{3}}$.
		Then for any divisor $q_1 \mid q$, we have
		\[
		|W| \leq (qT)^{o(1)}
		(
		N^{2}V^{-2}
		+ qq_1^{-\frac{1}{2}}T^{\frac{1}{2}}N^{3}V^{-4}
		+ q q_1^{\frac{1}{3}}TN^{2}V^{-4}
		+ qTN^{\frac{12}{5}}V^{-4}
		).
		\]
	\end{theoremrestated}
	
	\begin{remark}
		If the dominating term is $N^{2-2\sigma}$ then we have the optimal bound as conjectured by Montgomery (see \cite[Equation (4)]{HBLV}). If this is not the dominant term then the dominant term is decided by the size of $q_1$. If $q_1 > N^{\frac{6}{5}}$ then the dominating term is $qq_1^{\frac{1}{3}}TN^{2}V^{-4}$. If $\frac{N^{6/5}}{T}< q_1 < N^{\frac{6}{5}}$ then the dominating term is $qTN^{\frac{12}{5}}V^{-4}$. Lastly, if $q_1 < \frac{N^{6/5}}{T},$ then the term $qq_1^{-\frac{1}{2}}T^{\frac{1}{2}}N^{3}V^{-4}$ dominates.
	\end{remark}
	
	\vspace{1 pt} 
	
	\begin{proof}[Proof of Theorem \ref{Partial LVE} assuming Proposition \ref{Auxiliary theorem}:] The result follows from Ingham \eqref{qTmvt} and Huxley's estimates \eqref{hmhlve} in the cases that $V < N^{7/10}$ or if $V > N^{4/5}$, so we may focus our analysis on the intermediate regime $V=N^{\sigma}, \sigma \in [7/10, 4/5]$. This is because $qTN^{1-2\sigma} \leq qTN^{\frac{12-20\sigma}{5}}$ for $\sigma \leq 0.7$ and $qTN^{4-6\sigma} \leq qTN^{\frac{12-20\sigma}{5}}$ for $\sigma \geq 0.8$. By splitting $D_{N}$ into three parts and applying the triangle inequality, we get
		\begin{align*} |D_{N}(t, \chi)| &\leq \bigg|\sum_{N< n < \frac{6}{5}N} a_{n}\chi(n)n^{it}\bigg| + \bigg|\sum_{\frac{6}{5}N \leq n \leq \frac{9}{5}N} a_{n}\chi(n)n^{it}\bigg| + \bigg|\sum_{\frac{9}{5}N< n \leq 2N} a_{n}\chi(n)n^{it}\bigg|
			\\
			&=: |D^{(1)}_{N}(t, \chi)|+|D^{(2)}_{N}(t, \chi)|+|D^{(3)}_{N}(t, \chi)|.
		\end{align*}
		It is clear that we can write 
		\[ W=W_{1}\cup W_{2} \cup W_{3},
		\]
		where 
		\[W_{i}=\{ (t, \chi) \in W, |D^{(i)}_{N}(t, \chi)| \geq N^{\sigma}/3\}, \ \ 1\leq i \leq 3.\]
		It suffices to show that whenever $\sigma \in [0.7, 0.8]$, $(b_n)$ is a 1-bounded complex sequence and $W$ is a 1-separated set of pairs $(t,\chi)$ such that
		\begin{align*}
			\bigg| \sum_n w\bigg(\frac{n}{N}\bigg) b_n \chi(n) n^{it} \bigg| \geq N^{\sigma}/6
		\end{align*}
		for each $(t,\chi) \in W$ and $N \geq(qT)^{2/3}/2$, we have 
		\begin{align*}
			|W| \leq (qT)^{o(1)} (N^{2-2\sigma}+ qq_1^{-\frac{1}{2}}T^{\frac{1}{2}}N^{3-4\sigma}+ q q_1^{\frac{1}{3}}TN^{2-4\sigma} + qTN^{\frac{12-20\sigma}{5}}).
		\end{align*}
		This is because we can apply this estimate with $b_{n}=a_{n}$ for $6N/5 \leq n \leq 9N/5$ and $b_{n}=0$ for $n \notin [6N/5, 9N/5]$ and noting that $D^{(2)}_{N}(t, \chi)=S_{N}(t, \chi)$, we find that $|W_{2}|$ satisfies the desired estimation. Similarly, by replacing $N$ by $11N/15$ (respectively, $19N/15$) in the proposition and appropriately choosing $b_{n}$, we obtain the same upper bound for $|W_{1}|$ (respectively, $|W_{3}|$). 
		
		We now fix $\epsilon>0$ and choose a subset $W'\subset W$ such that
		$|W'|\geq |W|/(qT)^\epsilon$ and for any two distinct pairs
		$(t,\chi),(t',\chi')\in W'$, either $\chi\neq\chi'$ or $|t-t'|\geq(qT)^\epsilon$. This can be done for each character appearing in a pair in $W$, we choose the smallest $t$ in the set $\{t: (t,\chi) \in W\}$ and let $(t,\chi) \in W'$ and then repeatedly do this for the next smallest element that is at least $(qT)^{\epsilon}$ away from all the chosen elements. 
		
		If $qT \leq N$ then by \eqref{qTmvt}, we have 
		\begin{align*}
			|W| \leq (qT)^{\epsilon}|W'| \leq (qT)^{o(1)+\epsilon}N^{2-2\sigma}
		\end{align*}
		and letting $\epsilon \to 0$ sufficiently slowly gives the result in this case. If $N \leq qT \leq N^{6/5}$, then we apply Proposition \ref{Auxiliary theorem} to bound $|W'|$, which gives
		\begin{align*}
			|W| \lessapprox_{\epsilon} N^{2-2\sigma}+(qT)^{\frac{1}{2}} N^{3-4 \sigma} \lessapprox N^{2-2\sigma}+qTN^{\frac{12-20\sigma}{5}}. 
		\end{align*}
		
		We now suppose that $N^{\frac{6}{5}}\leq qT \leq N^{\frac{3}{2}}$. Let $q_1 \mid q$ and $1 \leq T_0 \leq T$. We perform subdivision in both the $q$ and $T$-aspects. It is proved in Lemma \ref{character subdivision} that we may write $\chi=\chi_1 f$ for some Dirichlet character $\chi_1$ modulo $q_1$, while $f$ is a 1-bounded function belonging to a set $S$ depending only on $q$ and $q_1$, with $\#S \leq \phi(q)/\phi(q_1)$.
		
		This lets us split $W'$ into at most $\phi(q)/\phi(q_1)$ sets $W'_{f}$ for some choice of $f \in S$ such that if $(t,\psi) \in W'_{f}$ then $(t,f\psi) \in W'$. Thus
		\begin{align*}
			|W'| = \sum_{f} |W'_f| \leq \frac{\phi(q)}{\phi(q_1)} |W'_{g}|,
		\end{align*}
		for some particular $g \in S$. We can also perform subdivision in $T$. We can split up $W_{g}'$ into sets $W_j$ so that all of the real numbers $t$ in $(t,\chi) \in W_j$ (for $j=0,\dots,\lfloor{T/T_0}\rfloor$) belong to some interval of length $T_0$ such that $1 \leq T_0 \leq T$, so that
		\begin{align*}
			|W_{g}'| \leq \sum_j |W_j| \ll \frac{T}{T_0} |W_k|,
		\end{align*}
		for some particular choice of $0\leq k \leq \lfloor{T/T_0\rfloor}$. Now, $W_{k}$ is still a $(qT)^{\epsilon}$-separated set because if $(t_1,\psi_1), (t_2,\psi_2) \in W_{k}$ are distinct then if $\psi_1=\psi_2$, we have $g \psi_1 = g \psi_2$ and $(t_1,g \psi_1), (t_2, g \psi_2) \in W'$ so we have that $|t_1 -t_2| \geq (qT)^{\epsilon}$.
		
		We now have a Dirichlet polynomial $S_N'$ with 1-bounded coefficients $b_n'$ for which $|S_N'(t,\chi)|\geq N^{\sigma}/6$ for $(t,\chi) \in W_k$, where $t$ is a real number belonging to an interval of length $T_0\geq1$ and $\chi$ is a Dirichlet character modulo $q_1$, which is not necessarily primitive. An imprimitive character $\chi$ is of some conductor $q_1^* \mid q_1$ and $q_1^* \mid q$ too. We can partition our set $W_k = \sqcup_{d \mid q_1} S_d$ where $S_d$ consists of pairs $(t,\chi)$ for which $\chi$ is induced by a primitive character modulo $d$. Since $(dT_0)^{2/3} \leq (qT)^{2/3}$, we may apply Proposition \ref{Auxiliary theorem} to each $S_d$ and sum over each $d \mid q_1$, giving us
		\begin{align}
			|W| \lessapprox \frac{(qT)^{1+\epsilon}}{q_1T_0}(N^{2-2\sigma}+(q_1T_0)^{\frac{1}{2}}N^{3-4\sigma} +   (q_1T_0)^{\frac{4}{3}}N^{2-4\sigma}). \label{LVE after subdiv}
		\end{align}
		The dominant term is determined by the size of $q_1$, the three cases being $q_1 > N^{\frac{6}{5}}$, $\frac{N^{\frac{6}{5}}}{T} < q_1 < N^{\frac{6}{5}}$ and $q_1 < \frac{N^{\frac{6}{5}}}{T}$.
		
		\textbf{Case 1: $q_1 > N^{\frac{6}{5}}$.} 
		
		In this case, $(q_1T_0)^{\frac{4}{3}}N^{2-4\sigma}>(q_1T_0)^{\frac{1}{2}}N^{3-4\sigma}$. Also for any $T_0 \geq 1$, $(q_1T_0)^{\frac{4}{3}}N^{2-4\sigma} > N^{\frac{18-20\sigma}{5}} \geq N^{2-2\sigma}$ for $\sigma \in [0.7,0.8
		]$. So $(q_1T_0)^{\frac{4}{3}}N^{2-4\sigma}$ is the dominant term for any choice of $T_0\geq 1$ so we choose $T_0=1$ and get
		\begin{align*}
			|W| \lessapprox (qT)^{\epsilon} qq_1^{\frac{1}{3}}T N^{2-4\sigma}.
		\end{align*}
		
		\textbf{Case 2: $\frac{N^{\frac{6}{5}}}{T}<q_1 < N^{\frac{6}{5}}$.} 
		
		In this case, $T_0 = \frac{N^{\frac{6}{5}}}{q_1}$ can be chosen to make the last two terms in \eqref{LVE after subdiv} balance, giving us $qTN^{\frac{12-20\sigma}{5}}$.  Note that $qTN^{\frac{12-20\sigma}{5}} > N^{\frac{18-20\sigma}{5}} \geq N^{2-2\sigma}$ for $\sigma \in [0.7,0.8]$. This yields the bound
		\begin{align*}
			|W| \lessapprox (qT)^{1+\epsilon}N^{\frac{12-20\sigma}{5}}.
		\end{align*}
		
		\textbf{Case 3: $q_1 < \frac{N^{\frac{6}{5}}}{T}$.} 
		
		We choose $T_0=T$. Since $q_1T < N^{\frac{6}{5}}$, we have that $(q_1T)^{\frac{4}{3}}N^{2-4\sigma} < (
		q_1T)^{\frac{1}{2}}N^{3-4\sigma}$. Then we have the bound
		\begin{align*}
			|W| \lessapprox (qT)^{\epsilon}\bigg(\frac{q}{q_1}N^{2-2\sigma}+ q q_1^{-\frac{1}{2}}T^{\frac{1}{2}}N^{3-4\sigma}\bigg).
		\end{align*}
		Note that
		\begin{align*}
			\min\bigg({\frac{q}{q_1}N^{2-2\sigma},qTN^{4-6\sigma}}\bigg) \ll qq_1^{-\frac{1}{2}}T^{\frac{1}{2}}N^{3-4\sigma}.
		\end{align*}
		Therefore by combining with \eqref{hmhlve}, our estimate for $|W|$ in this case is
		\begin{align*}
			|W| \lessapprox (qT)^{\epsilon}q q_1^{-\frac{1}{2}}T^{\frac{1}{2}}N^{3-4\sigma}.
		\end{align*}
		Our overall estimate for $|W|$ is, upon sending $\epsilon \to 0$ sufficiently slowly,
		\begin{align*}
			|W| \lessapprox N^{2-2\sigma}+qTN^{\frac{12-20\sigma}{5}} + qq_1^{\frac{1}{3}}T N^{2-4\sigma}+ q q_1^{-\frac{1}{2}}T^{\frac{1}{2}}N^{3-4\sigma}. \tag*{\qedhere} 
		\end{align*}
		
		\begin{lemma}[Character subdivision]\label{character subdivision}
			Let $q_1 \mid q$ and $\chi$ be a character modulo $q$. Then there exists a set $S$ depending only on $q$ and $q_1$ of size at most $\phi(q)/\phi(q_1)$ consisting of 1-bounded functions $f$ such that $\chi=\chi_1 f$ for some character $\chi_1$ modulo $q_1$.
		\end{lemma}
		\begin{proof}
			Let $\chi$ be a character modulo $q$. When $(n,q)>1$, we impose $f(n)=0$. By the decomposition of Dirichlet characters, we may reduce the proof to the case where \(q\) is a prime
			power. Thus, let $q=p^j$ and $q_1=p^k$, with $0\leq k\leq j$.
			
			For an odd prime $p$, let $v$ be a primitive root mod $p^j$. Then $\chi$ is determined entirely by the choice of $0\leq a < p^{j-1}(p-1)$, where $\chi(v) = e\Big(\frac{a}{p^{j-1}(p-1)}\Big)$. Let $a=a_1+p^{j-k}a_2$, where $0 \leq a_1<p^{j-k}$ and $0\leq a_2<p^{k-1}(p-1)$. Then 
			\begin{align*}
				\chi(n)=\chi(v^r) = e\bigg(\frac{a_1r}{p^{j-1}(p-1)}\bigg)e\bigg({\frac{a_2r}{p^{k-1}(p-1)}}\bigg)=f(n)\chi_1(n).
			\end{align*}
			Here $\chi_1$ is a character modulo $q_1=p^k$ that only depends on the choice of $a_2$ and $f$ is a 1-bounded function purely determined by $a_1$, of which there are at most $p^{j-k} = \phi(q)/\phi(q_1)$ choices. For a character modulo $2^j$ we may use the same argument upon writing $\chi(n) = e\left(\frac{av}{2}+\frac{a'v'}{2^{j-2}}\right)$ where $n \equiv (-1)^{v}5^{v'} \bmod 2^j$ where $v,v'$ are generators for the group $(\mathbb{Z}/2^j\mathbb{Z})^{\times}$.
		\end{proof}
		
	\end{proof}

	\section{\texorpdfstring{The matrix $M_{W}$ and its singular values}
		{The matrix MW and its singular values}}
	\label{matrsing}
	
	To prove Proposition \ref{Auxiliary theorem}, we work with the Dirichlet polynomial $S_{N}(t, \chi)$ with smoothed coefficients, defined as
	\[ S_{N}(t, \chi)=\sum_{n \in \mathbb{Z}} w\bigg(\frac{n}{N}\bigg)b_{n}\chi(n)n^{it},
	\]
	where $w$ is a smooth bump supported on $[1, 2]$, as defined in the previous section.
	
	Given a set $W=\{(t, \chi)\}$, let $M=M_{W}$ be the $|W| \times N$ matrix with entries
	\begin{equation}\label{defM}
		M_{(t, \chi), n} = w(n/N)\chi(n)n^{it},
	\end{equation}
	where $(t, \chi) \in W$ and $n \sim N$. The following lemma \cite[Lemma 4.1]{GM}, as a straightforward application of the singular value decomposition theorem, reduces the study of large value estimates for Dirichlet polynomials to the examination of the largest singular value of $M_{W}$.
	
	\begin{lemma}\label{rtls} 
		Let $M$ be the matrix defined in \eqref{defM} and $s_{1}(M)$ its largest singular value. If $|S_{N}(t, \chi)| \geq N^{\sigma}/6$ for all $(t, \chi) \in W$ and $|b_{n}|\leq 1$, then we have
		\[ |W| \ll N^{1-2\sigma}s_{1}(M)^{2}.
		\]
	\end{lemma}
	
	Our next lemma \cite[Lemma 4.2]{GM} establishes an upper bound for the largest singular value in terms of the traces of the matrices $MM^{*}$ and $(MM^{*})^{3}$.
	\begin{lemma}\label{lsvt} 
		We have 
		$$ s_{1}(M) \leq 2\bigg( \emph{tr}((MM^{*})^{3})-\frac{ \emph{tr}(MM^{*})^{3}}{|W|^{2}}\bigg)^{1/6}+2\bigg( \frac{ \emph{tr}(MM^{*})}{|W|}\bigg)^{1/2}.
		$$
	\end{lemma}
	
	We will expand the trace $\mbox{tr}(MM^{*})$ and $\mbox{tr}((MM^{*})^{3})$ and apply Poisson summation to estimate the sums that appear. Later, we will see that these sums involve the function
	\[h_{t}(u):=w(u)^{2}u^{it}.\]
	To this end, we provide basic estimates for its Fourier transform. The proof (cf. \cite[Lemma 4.3]{GM}) relies on the properties of the function $w(u)$ and repeated applications of integration by parts.
	
	\begin{lemma}\label{efh} The following bounds hold for $\hat{h}_t(\xi)$: 
		\begin{enumerate} 
			\item For any integer $j \geq 0$, 
			\[\hat{h}_{t}(\xi) \ll_{j} (1+|t|)^{j}/|\xi|^{j}.\]
			\item For any integer $j \geq 0$, 	\[\hat{h}_{t}(\xi) \ll_{j} (1+|\xi|)^{j}/|t|^{j}.\] \end{enumerate} 
	\end{lemma}
	
	We now estimate $\mbox{tr}(MM^{*})$ and $\mbox{tr}((MM^{*})^{3})$.
	\begin{lemma}(Hilbert-Schmidt Norm estimate)\label{trace_calc}
		We have
		\[\emph{tr}(MM^{*})=
		\frac{|W|N}{q}\sum_{m \in \mathbb {Z}}\sum_{\substack{a \bmod q \\ (a,q)=1}}e_q(am)\hat{h}_{0}\bigg(\frac{Nm}{q}\bigg)
		\]
	\end{lemma}
	\begin{proof}
		Expanding the trace and using $\chi(n)=0$ unless $(n, q)=1$, we have 
		\begin{align*}
			\mbox{tr}(MM^{*})&=\sum_{(t, \chi) \in W} \sum_{(n,q)=1}w\bigg(\frac{n}{N}\bigg)^{2}
			\\
			&=|W|\sum_{(n,q)=1}h_{0}\bigg(\frac{n}{N}\bigg)
		\end{align*}
		We have the result upon performing Poisson summation.
	\end{proof}
	
	Similarly, we expand the cubic trace. 
	\begin{lemma}[Expansion of the cubic trace]\label{cubictrace} Let $W$ be a finite set of pairs $(t, \chi)$,  where $|t| \leq T$ and for $(t, \chi)\neq (t', \chi')$ either $\chi \neq \chi'$ or $|t-t'|\geq (qT)^{\epsilon}$. Then we have
		\begin{align*}
			\emph{tr}((MM^{*})^{3}) = |W|\bigg(\frac{N}{q}\sum_{m \in \mathbb {Z}}\sum_{\substack{a \bmod q \\ (a,q)=1}}e_q(am)\hat{h}_{0}\bigg(\frac{Nm}{q}\bigg)\bigg)^{3}+\sum_{\overrightarrow{m}\in \mathbb{Z}^{3} \backslash \{\mathbf{0}\}}I_{\overrightarrow{m}}+O_{\epsilon}((qT)^{-100}),
		\end{align*}
		where
		\begin{align*}I_{\overrightarrow{m}}=\frac{N^{3}}{q^{3}}\sum_{\substack{(t_1,\chi_1),(t_2,\chi_2),(t_3,\chi_3) \in W \\ \text{not all equal}}}\sum_{\substack{a_1,a_2,a_3 \bmod q \\ (a_1a_2a_3,q)=1 }}&\chi_{1}(a_{1})\overline{\chi}_{2}(a_{1})
			\chi_{2}(a_{2})\overline{\chi}_{3}(a_{2})
			\chi_{3}(a_{3})\overline{\chi}_{1}(a_{3}) e_q(\overrightarrow{a}\cdot\overrightarrow{m})
			\\
			&\hat{h}_{t_{1}-t_{2}}\bigg(\frac{Nm_{1}}{q}\bigg)\hat{h}_{t_{2}-t_{3}}\bigg(\frac{Nm_{2}}{q}\bigg)\hat{h}_{t_{3}-t_{1}}\bigg(\frac{Nm_{3}}{q}\bigg).
		\end{align*}
	\end{lemma}
	\begin{proof}
		By expanding the cubic trace, we see that $\mbox{tr}((MM^*)^3)$ is equal to
		\begin{align*}
			\sum_{\substack{(t_{i}, \chi_{i})\in W \\ 1\leq i \leq 3}}\sum_{n_{1},n_{2},n_{3} \in \mathbb{Z}}
			h_{t_{1}-t_{2}}\bigg(\frac{n_{1}}{N}\bigg)\chi_{1}\overline{\chi}_{2}(n_{1})
			h_{t_{2}-t_{3}}\bigg(\frac{n_{2}}{N}\bigg)\chi_{2}\overline{\chi}_{3}(n_{2})
			h_{t_{3}-t_{1}}\bigg(\frac{n_{3}}{N}\bigg)\chi_{3}\overline{\chi}_{1}(n_{3}).
		\end{align*}
		Performing Poisson summation in the $n_{1}$-sum gives
		\begin{align*}\sum_{n_{1}\in \mathbb{Z}}h_{t_{1}-t_{2}}\bigg(\frac{n_{1}}{N}\bigg)\chi_{1}\overline{\chi}_{2}(n_{1})
			&=\sum_{\substack{a_1 \bmod q}}
			\sum_{\substack{n_{1}\in \mathbb{Z} \\n_{1}\equiv a_{1}  \bmod q}}h_{t_{1}-t_{2}}\bigg(\frac{n_{1}}{N}\bigg)\chi_{1}\overline{\chi}_{2}(n_{1})
			\\
			&=\frac{N}{q}\sum_{\substack{a_1 \bmod q}}
			\sum_{ m_{1}\in \mathbb{Z}} \hat{h}_{t_{1}-t_{2}}\bigg(\frac{Nm_{1}}{q}\bigg)e_q(a_1m_1)\chi_{1}\overline{\chi}_{2}(a_{1}).
		\end{align*}
		We derive analogous expressions for the $n_{2}$-sum and the $n_{3}$-sum, ultimately concluding that 
		\begin{align*}
			\mbox{tr}((MM^{*})^{3})=\sum_{\overrightarrow{m}\in \mathbb{Z}^{3}}\frac{N^{3}}{q^{3}}\sum_{\substack{(t_{i}, \chi_{i})\in W \\ 1\leq i \leq 3}}\sum_{\substack{a_1,a_2,a_3 \bmod q \\ (a_1a_2a_3,q)=1}}&\chi_{1}(a_{1})\overline{\chi}_{2}(a_{1})
			\chi_{2}(a_{2})\overline{\chi}_{3}(a_{2})
			\chi_{3}(a_{3})\overline{\chi}_{1}(a_{3}) e_q(\overrightarrow{a}\cdot\overrightarrow{m})
			\\
			&\hat{h}_{t_{1}-t_{2}}\bigg(\frac{Nm_{1}}{q}\bigg)\hat{h}_{t_{2}-t_{3}}\bigg(\frac{Nm_{2}}{q}\bigg)\hat{h}_{t_{3}-t_{1}}\bigg(\frac{Nm_{3}}{q}\bigg).
		\end{align*}
		Now we take out $Z$, the term coming from $(t_1,\chi_1)=(t_2,\chi_2)=(t_3,\chi_3)$: 
		\begin{align*}
			Z := 
			\sum_{\overrightarrow{m}\in \mathbb{Z}^{3}}\frac{N^{3}}{q^{3}}|W|\sum_{\substack{a_1,a_2,a_3 \bmod q \\ (a_1a_2a_3,q)=1}}\chi_{0}(a_{1})
			\chi_{0}(a_{2})
			\chi_{0}(a_{3}) \hat{h}_{0}\bigg(\frac{Nm_{1}}{q}\bigg)\hat{h}_{0}\bigg(\frac{Nm_{2}}{q}\bigg)\hat{h}_{0}\bigg(\frac{Nm_{3}}{q}\bigg)e_q(\overrightarrow{a}\cdot\overrightarrow{m}).
		\end{align*}
		This may be factored so that
		\begin{align*}
			Z=|W|\bigg(\frac{N}{q}\sum_{m \in \mathbb {Z}}\sum_{\substack{a \bmod q \\ (a,q)=1}}e_q(am)\hat{h}_{0}\bigg(\frac{Nm}{q}\bigg)\bigg)^{3}.
		\end{align*}
		We further separate the terms with $\overrightarrow{m}=\mathbf{0}$, which gives
		\begin{align*}
			\mbox{tr}((MM^{*})^{3})
			= Z &+ Z' +\sum_{\overrightarrow{m}\in \mathbb{Z}^{3} \backslash \{\mathbf{0}\}}I_{\overrightarrow{m}},
		\end{align*}
		where
		\begin{align*}
			Z':= \frac{N^{3}}{q^{3}}\sum_{\substack{(t_1,\chi_1),(t_2,\chi_2),(t_3,\chi_3) \in W \\ \text{not all equal}}}\sum_{\substack{a_1,a_2,a_3 \bmod q \\ (a_1a_2a_3,q)=1}}\chi_{1}\overline{\chi}_{2}(a_{1})
			\chi_{2}\overline{\chi}_{3}(a_{2})
			\chi_{3}\overline{\chi}_{1}(a_{3})\hat{h}_{t_{1}-t_{2}}(0)\hat{h}_{t_{2}-t_{3}}(0)\hat{h}_{t_{3}-t_{1}}(0).
		\end{align*}
		We consider the term $Z'$. Since $\sum_{\substack{a \bmod q}}\chi(a)=0$ unless $\chi=\chi_0$, we have that $\chi_1=\chi_2=\chi_3$. Since $(t_1,\chi_1), (t_2,\chi_2),(t_3,\chi_3)$ are not all equal, without loss of generality, we have $t_1 \neq t_2$. Since $\chi_1=\chi_2$, we thus have $|t_1 -t_2| \geq (qT)^{\epsilon}$. By Lemma \ref{efh} (2), we have that $|\hat{h}_{t_1-t_2}(0)| \ll_{\epsilon} (qT)^{-200}$ so
		\begin{align*}
			\operatorname{tr}\bigl((MM^{*})^{3}\bigr)
			&= Z
			+ \sum_{\overrightarrow{m}\in
				\mathbb{Z}^{3}\setminus\{\mathbf{0}\}}
			I_{\overrightarrow{m}}
			+ O_{\epsilon}\bigl((qT)^{-100}\bigr).
			\qedhere
		\end{align*}
	\end{proof}
	\noindent
	Combining Lemmas \ref{rtls}-\ref{cubictrace}, we establish the following proposition: 
	
	\begin{proposition}\label{W in terms of I}
		Let $W$ be a finite set of pairs $(t, \chi)$,  where $|t| \leq T$ and for $(t, \chi)\neq (t', \chi')$ either $\chi \neq \chi'$ or $|t-t'|\geq (qT)^{\epsilon}$. Assume that $|S_{N}(t, \chi)| \geq N^{\sigma}/6$ for all $(t, \chi)\in W$. Then we have that
		\begin{align*}
			|W| \ll_{\epsilon} N^{2-2 \sigma} + N^{1-2\sigma}\bigg(\sum_{\overrightarrow{m}\in \mathbb{Z}^{3} \backslash \{\mathbf{0}\}}I_{\overrightarrow{m}}\bigg)^{\frac{1}{3}},
		\end{align*}
		where $I_{\overrightarrow{m}}$ is defined as in Lemma \ref{cubictrace}.
	\end{proposition}
	The first term above corresponds to the best possible estimate, $|W| \ll_{\epsilon} N^{2-2\sigma}$, from Montgomery's large value conjecture \cite[Equation (4)]{HBLV}. Therefore, the following sections focus on obtaining a more precise estimate for the sum $I_{\overrightarrow{m}}$. If $N \geq qT$, then by \eqref{qTmvt} we have that $|W| \lessapprox_{\epsilon} N^{2-2\sigma}$ so we suppose that $(qT)^{\frac{2}{3}} \leq N \leq qT$ (as we have already assumed that $(qT)^{\frac{2}{3}} \leq N$).
	
	\section{\texorpdfstring{The components of the sum $S$ and the contribution of $S_{1}$}
		{The components of the sum S and the contribution of S1}}\label{TcomS}
	\noindent
	Recall from Proposition \ref{W in terms of I}, we have 
	\[ |W|^{3}\ll_{\epsilon} N^{6-6\sigma}+N^{3-6\sigma} \sum_{\overrightarrow{m}\in \mathbb{Z}^{3} \backslash \{\mathbf{0}\}}I_{\overrightarrow{m}},
	\]
	where
	\begin{align*}I_{\overrightarrow{m}}=\frac{N^{3}}{q^{3}}\sum_{\substack{(t_1,\chi_1),(t_2,\chi_2),(t_3,\chi_3) \in W \\ \text{not all equal}}}\sum_{\substack{a_1,a_2,a_3 \bmod q \\ (a_1a_2a_3,q)=1}}&\chi_{1}\overline{\chi}_{2}(a_{1})
		\chi_{2}\overline{\chi}_{3}(a_{2})
		\chi_{3}\overline{\chi}_{1}(a_{3}) e_q(\overrightarrow{a}\cdot\overrightarrow{m})
		\\
		& \hat{h}_{t_{1}-t_{2}}\bigg(\frac{Nm_{1}}{q}\bigg)\hat{h}_{t_{2}-t_{3}}\bigg(\frac{Nm_{2}}{q}\bigg)\hat{h}_{t_{3}-t_{1}}\bigg(\frac{Nm_{3}}{q}\bigg).
	\end{align*}
	Things simplify when $m_i=0$ for some $i=1,2,3$. Indeed, since
	$
	\sum_{a \bmod q}\chi(a)=0
	$
	unless $\chi=\chi_0$, one of the arithmetic sums vanishes unless the
	corresponding ratio of two characters is principal. Thus at least two of the
	characters must be equal. Also, by Lemma \ref{efh} (2), if $\chi_{1}=\chi_{2}$ but $t_{1} \neq t_{2}$, then 
	\begin{equation}\label{case1h}
		|\hat{h}_{t_{1}-t_{2}}(0)|\ll_{\epsilon}(qT)^{-100}.
	\end{equation}
	However, if $t_1=t_2$ then we have
	\begin{equation*}
		\hat{h}_{t_{1}-t_{2}}(0)=\hat{h}_{0}(0)=\int w(u)^{2}\:\mathrm{d}u \asymp 1.
	\end{equation*}
	These bounds and the fact that the characters are forced to be equal makes such $I_{\overrightarrow{m}}$ easier to deal with. Taking this into account, we split the sum into three components
	\begin{equation}\label{ss3}
		\sum_{\overrightarrow{m}\in \mathbb{Z}^{3} \backslash \{\mathbf{0}\}}I_{\overrightarrow{m}}=S_{1}+S_{2}+S_{3},
	\end{equation}
	where $S_{1}$ consists of terms with exactly one nonzero $m_{i}$, $S_{2}$ consists of terms with exactly two nonzero $m_{i}$ and $S_{3}$ consists of terms with all three $m_{i}$ being nonzero.
	
	We will close this section by establishing a bound for $S_{1}$. In the next section, we will bound $S_{2}$ using the approximate functional equation and the mean value theorem for Dirichlet polynomials. The main part of the paper (Sections \ref{TCS3}--\ref{S3bandpomp}) focuses on $S_{3}$, which will be most challenging.
	
	\begin{proposition}[$S_1$ bound]\label{S_1 bound}
		We have 
		\begin{align*}
			S_1 \ll_{\epsilon} (qT)^{-10}.
		\end{align*}
	\end{proposition}
	\begin{proof}
		By symmetry,
		\begin{align*}
			S_1=\frac{3N^{3}}{q^{3}} \sum_{\substack{(t_1,\chi_1),(t_2,\chi_2), (t_3,\chi_3) \in W \\ \text{not all equal}}}  &\sum_{m_3\neq 0}
			\hat{h}_{t_1-t_2}(0)\hat{h}_{t_2-t_3}(0)\hat{h}_{t_3-t_1}\bigg(\frac{m_3N}{q}\bigg) 
			\\
			&\sum_{\substack{a_1,a_2,a_3 \bmod q \\ (a_1a_2a_3,q)=1}}\chi_{1}\overline{\chi}_{2}(a_{1})
			\chi_{2}\overline{\chi}_{3}(a_{2})
			\chi_{3}\overline{\chi}_{1}(a_{3})
			e_q(a_3m_3).
		\end{align*}
		Since
		$\sum_{a \bmod q}\chi(a)=0$
		unless $\chi=\chi_0$, the sums over $a_1$ and $a_2$ vanish unless $
		\chi_1\overline{\chi_2}=\chi_0$ and
		$\chi_2\overline{\chi_3}=\chi_0.$
		Thus we may restrict to the case \(\chi_1=\chi_2=\chi_3\). Since the summation is restricted to triples for which the pairs \((t_i,\chi_i)\) are not all equal, it follows that at least one of \(t_1\neq t_2\) or \(t_2\neq t_3\) must hold. Hence, by \eqref{case1h}, either
		$
		|\hat{h}_{t_1-t_2}(0)|\ll_{\epsilon}(qT)^{-100}
		$
		or
		$
		|\hat{h}_{t_2-t_3}(0)|\ll_{\epsilon}(qT)^{-100}.
		$
		Inserting this bound and estimating the remaining sums trivially gives the desired result.
	\end{proof}
	
	\section{\texorpdfstring{The contribution of $S_{2}$}
		{The contribution of S2}}\label{contS2}
	The purpose of this section is to derive the following bound for the sum $S_{2}$.
	
	\begin{proposition}[$S_{2}$ bound]\label{S2B}
		We have
		\[S_{2}\lessapprox_{\epsilon} qT|W|N^{3-2\sigma}.\]
	\end{proposition}
	One of the key ingredients in proving this proposition is the approximate functional equation, also known as the reflection principle for Dirichlet polynomials (cf. \cite[Lemma 1]{Jutila1977}).
	
	\begin{lemma}[Approximate functional equation]\label{apfe}
		Let $\chi$ be a Dirichlet character modulo $q$. For any $t$ with $2T \geq |t|\sim T_{0}\geq (qT)^{\epsilon}$, we have 
		\begin{align*}
			&\bigg| \sum_{\substack{a \bmod q}}\sum_{m \in \mathbb {Z}}\chi(a)e_q(am)\hat{h}_{t}\bigg(\frac{Nm}{q}\bigg)\bigg|
			\\
			&\quad \ll\frac{q}{N^{1/2}}\int_{|u|\leq (qT)^{\epsilon/2}}\bigg|\sum_{1\leq m \leq (qT)^{2\epsilon}qT_{0}/N}\chi(m)m^{-1/2+i(t-u)}\bigg|\:\mathrm{d}u +O_{\epsilon}((qT)^{-100}).
		\end{align*}
		Additionally, if $|t|\leq (qT)^{\epsilon}$ and $\chi$ is not principal, we have
		\begin{align*}
			&\bigg| \sum_{\substack{a \bmod q \\ (a, q)=1}}\sum_{m \in \mathbb {Z}}\chi(a)e_q(am)\hat{h}_{t}\bigg(\frac{Nm}{q}\bigg)\bigg|
			\\
			&\quad \ll\frac{q}{N^{1/2}}\int_{|u|\leq (qT)^{\epsilon/2}}\bigg|\sum_{1\leq m \leq (qT)^{2\epsilon}q/N}\chi(m)m^{-1/2+i(t-u)}\bigg|\:\mathrm{d}u +O_{\epsilon}((qT)^{-100}).
		\end{align*}
	\end{lemma}
	\begin{proof}
		We first observe, by the Poisson summation formula, that
		\begin{align}\label{PPS}
			\frac{q}{N}\sum_{n \in \mathbb {Z}}h_{t}\bigg(\frac{n}{N}\bigg)\chi(n)
			&=\frac{q}{N}\sum_{\substack{a \bmod q}}\sum_{\substack{n\in \Z \\ n \equiv a \bmod q}}h_{t}\bigg(\frac{n}{N}\bigg)\chi(a)
			\notag
			\\
			&=\sum_{\substack{a \bmod q}}\chi(a)\sum_{m \in \mathbb {Z}}e_q(am)\hat{h}_{t}\bigg(\frac{Nm}{q}\bigg).
		\end{align}
		Let $H(s):=\int_{\mathbb{R}}w(u)^{2}u^{s-1}\:\mathrm{d}u$ be the Mellin transform of $w^{2}$, which is entire and satisfies
		\begin{equation}\label{PrpW}
			|H(s)|=\bigg| \int_{0}^{\infty}w(u)^{2}u^{\sigma+it-1}\:\mathrm{d}u\bigg| =\bigg| \int_{\mathbb{R}}w(e^{v})^{2}e^{v(\sigma+it)}\:\mathrm{d}v\bigg|\ll_{\sigma,j}|t|^{-j}.
		\end{equation}
		By Mellin inversion
		\[w(u)^{2}=\frac{1}{2\pi i}\int_{2-i \infty}^{2+i \infty}H(s)u^{-s}\:\mathrm{d}s,\]
		we see that from \eqref{PPS}
		
		\begin{align*}
			\sum_{\substack{a \bmod q \\ (a, q)=1}}\chi(a)\sum_{m \in \mathbb {Z}}e_q(am)\hat{h}_{t}\bigg(\frac{Nm}{q}\bigg)&=\frac{q}{N^{1+it}}\sum_{n \in \mathbb {Z}}w\bigg(\frac{n}{N}\bigg)^{2}\chi(n)n^{it}
			\\
			&=\frac{q}{N^{1+it}}\sum_{n \in \mathbb {Z}}\frac{1}{2\pi i}\int_{2-i\infty }^{2+i\infty}H(s)\bigg(\frac{n}{N}\bigg)^{-s}\chi(n)n^{it}\:\mathrm{d}s \\
			&= \frac{q}{N^{1+it}}\frac{1}{2\pi i}\int_{2-i \infty }^{2+i\infty}H(s)N^{s}L(s-it, \chi)\:\mathrm{d}s.
		\end{align*}
		The right hand side is equal to 
		\begin{align}\label{aMI}
			\frac{q}{N^{1+it}}\frac{1}{2\pi i}\int_{2-i \infty}^{2+i \infty}H(s)N^{s}L(s-it, \chi^{*})\prod_{p | q,\, p\nmid q^{*}}\bigg(1-\frac{\chi^{*}(p)}{p^{s-it}}\bigg)\:\mathrm{d}s,
		\end{align}
		by the factorization of $L(s-it, \chi)$, where $\chi$ is induced by a primitive character $\chi^{*}$ modulo $q^{*}$.
		By shifting the line of integration to $\Re(s) = -1$ and picking up the residue at $s = 1 + it$ when $L(s-it, \chi^{*}) = \zeta(s-it)$ (i.e., $\chi$ is principal), \eqref{aMI} is equal to 
		\begin{align*}
			\varepsilon(\chi)\frac{\phi(q)}{q}H(1+it)N^{1+it}+
			\frac{1}{2\pi i}\int_{-1-i \infty}^{-1+i \infty}H(s)N^{s}L(s-it, \chi^{*})\prod_{p | q,\, p\nmid q^{*}}\bigg(1-\frac{\chi^{*}(p)}{p^{s-it}}\bigg)\:\mathrm{d}s,
		\end{align*}
		where $\varepsilon(\chi)=1$ when $\chi$ is principal and $\varepsilon(\chi)=0$ otherwise. In the case $\varepsilon(\chi)=1$ and $|t|\sim T_{0}\geq (qT)^{\epsilon}$, we see that by \eqref{PrpW}
		\[\frac{\phi(q)}{q}H(1+it)N^{1+it} \ll_{\epsilon}(qT)^{-100}.\]
		Inserting this in the above equation, using the functional equation of $L(s-it, \chi^{*})$ (cf. \cite[Chapter 9]{DavenportBook}) and the factorization of $L(1-s+it, \overline{\chi})$, one has
		\begin{align*}
			&\frac{1}{2\pi i}\int_{2-i \infty}^{2+i \infty}H(s)N^{s}L(s-it, \chi^{*})\prod_{p | q,\, p\nmid q^{*}}\bigg(1-\frac{\chi^{*}(p)}{p^{s-it}}\bigg)\:\mathrm{d}s
			\\
			=&
			\frac{1}{2\pi i}\int_{-1-i\infty}^{-1+i\infty}H(s)N^{s}G(s-it,\chi^{*})L(1-s+it, \overline{\chi^{*}})\prod_{p | q,\, p\nmid q^{*}}\bigg(1-\frac{\chi^{*}(p)}{p^{s-it}}\bigg)\mathrm{d}s+O_{\epsilon}((qT)^{-100})
			\\
			=&
			\frac{1}{2\pi i}\int_{-1-i\infty}^{-1+i\infty}H(s)N^{s}G(s-it)P(s-it)L(1-s+it, \overline{\chi})\:\mathrm{d}s
			+O_{\epsilon}((qT)^{-100}).
		\end{align*}
		where
		\[G(s-it)=\bigg(\frac{q^*}{\pi}\bigg)^{\frac{1}{2}-s+it}\frac{\tau(\chi^{*})}{i^{\delta}\sqrt{q^*}}\frac{\Gamma(\frac{1-s+it+\delta}{2})}{\Gamma(\frac{s-it+\delta}{2})},\quad P(s-it) = \prod_{p | q,\, p\nmid q^{*}}\frac{1-\frac{\chi^{*}(p)}{p^{s-it}}}{1-\frac{\overline{\chi^{*}}(p)}{p^{1-s+it}}}\]
		and 
		\begin{equation*} \tau(\chi^{*})=\sum_{a \bmod q^* }\chi^{*}(a)e_{q*}(a), \quad   \delta=\bigg\{
			\begin{array}{rcl}
				0     &      &\mbox{if} \ \ \chi^{*}(-1)=1,       \\
				1              &      & \mbox{if} \ \  \chi^{*}(-1)=-1.
			\end{array} \bigg.
		\end{equation*}
		
		We define $M=(qT)^{2\epsilon}qT_{0}/N$. If \( |t| \leq (qT)^{\epsilon} \), we instead set  $M=(qT)^{2\epsilon}q/N$. Additionally, we write
		\[L(1-s+it,\overline{\chi})=\sum_{1\leq m \leq M}\frac{\overline{\chi}(m)}{m^{1-s+it}}+\sum_{m> M}\frac{\overline{\chi}(m)}{m^{1-s+it}}.\]
		Substituting this expression into the integral above and shifting the contour of the integral involving terms with $m \leq M$ to $\Re (s)=1/2$, while shifting the integral involving terms with $m > M$ to $\Re (s)=-2k$ for a large $k \in \mathbb{N}$, one concludes that
		\begin{equation}\label{inti12}
			\frac{1}{2\pi i}\int_{-1-i\infty}^{-1+i\infty}H(s)N^{s}G(s-it)P(s-it)L(1-s+it, \overline{\chi})\:\mathrm{d}s=I_{1}+I_{2},
		\end{equation}
		where
		\begin{equation*}
			I_1:=\frac{1}{2\pi }\int_{-\infty}^{\infty}H(1/2+iu)N^{\frac{1}{2}+iu}G(1/2+i(u-t))P(1/2+iu-it)\sum_{1 \leq m \leq M}\frac{\overline{\chi}(m)}{m^{1/2-iu+it}}\:\mathrm{d}u
		\end{equation*}
		and 
		\begin{align*}
			I_{2}:=\frac{1}{2\pi }\int_{-\infty}^{\infty}H(-2k+iv)N^{-2k+iv}&G(-2k+iv-it)P(-2k+iv-it)\sum_{m > M}\frac{\overline{\chi}(m)}{m^{2k+1-iv+it}}\:\mathrm{d}v.
		\end{align*}
		Thanks to the rapid decay of $H(s)$ (see \eqref{PrpW}), we can truncate the integrals to $|u|, |v| \leq (qT)^{\epsilon/2}$ at the cost of an $O_{\epsilon, k}((qT)^{-200})$ error term. Using the basic estimate and the functional equation of the Gamma function (cf. \cite[Chapter 10]{DavenportBook}), as well as the property $|\tau(\chi^{*})|=\sqrt{q^{*}}$, we deduce that for $|u|, |v| \leq (qT)^{\epsilon/2}$, $|t|\sim T_{0}\geq (qT)^{\epsilon}$, 
		\begin{align*}
			|G(-2k+i(v-t))|&=\bigg(\frac{q^{*}}{\pi}\bigg)^{\frac{1}{2}+2k}\bigg|\frac{\Gamma(\frac{1+2k+i(t-v)+\delta}{2})}{\Gamma(\frac{-2k+i(v-t)+\delta}{2})}\bigg|
			\\
			&=\bigg(\frac{q^{*}}{\pi}\bigg)^{\frac{1}{2}+2k}\bigg|\frac{\Gamma(\frac{1+2k+i(t-v)+\delta}{2})}{\Gamma(\frac{2k+i(v-t)+\delta}{2})}\bigg|\prod_{j=-k+1}^{k}\bigg|-j+\frac{i(v-t)}{2}+\frac{\delta}{2}\bigg| \\
			&\ll_{k} (q^{*}T_{0})^{\frac{1}{2}+2k}.
		\end{align*}
		Also, we have that
		\begin{align*}
			|G(1/2+i(u-t))|
			=\bigg|\frac{\Gamma(\frac{1/2+i(t-u)+\delta}{2})}{\Gamma(\frac{1/2+i(u-t)+\delta}{2})}\bigg|
			\ll 1.
		\end{align*}
		It is also clear that when $|u|, |v| \leq (qT)^{\epsilon/2}$, $|t| \leq (qT)^{\epsilon}$, 
		\begin{align*}
			|G(-2k+i(v-t))|\ll_{k}(q^{*})^{\frac{1}{2}+2k}(qT)^{\epsilon(\frac{1}{2}+2k)},
		\end{align*}
		\begin{align*}
			|G(1/2+i(u-t))|
			\ll 1.
		\end{align*}
		Inserting these bounds into the expressions for $I_{1}$ and $I_{2}$\footnote{Here, we applied the crude bounds  $\bigg|1-\frac{\chi^{*}(p)}{p^{-2k+iv-it}}\bigg| \leq p^{2k+1}$ and
			$\bigg|1-\frac{\overline{\chi^{*}}(p)}{p^{1+2k-iv+it}}\bigg|^{-1}\leq p$ to control the product factors in estimating $I_{2}$.}, we obtain that when $|t|\sim T_{0}\geq (qT)^{\epsilon}$,
		\[|I_{2}| \ll_{k} 
		\frac{(qT)^{\epsilon/2}N^{-2k}(q^{*}T_{0})^{1/2+2k}}{((qT)^{2\epsilon}qT_{0}/N)^{2k}}\bigg(\frac{q}{q^{*}}\bigg)^{2k+2}
		+O_{\epsilon, k}((qT)^{-200})\]
		and when $|t| \leq (qT)^{\epsilon}$
		\[|I_{2}| \ll_{k} 
		\frac{(qT)^{\epsilon/2}N^{-2k}(q^{*})^{1/2+2k}(qT)^{\epsilon(1/2+2k)}}{((qT)^{2\epsilon}q/N)^{2k}}\bigg(\frac{q}{q^{*}}\bigg)^{2k+2}
		+O_{\epsilon, k}((qT)^{-200})\]
		and finally
		\[|I_{1}| \ll N^{\frac{1}{2}} \int_{|u|\leq(qT)^{\epsilon/2}}\bigg|\sum_{1 \leq m \leq M}\frac{\overline{\chi}(m)}{m^{\frac{1}{2}-iu+it}}\bigg|\:\mathrm{d}u+O_{\epsilon}((qT)^{-200}).\]
		By selecting $k$ sufficiently large in terms of $\epsilon$, we then see that $I_{2}=O_{\epsilon}((qT)^{-200})$. Finally, substituting these bounds into \eqref{inti12} and recalling \eqref{aMI} yields the result.
	\end{proof}
	
	Proposition \ref{S2B} is readily established by combining the preceding lemma with the mean value theorem for Dirichlet polynomials.
	
	\vspace{1 pt}
	\begin{proof}[Proof of Proposition \ref{S2B}]
		By symmetry and the orthogonality of characters, we have
		\begin{align*}
			S_{2}&=\frac{3N^{3}}{q^{3}}\sum_{m_{1},\, m_{2}\neq 0}\sum_{\substack{(t_{i}, \chi_{i})\in W \\ 1\leq i \leq 3 \\ \text{not all equal}}}\sum_{\substack{a_1,a_2,a_3 \bmod q \\ (a_1a_2a_3,q)=1}}\chi_{1}\overline{\chi}_{2}(a_{1})
			\chi_{2}\overline{\chi}_{3}(a_{2})
			\chi_{3}\overline{\chi}_{1}(a_{3})e_q(a_1m_1+a_2m_2)
			\\
			&\quad\quad\quad\quad\quad\quad\quad\quad\quad\quad\quad\quad\quad\quad\quad\quad  \cdot\hat{h}_{t_{1}-t_{2}}\bigg(\frac{Nm_{1}}{q}\bigg)\hat{h}_{t_{2}-t_{3}}\bigg(\frac{Nm_{2}}{q}\bigg)\hat{h}_{t_{3}-t_{1}}(0)
			\\
			&=\frac{3\phi(q)N^{3}}{q^{3}}\sum_{m_{1},\, m_{2}\neq 0}\sum_{\substack{\chi_{1}=\chi_{3} \\ (t_{i}, \chi_{i})\in W, \,1\leq i \leq 3 \\ \text{not all equal}}}\sum_{\substack{a_1,a_2 \bmod q \\ (a_1a_2,q)=1}}\chi_{1}\overline{\chi}_{2}(a_{1})
			\chi_{2}\overline{\chi}_{3}(a_{2}) e_q(a_1m_1+a_2m_2)
			\\
			&\quad\quad\quad\quad\quad\quad\quad\quad\quad\quad\quad\quad\quad\quad\quad\quad\phantom{abcdef}\cdot\hat{h}_{t_{1}-t_{2}}\bigg(\frac{Nm_{1}}{q}\bigg)\hat{h}_{t_{2}-t_{3}}\bigg(\frac{Nm_{2}}{q}\bigg)\hat{h}_{t_{3}-t_{1}}(0).
		\end{align*}
		When $\chi_{1}=\chi_{3}$, the condition $t_{1}\neq t_{3}$ implies $|t_{1}-t_{3}|\geq (qT)^{\epsilon}$. In this case, it follows from \eqref{case1h} that  $|\hat{h}_{t_{3}-t_{1}}(0)|=O_{\epsilon}((qT)^{-100})$. Additionally, applying the bound $\hat{h}_{t}(u)\ll (1+|t|^{2})/|u|^{2}$ from Lemma \ref{efh} (1) to $\hat{h}_{t_{1}-t_{2}}(Nm_{1}/q)$ and $\hat{h}_{t_{2}-t_{3}}(Nm_{2}/q)$, we find that the terms with $t_{1}\neq t_{3}$ contribute $O_{\epsilon}((qT)^{-10})$ in total. The condition on the sum over $W$ now becomes $(t_1,\chi_1) \neq (t_2,\chi_2)$. Thus, we conclude that
		\begin{align*}
			S_{2}=\frac{3\phi(q)N^{3}\hat{h}_{0}(0)}{q^{3}}\sum_{m_{1},\, m_{2}\neq 0}\sum_{\substack{(t_1,\chi_1) \neq (t_2,\chi_2) }}&\sum_{\substack{a_1,a_2 \bmod q \\ (a_1a_2,q)=1}}\chi_{1}\overline{\chi}_{2}(a_{1})
			\chi_{2}\overline{\chi}_{1}(a_{2})
			e_q(a_1m_1+a_2m_2)
			\\
			&\ \cdot\hat{h}_{t_{1}-t_{2}}\bigg(\frac{Nm_{1}}{q}\bigg)\hat{h}_{t_{2}-t_{1}}\bigg(\frac{Nm_{2}}{q}\bigg)+O_{\epsilon}((qT)^{-10}).
		\end{align*}
		Recalling $h_{t}(u)=w(u)^{2}u^{it}$, we observe that $h_{-t}(u)=\overline{h_{t}(u)}$, which implies $\hat{h}_{-t}(\xi)=\overline{\hat{h}_{t}(-\xi)}$. In particular, this yields the relation: 
		\[\hat{h}_{t_{2}-t_{1}}\bigg(\frac{Nm_{2}}{q}\bigg)=\overline{\hat{h}_{t_{1}-t_{2}}\bigg(\frac{-Nm_{2}}{q}\bigg)}.\]
		Consequently, we have
		\begin{align*}
			S_{2}=\frac{3\phi(q)N^{3}\hat{h}_{0}(0)}{q^{3}}\sum_{\substack{(t_1.\chi_1)\neq (t_2,\chi_2)}}\bigg|\sum_{m\neq 0}\sum_{\substack{a \bmod q \\ (a, q)=1}}\chi_{1}
			\overline{\chi}_{2}(a)
			e_q(am)
			\hat{h}_{t_{1}-t_{2}}\bigg(\frac{Nm}{q}\bigg)\bigg|^{2}+O_{\epsilon}((qT)^{-10}).
		\end{align*}
		First, we note that if $|t_1-t_2| \leq (qT)^{\epsilon}$ then $\chi_1\neq \chi_2$. We split the sum according to this condition and then split the sum over $|t_1-t_2| \geq (qT)^{\epsilon}$ dyadically so that
		\begin{align*}
			S_2 \ll S_{2,1}+S_{2,2} + O_{\epsilon}((qT)^{-10}),
		\end{align*}
		where 
		\begin{align*}
			S_{2,1}:=\frac{\phi(q)N^{3}}{q^{3}}\sum_{\substack{(t_1,\chi_1),(t_2,\chi_2) \in W \\ \chi_{1}\neq \chi_{2}\\|t_{1}-t_{2}|\leq (qT)^{\epsilon} }}\bigg|\sum_{m\neq 0}\sum_{\substack{a \bmod q\\ (a, q)=1}}\chi_{1}
			\overline{\chi}_{2}(a)
			e_q(am)
			\hat{h}_{t_{1}-t_{2}}\bigg(\frac{Nm}{q}\bigg)\bigg|^{2}
		\end{align*}
		and 
		\begin{align*}
			S_{2,2}:=\frac{\phi(q)N^{3}T^{\epsilon}}{q^{3}}\sup_{1\leq j \leq \lceil\frac{\log 2T}{\log 2}\rceil+1 }\sum_{\substack{(t_1,\chi_1), (t_2,\chi_2) \in W \\|t_{1}-t_{2}|\sim (qT)^{\epsilon}2^{j-1} }}\bigg|\sum_{m\neq 0}\sum_{\substack{a \bmod q \\ (a, q)=1}}\chi_{1}
			\overline{\chi}_{2}(a)
			e_q(am)
			\hat{h}_{t_{1}-t_{2}}\bigg(\frac{Nm}{q}\bigg)\bigg|^{2}.
		\end{align*}
		We work with $S_{2,2}$ first, the analysis for $S_{2,1}$ is very similar. By Lemma \ref{apfe}, we have\footnote{The terms with $m=0$ can be safely excluded, introducing only an acceptable error term of $O_{\epsilon}((qT)^{-100})$, as guaranteed by the bound in \eqref{case1h}.} for $|t_{1}-t_{2}|\sim (qT)^{\epsilon}2^{j-1}$
		\begin{align*}
			&\bigg|\sum_{m\neq 0}\sum_{\substack{a \bmod q \\ (a, q)=1}}\chi_{1}
			\overline{\chi}_{2}(a)
			e_q(am)
			\hat{h}_{t_{1}-t_{2}}\bigg(\frac{Nm}{q}\bigg)\bigg|
			\\
			&\ll \frac{q}{N^{1/2}}\int_{|u|\leq (qT)^{\epsilon/2}}\bigg|\sum_{1\leq m \leq (qT)^{2\epsilon}q(qT)^{\epsilon}2^{j-1}/N}\chi_{1}
			\overline{\chi}_{2}(m)m^{-1/2+i(t_{1}-t_{2}-u)}\bigg|\:\mathrm{d}u +O_{\epsilon}((qT)^{-100}).
		\end{align*}
		Squaring and summing over $(t_{1}, \chi_{1}), (t_{2}, \chi_{2})\in W$ with $|t_{1}-t_{2}|\sim (qT)^{\epsilon}2^{j-1}$ yields that $S_{2,2}$ is
		\begin{align*}
			\ll N^{2}T^{2\epsilon}\sup_{\substack{1\leq j \leq \lceil\frac{\log 2T}{\log 2}\rceil+1 \\ |u|\leq(qT)^{\epsilon/2} }}\sum_{\substack{(t_1,\chi_1), (t_2,\chi_2) \in W \\|t_{1}-t_{2}|\sim (qT)^{\epsilon}2^{j-1}}}\bigg|\sum_{1\leq m \leq \frac{(qT)^{3\epsilon}q2^{j-1}}{N}}\frac{\chi_{1}
				\overline{\chi}_{2}(m)}{m^{1/2+i(t_{2}-t_{1}+u)}}\bigg|^{2}+O_{\epsilon}((qT)^{-10}).
		\end{align*}
		We can relax the condition of the outer sum to $(t_1,\chi_1),(t_2,\chi_2)\in W$ to derive an upper bound. We split the summation range over $m$ into dyadic ranges. The dyadic range when $m \sim 1$ gives a contribution that dominates the error term so we may absorb it into the main term. Then we have
		\[S_{2,2} \ll_{\epsilon} (qT)^{3\epsilon}N^{2}\sum_{\substack{(t_1,\chi_1),(t_2,\chi_2)\in W}}\bigg|\sum_{ m\sim M}\frac{a_{m}\chi_{1}
			\overline{\chi}_{2}(m)}{m^{1/2+i(t_{2}-t_{1})}}\bigg|^{2},\]
		for some choice of coefficients $|a_m| \leq 1$ and some choice of $M\leq (qT)^{3\epsilon}qT/N$.
		Similarly\footnote{The orthogonality relation for Dirichlet characters allows us to exclude the $m=0$ terms in the application of Lemma \ref{apfe} as for $S_{2,1}$, we have $\chi_1 \neq \chi_2$.}, the same estimate holds for $S_{2,1}$. Hence, we arrive at
		\begin{equation}\label{S2B1}
			S_{2}\lessapprox_{\epsilon} \frac{N^{2}}{M}\sum_{\substack{(t_1,\chi_1), (t_2,\chi_2) \in W}}\bigg|\sum_{ m\sim M}\frac{a_{m}\chi_{1}
				\overline{\chi}_{2}(m)}{m^{i(t_{2}-t_{1})}}\bigg|^{2},
		\end{equation}
		where $|a_{m}|\leq 1$ and $M\leq (qT)^{3\epsilon}qT/N$.
		
		Now we proceed with an argument based on the mean value theorem. Notice that $|{\sum_{n \sim N} w(n/N)b_n n^{it_2}\chi_2(n)}| \geq N^{\sigma}$ for $(t_2,\chi_2) \in W$. We have  
		\begin{align*}
			S_{2} &\lessapprox_{\epsilon}  \frac{N^{2-2\sigma}}{M}\sum_{\substack{(t_1,\chi_1),(t_2,\chi_2) \in W}} \bigg|\sum_{m \sim M} a_m m^{i(t_2-t_1)} \chi_2\overline{\chi_1}(m)\bigg|^2 \cdot \bigg|\sum_{n \sim N} w\bigg(\frac{n}{N}\bigg)b_n n^{it_2} \chi_2(n)\bigg|^2 \\
			& \leq \frac{|W| N^{2-2\sigma}}{M}\sup\limits_{(t_1,\chi_1) \in W} \sum_{(t_2,\chi_2) \in W} \bigg|\sum_{MN < l \leq 4NM} c_l^{(t_1,\chi_1)}l^{it_2}\chi_2(l)\bigg|^2,
		\end{align*}
		for some coefficients $c_l$ depending on $(t_1,\chi_1)$ such that $|c_l^{(t_1,\chi_1)}| \lessapprox 1$, by the divisor bound. Then by the mean value theorem for Dirichlet polynomials with characters \cite[Theorem 9.12]{iwanieckowalski}, we have that
		\begin{align*}
			S_{2}& \lessapprox_{\epsilon} \frac{|W|N^{2-2\sigma}}{M}(N^2M^2+qTNM) \lessapprox_{\epsilon} qT|W|N^{3-2\sigma},
		\end{align*}
		since $NM \lessapprox_{\epsilon} qT$. 
	\end{proof}
	\begin{remark}
		This argument to bound $S_2$ differs from the argument given by Guth-Maynard to prove their analogous Proposition 6.1. In particular, we reintroduce our large Dirichlet polynomial here whereas their proof only depended on the 1-separation of $W$. This is because the term $qT|W|N^{3-2\sigma}$ will appear in our bound for $S_3$ and therefore this slightly strengthens Proposition \ref{Auxiliary theorem}. This simplifies our application of subdivision in the proof of Theorem \ref{Partial LVE} and slightly strengthens the result. 
	\end{remark}
	
	\section{\texorpdfstring{The contribution of $S_{3}$}
		{The contribution of S3}}\label{TCS3}
	
	The purpose of this section is to begin the analysis of the $S_{3}$ term. Let us first recall that 
	\[S_{3}=\sum_{m_{1}, m_{2}, m_{3}\neq 0}I_{\overrightarrow{m}},\]
	where
	\begin{align*}I_{\overrightarrow{m}}=\frac{N^{3}}{q^{3}}\sum_{\substack{(t_1,\chi_1),(t_2,\chi_2),(t_3,\chi_3) \in W \\ \text{not all equal}}}\sum_{\substack{a_1,a_2,a_3 \bmod q \\ (a_1a_2a_3,q)=1}}&\chi_{1}\overline{\chi}_{2}(a_{1})
		\chi_{2}\overline{\chi}_{3}(a_{2})
		\chi_{3}\overline{\chi}_{1}(a_{3})
		e_q(\overrightarrow{a}\cdot\overrightarrow{m})
		\\
		&\cdot \hat{h}_{t_{1}-t_{2}}\bigg(\frac{Nm_{1}}{q}\bigg)\hat{h}_{t_{2}-t_{3}}\bigg(\frac{Nm_{2}}{q}\bigg)\hat{h}_{t_{3}-t_{1}}\bigg(\frac{Nm_{3}}{q}\bigg).
	\end{align*}
	By Lemma \ref{efh} (1), we have $|\hat{h}_{t}(\xi)|\ll_{j}(1+|t|)^{j}/|\xi|^{j}$. This implies that $I_{\overrightarrow{m}}$ is negligible unless $|m_{i}|\leq (qT)^{\epsilon}qT/N$ for $1\leq i \leq 3$. Therefore, we can write
	\[
	S_{3}=\sum_{0<|m_{1}|, |m_{2}|, |m_{3}|\leq (qT)^{\epsilon}qT/N}I_{\overrightarrow{m}}+O_{\epsilon}((qT)^{-100}).
	\]
	We now aim to bound $|I_{\overrightarrow{m}}|$. Before we do this, we remove the condition that $(t_1,\chi_1)$, $(t_2,\chi_2)$, $(t_3,\chi_3)$ are not equal to each other so that we can separate the sums over $W$. That is, we want a bound for $S_3$ in terms of $I_{\overrightarrow{m}}'$ (for $m_1,m_2,m_3 \neq 0$), where 
	\begin{align*}
		I_{\overrightarrow{m}}'=\frac{N^{3}}{q^{3}}\sum_{\substack{(t_1,\chi_1),(t_2,\chi_2),(t_3,\chi_3) \in W}}\sum_{\substack{a_1,a_2,a_3 \bmod q \\ (a_1a_2a_3,q)=1}}&\chi_{1}\overline{\chi}_{2}(a_{1})
		\chi_{2}\overline{\chi}_{3}(a_{2})
		\chi_{3}\overline{\chi}_{1}(a_{3}) e_q(\overrightarrow{a}\cdot\overrightarrow{m})
		\\
		&\cdot \hat{h}_{t_{1}-t_{2}}\bigg(\frac{Nm_{1}}{q}\bigg)\hat{h}_{t_{2}-t_{3}}\bigg(\frac{Nm_{2}}{q}\bigg)\hat{h}_{t_{3}-t_{1}}\bigg(\frac{Nm_{3}}{q}\bigg).
	\end{align*}
	This can be done by removing the terms when $(t_1,\chi_1)= (t_2,\chi_2) = (t_3,\chi_3)$:
	\begin{align*}
		S_3 = \sum_{0 < |m_1|,|m_2|,|m_3| \leq(qT)^{\epsilon} \frac{qT}{N}} I_{\overrightarrow{m}}' {-} Z'+O_{\epsilon}((qT)^{-100}),
	\end{align*}
	where
	\begin{align*}
		Z':=|W|\bigg(\frac{N}{q}\sum_{0 < |m|\leq(qT)^{\epsilon} \frac{qT}{N}} \hat{h}_{0}\bigg(\frac{Nm}{q}\bigg)\sum_{\substack{ a \bmod q \\ (a,q)=1}}e_q(am)\bigg)^{3}.
	\end{align*}
	Fortunately, the condition $m\neq 0$ makes this term small. By Lemma \ref{efh}(1), for $m\neq 0$,
	$
	|\hat{h}_0(Nm/q)|
	\ll \frac{q}{N|m|}.
	$
	On the other hand, the arithmetic sum is the Ramanujan sum $C_q(m)$ and $|C_q(m)|
	\leq (m,q).
	$
	It follows that
	\[
	|Z'|
	\ll
	|W|
	\bigg(
	\sum_{0<m \leq(qT)^{\epsilon} qT/N}
	\frac{(m,q)}{m}
	\bigg)^3
	\lessapprox_{\epsilon} |W|.
	\]
	Therefore we have 
	\begin{align}\label{S3cut}
		|S_3| \lessapprox_{\epsilon} \sum_{0 < |m_1|,|m_2|, |m_3| \leq(qT)^{\epsilon} \frac{qT}{N}} |I_{\overrightarrow{m}}'| + |W|,
	\end{align}
	where the error term has been absorbed as it is dominated by $|W|$.
	Our first step is to establish an upper bound for $I_{\overrightarrow{m}}'$. Before stating the result, we introduce the function
	\begin{equation}\label{funcR}
		R(v,a):=\sum_{(t, \chi) \in W}v^{it}\chi(a), \quad v > 0, \ a\in \mathbb{Z}/q\mathbb{Z},
	\end{equation}
	which plays a crucial role throughout this section.
	\begin{proposition}\label{Imbound}
		We have
		\begin{align*}
			|I_{\overrightarrow{m}}'|\ll_{\epsilon} \frac{N^{3}}{q^{3}}\int\limits_{\substack{|m_{1}v_{1}+m_{2}v_{2}+m_{3}| \leq(qT)^{\epsilon}q/N  \\ v_{1}, v_{2}\in [1/2, 2]}}\sum_{\substack{b_1 \bmod q \\ (b_{1}, q)=1}}\sum_{\substack{b_2 \bmod q \\ (b_{2}, q)=1}}\bigg| R(v_{1}, b_{1})R\bigg(\frac{v_{2}}{v_{1}}, b_{1}^{-1}b_{2}\bigg)R\bigg(\frac{1}{v_{2}}, b_{2}^{-1}\bigg)\bigg|
			\\
			\cdot(b_{1}m_{1}+b_{2}m_{2}+m_{3},q)\:\mathrm{d}v_{1}\mathrm{d}v_{2}+O_{\epsilon}((qT)^{-200}),
		\end{align*}
	\end{proposition}
	\begin{proof}
		Expanding the definition of $\hat{h}$ as an integral and interchanging the order of summation and integration, one finds 
		\begin{align*}
			I_{\overrightarrow{m}}'&=\frac{N^{3}}{q^{3}}\sum_{\substack{(t_{i}, \chi_{i})\in W \\ 1\leq i \leq 3}}\sum_{\substack{a_1,a_2,a_3 \bmod q \\ (a_1a_2a_3,q)=1}}\int_{\mathbb{R}^{3}}\chi_{1}(a_{1})\overline{\chi}_{2}(a_{1})
			\chi_{2}(a_{2})\overline{\chi}_{3}(a_{2})
			\chi_{3}(a_{3})\overline{\chi}_{1}(a_{3})e_q(\overrightarrow{a}\cdot\overrightarrow{m})
			\\
			&\quad \quad \quad \quad \quad \quad\quad \quad \quad \ \ \cdot e\bigg(-\frac{N}{q}\overrightarrow{m}\cdot\overrightarrow{u}\bigg)w_{1}(\overrightarrow{u})u_{1}^{i(t_{1}-t_{2})}u_{2}^{i(t_{2}-t_{3})}u_{3}^{i(t_{3}-t_{1})}\:\mathrm{d}u_{1}\mathrm{d}u_{2}\mathrm{d}u_{3}
			\\
			&=\frac{N^{3}}{q^{3}}\sum_{\substack{a_1,a_2,a_3 \bmod q \\ (a_1a_2a_3,q)=1}}\int_{\mathbb{R}^{3}}e_q(\overrightarrow{a}\cdot\overrightarrow{m})
			e\bigg(-\frac{N}{q}\overrightarrow{m}\cdot\overrightarrow{u}\bigg)w_{1}(\overrightarrow{u})
			\\
			&\quad \quad \quad\quad \quad \quad\quad \quad \quad\cdot R\bigg(\frac{u_{1}}{u_{3}}, a_{1}a_{3}^{-1}\bigg)R\bigg(\frac{u_{2}}{u_{1}}, a_{2}a_{1}^{-1}\bigg)R\bigg(\frac{u_{3}}{u_{2}}, a_{3}a_{2}^{-1}\bigg)\:\mathrm{d}u_{1}\mathrm{d}u_{2}\mathrm{d}u_{3},
		\end{align*}
		where
		\[w_{1}(\overrightarrow{u})=w(u_{1})^{2}w(u_{2})^{2}w(u_{3})^{2}.\]
		As pointed out in \cite[p. 21]{GM}, a key observation is that the arguments of the $R$ functions above are restricted to a two-dimensional subvariety defined by $z_{1}z_{2}z_{3}=1$. Additionally, $(a_{1}a_{3}^{-1})\cdot(a_{2}a_{1}^{-1})\cdot(a_{3}a_{2}^{-1})\equiv 1 \bmod q$. With this in mind, we make the change of variables
		\[v_{1}=\frac{u_{1}}{u_{3}}, \quad v_{2}=\frac{u_{2}}{u_{3}}, \quad 
		b_{1}=a_{1}a_{3}^{-1}, \quad
		b_{2}=a_{2}a_{3}^{-1},\]
		and note that
		\[e\bigg(-\frac{N}{q}\overrightarrow{m}\cdot\overrightarrow{u}\bigg)=e\bigg(-\frac{N}{q}(m_{1}u_{1}+m_{2}u_{2}+m_{3}u_{3})\bigg)=e\bigg(-\frac{N}{q}(m_{1}v_{1}+m_{2}v_{2}+m_{3})u_{3}\bigg),\]
		\[e_q(\overrightarrow{a}\cdot\overrightarrow{m})=e_q({a_{1}m_{1}+a_{2}m_{2}+a_{3}m_{3}})=e_q({(b_{1}m_{1}+b_{2}m_{2}+m_{3})a_{3}})\]
		and
		\[\mathrm{d}u_{1}\mathrm{d}u_{2}\mathrm{d}u_{3}=u_{3}^{2}\mathrm{d}v_{1}\mathrm{d}v_{2}\mathrm{d}u_{3}.\]
		Then our equation becomes 
		\begin{align*}
			I_{\overrightarrow{m}}'
			&=\frac{N^{3}}{q^{3}}\sum_{\substack{b_1,b_2,a_3 \bmod q \\ (b_1b_2,q)=1 \\ (a_3,q)=1}}\int_{\mathbb{R}^{3}}e\bigg(-\frac{N}{q}(m_{1}v_{1}+m_{2}v_{2}+m_{3})u_{3}\bigg)e_q({(b_{1}m_{1}+b_{2}m_{2}+m_{3})a_{3}})
			\\
			&\quad \quad \quad\cdot w(v_{1}u_{3})^{2}w(v_{2}u_{3})^{2}w(u_{3})^{2}R(v_{1}, b_{1})R\bigg(\frac{v_{2}}{v_{1}}, b_{2}b_{1}^{-1}\bigg)R\bigg(\frac{1}{v_{2}}, b_{2}^{-1}\bigg)u_{3}^{2}\:\mathrm{d}v_{1}\mathrm{d}v_{2}\mathrm{d}u_{3}.
		\end{align*}
		Since the $R$ factors do not depend on $u_{3}, a_{3}$, performing the $u_{3}$ integral and the $a_{3}$ summation first yields
		\begin{align}\label{eqIm}
			I_{\overrightarrow{m}}'
			&=\frac{N^{3}}{q^{3}}\sum_{\substack{ b_1,b_2 \bmod q \\ (b_1b_2, q)=1}}\int_{\mathbb{R}^{2}}C_{q}(b_{1}m_{1}+b_{2}m_{2}+m_{3})R(v_{1}, b_{1})R\bigg(\frac{v_{2}}{v_{1}}, b_{2}b_{1}^{-1}\bigg)R\bigg(\frac{1}{v_{2}}, b_{2}^{-1}\bigg)
			\notag
			\\
			&\quad\cdot \bigg(\int_{\mathbb{R}}e\bigg(-\frac{N}{q}(m_{1}v_{1}+m_{2}v_{2}+m_{3})u_{3}\bigg)w(v_{1}u_{3})^{2}w(v_{2}u_{3})^{2}w(u_{3})^{2}u_{3}^{2}\:\mathrm{d}u_{3}\bigg)\:\mathrm{d}v_{1}\mathrm{d}v_{2}.
		\end{align}
		As $w(u)$ is a smooth bump supported on $[1, 2]$, we observe that \[w(v_{1}u_{3})^{2}w(v_{2}u_{3})^{2}w(u_{3})^{2}u_{3}^{2}\] vanishes unless $v_{1}, v_{2} \in [1/2, 2]$. Furthermore, its $j^{th}$ derivative with respect to $u_{3}$ is bounded by $O_{j}(1)$. Thus, by repeated integration by parts, the inner integral is $O_{\epsilon}((qT)^{-300})$ unless $|m_{1}v_{1}+m_{2}v_{2}+m_{3}| \leq(qT)^{\epsilon}q/N$ and $v_{1}, v_{2}\in [1/2, 2]$. On the other hand, the inner integral is $O(1)$ trivially. 
		Inserting these bounds into \eqref{eqIm} and applying the triangle inequality along with the fact that $|C_{q}(m)| \leq (m,q)$, yields the desired result.
	\end{proof}
	
	We now introduce a smooth approximation of $|R|$ which is defined as follows.
	Let $\tilde{\psi}_{0}(x)$ be a non-negative smooth bump supported on $[-2, 2]$ such that $\tilde{\psi}_{0}(x)=1$ for $|x|\leq 1$, $||\tilde{\psi}_0^{(j)}||_{\infty} \ll_j 1$ and
	\begin{align}
		\tilde{\psi}(x)=\tilde{\psi}_{0}\bigg(\frac{x}{2(qT)^{\epsilon}}\bigg) \label{psi definition}
	\end{align}
	For a given parameter $M_{2}>0$, we define
	\begin{equation}\label{deftiR}
		\tilde{R}_{M_{2}}(u, a):= \bigg(\int\limits_{ u' \in [1/2,\, 2]}\frac{NM_{2}}{q}\tilde{\psi}\bigg(\frac{NM_{2}}{q}(u-u')\bigg)| R(u', a)|^{2}
		\:\mathrm{d}u' \bigg)^{\frac{1}{2}}.
	\end{equation}

	The following proposition provides an expansion of $S_{3}$ in terms of integrals involving $\tilde{R}$ and $R$.
	
	\begin{proposition}\label{S3B1}
		There is a choice of $0<M_{1}\leq M_{3}\leq M_{2}\leq (qT)^{\epsilon}qT/N$ such that
		\[S_{3}\lessapprox_{\epsilon} \frac{N^{2}}{q^{2}M_{2}}\sum_{\substack{|m_{i}|\sim M_{i} \\ 1\leq i \leq 3}} \tilde{I}_{\overrightarrow{m}}+|W|,\]
		where
		\begin{align*}\tilde{I}_{\overrightarrow{m}}:=\int\limits_{ v_{1} \in [1/2, 2]}\sum_{\substack{b_1,b_2 \bmod q \\ (b_1b_2,q)=1}}\bigg| R(v_{1}, b_{1})\tilde{R}_{M_{2}}\bigg(\frac{m_{1}v_{1}+m_{3}}{m_{2}v_{1}}, b_{1}^{-1}b_{2}\bigg)\tilde{R}_{M_{2}}\bigg(\frac{m_{1}v_{1}+m_{3}}{m_{2}}, b_{2}\bigg)\bigg|
			\\
			\cdot(b_{1}m_{1}-b_{2}m_{2}+m_{3},q)\:\mathrm{d}v_{1}.
		\end{align*}
	\end{proposition}
	\begin{proof}
		From \eqref{S3cut} we have that
		\[|S_{3}|\lessapprox_{\epsilon} \sum_{0<|m_{1}|, |m_{2}|, |m_{3}|\leq (qT)^{\epsilon}qT/N}|I_{\overrightarrow{m}}'|+|W|,\]
		where, as shown in the first equation of the proof of Proposition \ref{Imbound},
		\begin{align*}
			I_{\overrightarrow{m}}'
			&=\frac{N^{3}}{q^{3}}\sum_{\substack{a_1,a_2,a_3 \bmod q \\ (a_1a_2a_3,q)=1}}\int_{\mathbb{R}^{3}}e_q(\overrightarrow{a}\cdot\overrightarrow{m})
			e\bigg(-\frac{N}{q}\overrightarrow{m}\cdot\overrightarrow{u}\bigg)w_{1}(\overrightarrow{u})
			\\
			&\quad \quad \quad\quad \quad \quad\quad \quad \quad\cdot R\bigg(\frac{u_{1}}{u_{3}}, a_{1}a_{3}^{-1}\bigg)R\bigg(\frac{u_{2}}{u_{1}}, a_{2}a_{1}^{-1}\bigg)R\bigg(\frac{u_{3}}{u_{2}}, a_{3}a_{2}^{-1}\bigg)\:\mathrm{d}u_{1}\mathrm{d}u_{2}\mathrm{d}u_{3}.
		\end{align*}
		From the definition of $R(v,a)$,  it follows that (see \eqref{funcR})
		\[R\bigg(\frac{1}{v}, a^{-1}\bigg)=\overline{R(v,a)}.\]
		Consequently, we have the relation $I'_{m_{1},m_{2},m_{3}}=\bar{I'}_{-m_{2},-m_{1},-m_{3}}$. Similarly, we also find: $I'_{m_{1},m_{2},m_{3}}=\bar{I'}_{-m_{1},-m_{3},-m_{2}}$, $I'_{m_{1},m_{2},m_{3}}=\bar{I'}_{-m_{3},-m_{2},-m_{1}}$. 
		These symmetries allow us to reduce the estimation to the case $|m_{1}|\leq |m_{3}|\leq |m_{2}|$ at the cost of a factor $6$. Furthermore, by selecting dyadic scales to maximize the right-hand side, we arrive at
		\begin{equation}\label{trsdyas3}
			|S_{3}|\lessapprox_{\epsilon} \sum_{|m_{i}|\sim M_{i}, \, 1\leq i \leq 3} |I_{\overrightarrow{m}}'|+|W|,
		\end{equation}
		where $M_{1}\leq M_{3}\leq M_{2}\leq (qT)^{\epsilon}qT/N$.
		Applying Proposition \ref{Imbound} gives, for $|m_{2}|\sim M_{2}$,
		\begin{align*}
			|I_{\overrightarrow{m}}'|\ll_{\epsilon} \frac{N^{3}}{q^{3}}&\int\limits_{ v_{1} \in [1/2,\, 2]}\sum_{\substack{b_1 \bmod q \\ (b_{1}, q)=1}}\sum_{\substack{b_2 \bmod q \\ (b_{2}, q)=1}}|R(v_{1}, b_{1})|(b_{1}m_{1}+b_{2}m_{2}+m_{3},q)
			\\
			&\quad\cdot\int\limits_{\substack{|v_{2}-\frac{m_{1}v_{1}+m_{3}}{-m_{2}}|\leq \frac{(qT)^{\epsilon}q}{NM_{2}} \\ v_{2} \in [1/2,\, 2]}}\bigg| R\bigg(\frac{v_{2}}{v_{1}}, b_{1}^{-1}b_{2}\bigg)R\bigg(\frac{1}{v_{2}}, b_{2}^{-1}\bigg)\bigg|
			\:\mathrm{d}v_{2}\mathrm{d}v_{1}+O_{\epsilon}((qT)^{-200}).
		\end{align*}
		Using Cauchy-Schwarz and noting that $|R(1/v_2, b^{-1})|=|R(v_{2}, b)|$, the inner $v_2$ integral is bounded by 
		\begin{align*}
			&\bigg(\int\limits_{\substack{|v_{2}-\frac{m_{1}v_{1}+m_{3}}{-m_{2}}|\leq \frac{(qT)^{\epsilon}q}{NM_{2}} \\ v_{2} \in [1/2,\, 2]}}\bigg| R\bigg(\frac{v_{2}}{v_{1}}, b_{1}^{-1}b_{2}\bigg)\bigg|^{2}
			\:\mathrm{d}v_{2}\bigg)^{\frac{1}{2}}
			\bigg(\int\limits_{\substack{|v_{2}-\frac{m_{1}v_{1}+m_{3}}{-m_{2}}|\leq \frac{(qT)^{\epsilon}q}{NM_{2}} \\ v_{2} \in [1/2,\, 2]}}| R(v_{2}, b_{2})|^{2}
			\:\mathrm{d}v_{2}\bigg)^{\frac{1}{2}}
			\\
			\leq& \frac{q}{NM_{2}}\tilde{R}_{M_{2}}\bigg(\frac{m_{1}v_{1}+m_{3}}{-m_{2}v_{1}}, b_{1}^{-1}b_{2}\bigg)\tilde{R}_{M_{2}}\bigg(\frac{m_{1}v_{1}+m_{3}}{-m_{2}}, b_{2}\bigg).
		\end{align*}
		We therefore obtain 
		\[|I_{m_{1}, m_{2}, m_{3}}'| \ll_{\epsilon} \frac{N^{2}}{q^{2}M_{2}}\tilde{I}_{m_{1}, -m_{2}, m_{3}}+O_{\epsilon}((qT)^{-200}).\]
		Since we are summing over $m_{2}$ with $|m_{2}|\sim M_{2}$, we can replace $-m_{2}$ with $m_{2}$ without altering the overall sum. Inserting this in \eqref{trsdyas3} then gives the result.
	\end{proof}
	
	With Proposition \ref{S3B1} established, three additional tools are required to ultimately derive the $S_{3}$ bound: a sharp estimate for summing over affine transformations with GCD twists, moment bounds for $\tilde{R}$ and $R$ and the energy bound\footnote{For a finite set $W$ of pairs $(t, \chi)$, we define its energy $E(W)$ as
		\[E(W):=\# \{ (t_1,\chi_1), (t_2,\chi_2), (t_3,\chi_3) , (t_4,\chi_4) \in W :     |t_{1}+t_{2}-t_{3}-t_{4}| \leq 1,  \chi_{1}\chi_{2}=\chi_{3}\chi_{4}\}.\]} for $W$. These will be developed in the following three sections.
	
	\section{Summing over affine transformations with GCD twists}\label{sumoaffgcd}
	For parameters satisfying $0<M \leq (qT)^{\varpi}$ with $0<\varpi <1$ and given a sequence of compactly supported smooth functions $\overrightarrow{f}=\{f_{b}\}_{\substack{1 \leq b \leq q \\ (b, q)=1}}$, we define 
	\[J(\overrightarrow{f}):= \sup_{\substack{0<M_{i}< M \\ 1 \leq i \leq 3}} \int_{\mathbb{R}}\sum_{\substack{1 \leq a \leq q \\ (a, q)=1}}
	\bigg(\sum_{\substack{1 \leq b \leq q \\ (b, q)=1}}\sum_{\substack{|m_{1}|\sim M_{1} \\ m_{2}\sim M_{2} \\ |m_{3}|\ll M_{3}}}f_{b}\bigg(\frac{m_{1}u+m_{3}}{m_{2}}\bigg)(am_{1}+bm_{2}+m_{3}, q)\bigg)^{2} \:\mathrm{d}u,\]
	which is an average of sums of affine transformations with GCD twists. The objective of this section is to establish
	\begin{proposition}\label{bsoat}
		Suppose that $\overrightarrow{f}=\{f_{b}\}_{\substack{1 \leq b \leq q \\ (b, q)=1}}$ is a sequence of smooth functions such that each $f_{b}(u)$ is non-negative and supported on $|u|\ll 1$ and $\sum_{\substack{1 \leq b \leq q \\ (b, q)=1}}|\hat{f}_{b}(\xi)| \lessapprox_{j} (T/|\xi|)^{j}\phi(q)\sup_{b,u}f_{b}(u) $ for all $j \in \mathbb{N}$. Then 
		\[J(\overrightarrow{f})\lessapprox \phi(q) M^{6}\bigg(\int_{\mathbb{R}}\sum_{\substack{1 \leq b \leq q \\ (b, q)=1}}f_{b}(u)\:\mathrm{d}u\bigg)^{2}+
		\phi(q)^{2} M^{4}\int_{\mathbb{R}}\sum_{\substack{1 \leq b \leq q \\ (b, q)=1}}f_{b}(u)^{2}\:\mathrm{d}u.\]
	\end{proposition}
	
	\subsection{\texorpdfstring{Reduction to an iterative bound for $J(f)$}
		{Reduction to an iterative bound for J(f)}}\label{redintb}

	Proposition \ref{bsoat} plays a crucial role in estimating $S_3$ and will be proved using an induction argument, supported by the following iterative bound for $J(\overrightarrow{f})$.
	\begin{lemma}\label{IterabJf}
		Let $\overrightarrow{f}=\{f_{b}\}$ be as in Proposition \ref{bsoat}. Then there is a bump function $\psi(x)$ supported on $|x| \ll 1$ such that
		\[J(\overrightarrow{f})\lessapprox \phi(q) M^{6}\bigg(\int_{\mathbb{R}}\sum_{\substack{1 \leq b \leq q \\ (b, q)=1}}f_{b}(u)\:\mathrm{d}u\bigg)^{2}+
		\phi(q)\bigg(M^{4}\int_{\mathbb{R}}\sum_{\substack{1 \leq b \leq q \\ (b, q)=1}}f_{b}(u)^{2}\:\mathrm{d}u\bigg)^{1/2}J(\overrightarrow{\tilde{f}})^{1/2},\]
		where $\overrightarrow{\tilde{f}}=\{\tilde{f}_{b}\}$ and $\tilde{f}_{b}(u)=\int_{\mathbb{R}}T\psi(T(u-u'))f_{b}(u')\:\mathrm{d}u'$.
	\end{lemma}
	\vspace{1pt}
	\begin{proof}[Proof of Proposition \ref{bsoat} assuming Lemma \ref{IterabJf}]
		We wish to show that for any $\epsilon > 0$ there is a $C(\epsilon)>0$ such that
		\begin{equation}\label{goalintf}
			J(\overrightarrow{f})\leq C(\epsilon)(qT)^{\epsilon}  \bigg(\phi(q) M^{6}\bigg(\int_{\mathbb{R}}\sum_{\substack{1 \leq b \leq q \\ (b, q)=1}}f_{b}(u)\:\mathrm{d}u\bigg)^{2}+
			\phi(q)^{2} M^{4}\int_{\mathbb{R}}\sum_{\substack{1 \leq b \leq q \\ (b, q)=1}}f_{b}(u)^{2}\:\mathrm{d}u\bigg).
		\end{equation}
		We prove the statement via downward induction on $\epsilon$. The base case $\epsilon=100$ holds trivially. For the inductive step, we assume \eqref{goalintf} holds with parameter $3\epsilon/2$ and aim to establish it for $\epsilon$. Application of Lemma \ref{IterabJf} gives 
		\[J(\overrightarrow{f})\lessapprox \phi(q) M^{6}\bigg(\int_{\mathbb{R}}\sum_{\substack{1 \leq b \leq q \\ (b, q)=1}}f_{b}(u)\:\mathrm{d}u\bigg)^{2}+
		\phi(q)\bigg(M^{4}\int_{\mathbb{R}}\sum_{\substack{1 \leq b \leq q \\ (b, q)=1}}f_{b}(u)^{2}\:\mathrm{d}u\bigg)^{1/2}J(\overrightarrow{\tilde{f}})^{1/2}.\]
		The compact support conditions $|u|\ll 1$ for $f_{b}(u)$ and $|x| \ll 1$ for $\psi(x)$ ensure that the modified function $\tilde{f}_{b}(u)$ maintains the same support restriction $|u|\ll 1$.  Moreover, the Fourier transform of $\tilde{f}_{b}$ also satisfies the desired decay. Thus by the induction hypothesis, we have that
		\[J(\overrightarrow{\tilde{f}})\ll_{\epsilon}(qT)^{3\epsilon/2} \bigg(\phi(q) M^{6}\bigg(\int_{\mathbb{R}}\sum_{\substack{1 \leq b \leq q \\ (b, q)=1}}\tilde{f}_{b}(u)\:\mathrm{d}u\bigg)^{2}+
		\phi(q)^{2} M^{4}\int_{\mathbb{R}}\sum_{\substack{1 \leq b \leq q \\ (b, q)=1}}\tilde{f}_{b}(u)^{2}\:\mathrm{d}u\bigg).\]
		Since $\tilde{f}_{b}$ is a smoothed version of $f_{b}$, we have the $L^{1}$ and $L^{2}$ norm comparisons: $\int \tilde{f}_{b}(u)\:\mathrm{d}u \ll \int f_{b}(u)\:\mathrm{d}u$ and $\int \tilde{f}_{b}(u)^{2}\:\mathrm{d}u \ll \int f_{b}(u)^{2}\:\mathrm{d}u$. Inserting these bounds into our bound for $J(\overrightarrow{f})$ yields
		\[J(\overrightarrow{f})\lessapprox_{\epsilon} (qT)^{3\epsilon/4}  \bigg(\phi(q) M^{6}\bigg(\int_{\mathbb{R}}\sum_{\substack{1 \leq b \leq q \\ (b, q)=1}}f_{b}(u)\:\mathrm{d}u\bigg)^{2}+
		\phi(q)^{2} M^{4}\int_{\mathbb{R}}\sum_{\substack{1 \leq b \leq q \\ (b, q)=1}}f_{b}(u)^{2}\:\mathrm{d}u\bigg).\]
		Thus \eqref{goalintf} holds for $C(\epsilon)$ sufficiently large, which completes the induction. 
	\end{proof}
	
	It remains to prove Lemma \ref{IterabJf}.
	
	\subsection{Proof of Lemma \ref{IterabJf}}\label{affsub}
	
	\subsubsection{Setup for the argument}\label{affsub1}
	
	\
	
	By rescaling we may assume that $\sup_{b,u}f_{b}(u) \asymp 1$. Let $\psi_{1}(x)$ be a smooth bump function supported on $|x| \ll 1$, such that $\psi_{1}(m_{3}/M_{3})$ majorizes the summation condition $|m_{3}| \ll M_{3}$. Then we can bound the inner double sum in $J(\overrightarrow{f})$ by
	\begin{align*}
		g_{a}(u):= \sum_{\substack{1 \leq b \leq q \\ (b, q)=1}}\sum_{\substack{|m_{1}|\sim M_{1} \\ m_{2}\sim M_{2}}}\sum_{m_{3} \in \mathbb{Z}}\psi_{1}\bigg(\frac{m_{3}}{M_{3}}\bigg)f_{b}\bigg(\frac{m_{1}u+m_{3}}{m_{2}}\bigg)(am_{1}+bm_{2}+m_{3}, q).
	\end{align*}
	By applying Plancherel's theorem for the choice of $M_{1}, M_{2}, M_{3}$ that achieves the supremum, we obtain
	\[J(\overrightarrow{f}) \leq \sum_{\substack{1 \leq a \leq q \\ (a, q)=1}}\int_{\mathbb{R}}|g_{a}(u)|^{2}\:\mathrm{d}u
	= \sum_{\substack{1 \leq a \leq q \\ (a, q)=1}}\int_{\mathbb{R}} |\hat{g}_{a}(\xi)|^{2}\:\mathrm{d}\xi.\]
	This reduces our task to estimating  $\hat{g}_{a}(\xi)$, the Fourier transform of $g_a(u)$. Notice that $\hat{g}_{a}(\xi)$ is equal to
	\[\sum_{\substack{1 \leq b \leq q \\ (b, q)=1}}\sum_{\substack{|m_{1}|\sim M_{1} \\ m_{2}\sim M_{2}}}\int_{\mathbb{R}}\sum_{m_{3} \in \mathbb{Z}}\psi_{1}\bigg(\frac{m_{3}}{M_{3}}\bigg)f_{b}\bigg(\frac{m_{1}u+m_{3}}{m_{2}}\bigg)(am_{1}+bm_{2}+m_{3}, q)e(-\xi u) \:\mathrm{d}u.\]
	By making the change of variables $u=\tilde{u}-\frac{m_{3}}{m_{1}}$, we get
	\begin{align*}
		\hat{g}_{a}(\xi)=\sum_{\substack{1 \leq b \leq q \\ (b, q)=1}}\sum_{\substack{|m_{1}|\sim M_{1} \\ m_{2}\sim M_{2}}}&\bigg(\int_{\mathbb{R}}
		f_{b}\bigg(\frac{m_{1}\tilde{u}}{m_{2}}\bigg)e(-\xi \tilde{u})\:\mathrm{d}\tilde{u}\bigg)
		\\
		&\cdot \bigg(\sum_{m_{3} \in \mathbb{Z}}\psi_{1}\bigg(\frac{m_{3}}{M_{3}}\bigg)(am_{1}+bm_{2}+m_{3}, q)e\bigg(\frac{m_{3}}{m_{1}}\xi\bigg)\bigg).
	\end{align*}
	The integral over $\tilde{u}$ is given by $\frac{m_{2}}{m_{1}}\hat{f}_{b}(\frac{m_{2}}{m_{1}}\xi)$,  and the last sum, according to the Poisson summation formula, is
	\begin{align*}
		&\sum_{d \bmod q}(am_{1}+bm_{2}+d, q)\sum_{m_{3} \equiv d \, (\!\bmod q)}\psi_{1}\bigg(\frac{m_{3}}{M_{3}}\bigg)e\bigg(\frac{m_{3}}{m_{1}}\xi\bigg)
		\\
		=&\sum_{d \bmod q}(am_{1}+bm_{2}+d, q)\sum_{n \in \mathbb{Z}}\frac{M_{3}}{q}e_q(dn)\hat{\psi}_{1}\bigg(\bigg(\frac{n}{q}-\frac{\xi}{m_{1}}\bigg)M_{3}\bigg).
	\end{align*}
	Lastly, we map $d \to d+am_1+bm_2$. We therefore arrive at
	\begin{align*}
		\hat{g}_{a}(\xi)=\sum_{\substack{1 \leq b \leq q \\ (b, q)=1}}\sum_{\substack{|m_{1}|\sim M_{1} \\ m_{2}\sim M_{2}}}&\frac{m_{2}}{m_{1}}\hat{f}_{b}\bigg(\frac{m_{2}}{m_{1}}\xi\bigg) \sum_{d \bmod q}(d, q)\sum_{n \in \mathbb{Z}}\frac{M_{3}}{q} e_q((d-am_1-bm_2)n)\hat{\psi}_{1}\bigg(\bigg(\frac{n}{q}-\frac{\xi}{m_{1}}\bigg)M_{3}\bigg).
	\end{align*}
	We have to estimate $\int_{\mathbb{R}}\sum_{a}|\hat{g}_{a}(\xi)|^{2}\: \mathrm{d} \xi$. If $|\xi|>(qT)^{2}$, because $\sum_{b}|\hat{f}_{b}(\xi)|$ is rapidly decaying for $|\xi|> T$, it follows that $\sum_{b}\hat{f}_{b}(\frac{m_{2}}{m_{1}}\xi)$ is negligible and $|\hat{g}_{a}(\xi)|\ll (qT)^{-101}|\xi|^{-2}$. As a result, we have
	\[\int_{|\xi|> (qT)^{2}}\sum_{\substack{1 \leq a \leq q \\ (a, q)=1}}|\hat{g}_{a}(\xi)|^{2}\: \mathrm{d} \xi=O((qT)^{-100}).\] 
	Thus, the contribution from the region $|\xi| > (qT)^2$ is negligible and our analysis can focus on the range $|\xi|\leq (qT)^{2}$. On the other hand, as $\hat{\psi}_{1}$ is rapidly decaying, $\hat{\psi}_{1}((\frac{n}{q}-\frac{\xi}{m_{1}})M_{3})$ is negligible unless $|\xi - \frac{nm_{1}}{q}|\lessapprox \frac{M_{1}}{M_{3}}$, it follows that
	\begin{align}\label{hgab}
		\hat{g}_{a}(\xi)= M_{3}\sum_{\substack{1 \leq b \leq q \\ (b, q)=1}}&\sum_{\substack{|m_{1}|\sim M_{1} \\ m_{2}\sim M_{2}}}\frac{m_{2}}{m_{1}}\hat{f}_{b}\bigg(\frac{m_{2}}{m_{1}}\xi\bigg)\sum_{d \bmod q}\frac{(d, q)}{q}
		\notag
		\\
		&\cdot \sum_{n: |\xi - \frac{nm_{1}}{q}|\lessapprox \frac{M_{1}}{M_{3}}}e_q((d-am_1-bm_2)n)\hat{\psi}_{1}\bigg(\bigg(\frac{n}{q}-\frac{\xi}{m_{1}}\bigg)M_{3}\bigg)+O((qT)^{-103}).
	\end{align}
	We separate the terms on the right-hand side of \eqref{hgab} where $q \nmid n$ to define $\hat{g}_{a, 1}(\xi)$ and define $\hat{g}_{a, 2}(\xi)$ using the remaining terms. That is,
	\begin{align}
		\hat{g}_{a,1}(\xi)=&M_{3}\sum_{\substack{1 \leq b \leq q \\ (b, q)=1}}\sum_{\substack{|m_{1}|\sim M_{1} \\ m_{2}\sim M_{2}}}\frac{m_{2}}{m_{1}}\hat{f}_{b}\bigg(\frac{m_{2}}{m_{1}}\xi\bigg)\sum_{d \bmod q}\frac{(d, q)}{q}
		\notag
		\\
		&\quad\quad\quad\quad\quad\cdot \sum_{\substack{n: |\xi - \frac{nm_{1}}{q}|\lessapprox \frac{M_{1}}{M_{3}} \\ q\nmid n }}e_q((d-am_1-bm_2)n)\hat{\psi}_{1}\bigg(\bigg(\frac{n}{q}-\frac{\xi}{m_{1}}\bigg)M_{3}\bigg) \label{ga1 definition}
	\end{align}
	and 
	\begin{align}
		\hat{g}_{a,2}(\xi)=&M_3\sum_{\substack{1 \leq b \leq q \\ (b, q)=1}}\sum_{\substack{|m_{1}|\sim M_{1} \\ m_{2}\sim M_{2}}} \frac{m_{2}}{m_{1}}\hat{f}_{b}\bigg(\frac{m_{2}}{m_{1}}\xi\bigg) \sum_{\ell: |\xi - \ell m_{1}|\lessapprox \frac{M_{1}}{M_{3}} }\hat{\psi}_{1}\bigg(\bigg(\ell - \frac{\xi}{m_{1}}\bigg)M_{3}\bigg).  \label{ga2 definition}
	\end{align}
	Note that we have defined $\hat{g}_{a,2}(\xi)$ with no sum over $d$. This is because in $\hat{g}_a(\xi)$, when $q \mid n$, the only dependence on $d$ is $(d,q)$ and we have that $\sum_{d \bmod q} \frac{(d,q)}{q} \lessapprox 1$. Next, we have that
	\begin{align*}
		&\sum_{\substack{1 \leq a \leq q \\ (a, q)=1}}\int_{\mathbb{R}}|\hat{g}_{a}(\xi)|^{2}\: \mathrm{d} \xi\
		\ll\sum_{\substack{1 \leq a \leq q \\ (a, q)=1}}\int_{-(qT)^{2}}^{(qT)^{2}}|\hat{g}_{a,1}(\xi)|^{2}+ |\hat{g}_{a,2}(\xi)|^{2}\: \mathrm{d} \xi+O((qT)^{-100}).
	\end{align*}
	For a small $\eta>0$, we decompose the above integrals as follows:
	\begin{align*}
		\int\limits_{|\xi|\leq (qT)^{2}}\sum_{\substack{1 \leq a \leq q \\ (a, q)=1}}|\hat{g}_{a,2}(\xi)|^{2}\: \mathrm{d} \xi = I+II,
	\end{align*}
	where 
	\begin{align*}
		I:=\int\limits_{|\xi|\leq (qT)^{\eta}\frac{M_{1}}{M_{3}}}\sum_{\substack{1 \leq a \leq q \\ (a, q)=1}}|\hat{g}_{a,2}(\xi)|^{2}\: \mathrm{d} \xi, \quad II:=\int\limits_{(qT)^{\eta}\frac{M_{1}}{M_{3}}\leq |\xi|\leq (qT)^{2}}\sum_{\substack{1 \leq a \leq q \\ (a, q)=1}}|\hat{g}_{a,2}(\xi)|^{2}\: \mathrm{d} \xi
	\end{align*}
	and
	\begin{align*}
		\int\limits_{|\xi|\leq (qT)^{2}}\sum_{\substack{1 \leq a \leq q \\ (a, q)=1}}|\hat{g}_{a,1}(\xi)|^{2}\: \mathrm{d} \xi = I'+II',
	\end{align*}
	where 
	\begin{align*}
		I':= \int\limits_{|\xi|\leq q^{-4}}\sum_{\substack{1 \leq a \leq q \\ (a, q)=1}}|\hat{g}_{a,1}(\xi)|^{2}\: \mathrm{d} \xi, \quad II':= \int\limits_{q^{-4}\leq |\xi|\leq (qT)^{2}}\sum_{\substack{1 \leq a \leq q \\ (a, q)=1}}|\hat{g}_{a,1}(\xi)|^{2}\: \mathrm{d} \xi.
	\end{align*}
	Based on the above, we have that
	\begin{equation}\label{sepJf}
		J(\overrightarrow{f})\ll I+II+I'+II'+O((qT)^{-100}).
	\end{equation}
	Consequently, our objective is reduced to estimating $I$, $II$, $I'$ and $II'$ respectively.
	
	\subsubsection{The contribution of $II'$} \label{affTcII'}
	
	Notice that 
	\begin{align*}
		\hat{g}_{a,1}(\xi)=&M_{3}\sum_{\substack{1 \leq b \leq q \\ (b, q)=1}}\sum_{\substack{|m_{1}|\sim M_{1} \\ m_{2}\sim M_{2}}}\frac{m_{2}}{m_{1}}\hat{f}_{b}\bigg(\frac{m_{2}}{m_{1}}\xi\bigg)\sum_{d \bmod q}\frac{(d, q)}{q}
		\notag
		\\
		&\quad\quad\quad\quad\quad\cdot \sum_{\substack{n: |\xi - \frac{nm_{1}}{q}|\lessapprox \frac{M_{1}}{M_{3}} \\ q\nmid n }}e_q((d-am_1-bm_2)n)\hat{\psi}_{1}\bigg(\bigg(\frac{n}{q}-\frac{\xi}{m_{1}}\bigg)M_{3}\bigg)
		\\
		= &M_{3}\sum_{|m_{1}|\sim M_{1}}\sum_{\substack{n: |\xi - \frac{nm_{1}}{q}|\lessapprox \frac{M_{1}}{M_{3}} \\ q\nmid n }}
		\sum_{d \bmod q} \frac{(d, q)}{q}
		\hat{\psi}_{1}\bigg(\bigg(\frac{n}{q}-\frac{\xi}{m_{1}}\bigg)M_{3}\bigg)
		e_q((d-am_{1})n)
		\\
		&\quad\quad\quad\quad\quad\cdot\sum_{m_{2}\sim M_{2}}
		\sum_{\substack{1 \leq b \leq q \\ (b, q)=1}}\frac{m_{2}}{m_{1}}\hat{f}_{b}\bigg(\frac{m_{2}}{m_{1}}\xi\bigg)
		e_q(-bm_{2}n).
	\end{align*}
	We therefore have
	\begin{align*}
		&\sum_{\substack{1 \leq a \leq q \\ (a, q)=1}}|\hat{g}_{a,1}(\xi)|^{2}=\sum_{\substack{1 \leq a \leq q \\ (a, q)=1}}
		M_{3}^{2}\sum_{|m_{1}|\sim M_{1}}\sum_{\substack{n: |\xi - \frac{nm_{1}}{q}|\lessapprox \frac{M_{1}}{M_{3}} \\ q\nmid n }}
		\sum_{d \bmod q}
		\sum_{|m_{1}'|\sim M_{1}}\sum_{\substack{n': |\xi - \frac{n'm_{1}'}{q}|\lessapprox \frac{M_{1}}{M_{3}} \\ q\nmid n' }}
		\sum_{d' \bmod q}\frac{(d, q)}{q}
		\\
		&\cdot\frac{(d', q)}{q}\hat{\psi}_{1}\bigg(\bigg(\frac{n}{q}-\frac{\xi}{m_{1}}\bigg)M_{3}\bigg)
		\overline{\hat{\psi}}_{1}\bigg(\bigg(\frac{n'}{q}-\frac{\xi}{m_{1}'}\bigg)M_{3}\bigg)
		e_q(dn-d'n')
		e_q((m_{1}'n'-m_{1}n)a)
		\\
		&\bigg(\sum_{m_{2}\sim M_{2}}
		\sum_{\substack{1 \leq b \leq q \\ (b, q)=1}}\frac{m_{2}}{m_{1}}\hat{f}_{b}\bigg(\frac{m_{2}}{m_{1}}\xi\bigg)
		e_q(-bm_{2}n)\bigg)\bigg(\sum_{m_{2}'\sim M_{2}}
		\sum_{\substack{1 \leq b' \leq q \\ (b', q)=1}}\frac{m_{2}'}{m_{1}'}\overline{\hat{f}}_{b}\bigg(\frac{m_{2}'}{m_{1}'}\xi\bigg)
		e_q(bm_{2}'n')\bigg).
	\end{align*}
	For the \(a\)-sum, we apply the basic estimate $|C_q(m)| \leq (m,q)$ for Ramanujan's sum, which gives: $| \sum_{(a,q)=1} e_q((m_1' n' - m_1 n)a) | \leq (m_1' n' - m_1 n, q).$ We then impose the condition \(m_1' n' - m_1 n \equiv h \bmod{q}\), allowing \(h\) to range over a complete residue system modulo \(q\). This reformulation allows us to express \((m_1' n' - m_1 n, q)\) as \((h, q)\) and factor the \(h\)-sum out. Doing this along with a use of the triangle inequality gives us the bound
	\begin{align*}
		&\sum_{\substack{1 \leq a \leq q \\ (a, q)=1}}|\hat{g}_{a,1}(\xi)|^{2}\leq M_{3}^{2}\sum_{\substack{h \bmod q }} (h,q)
		\sum_{|m_{1}|\sim M_{1}}\sum_{\substack{n: |\xi - \frac{nm_{1}}{q}|\lessapprox \frac{M_{1}}{M_{3}} \\ q\nmid n }}
		\sum_{d \bmod q}
		\sum_{|m_{1}'|\sim M_{1}}\sum_{\substack{n': |\xi - \frac{n'm_{1}'}{q}|\lessapprox \frac{M_{1}}{M_{3}}, q\nmid n' \\ m_1'n'-m_1n \equiv h \bmod q }}
		\sum_{d' \bmod q}
		\\
		&\cdot\frac{(d, q)}{q}\frac{(d', q)}{q}\bigg|\hat{\psi}_{1}\bigg(\bigg(\frac{n}{q}-\frac{\xi}{m_{1}}\bigg)M_{3}\bigg)\bigg| \cdot \bigg|
		\overline{\hat{\psi}}_{1}\bigg(\bigg(\frac{n'}{q}-\frac{\xi}{m_{1}'}\bigg)M_{3}\bigg)\bigg|
		\\
		&\bigg|\sum_{m_{2}\sim M_{2}}
		\sum_{\substack{1 \leq b \leq q \\ (b, q)=1}}\frac{m_{2}}{m_{1}}\hat{f}_{b}\bigg(\frac{m_{2}}{m_{1}}\xi\bigg)
		e_q(-bm_{2}n)\bigg|\cdot\bigg|\sum_{m_{2}'\sim M_{2}}
		\sum_{\substack{1 \leq b' \leq q \\ (b', q)=1}}\frac{m_{2}'}{m_{1}'}\overline{\hat{f}}_{b'}\bigg(\frac{m_{2}'}{m_{1}'}\xi\bigg)
		e_q({bm_{2}'n'})\bigg|.
	\end{align*}
	After the use of the triangle inequality, the only terms depending on $d,d'$ are $(d,q),(d',q)$ and we have that $\sum_{d \bmod q} \frac{(d,q)}{q} \lessapprox 1$ and the similarly for $d'$. Next, by Cauchy-Schwarz we have that $\sum_{\substack{1 \leq a \leq q \\ (a, q)=1}}|\hat{g}_{a,1}(\xi)|^{2}$ is bounded by $M_3^2G_1(\xi)^{1/2}G_2(\xi)^{1/2}$, where 
	\begin{align*}
		G_1(\xi):= \sum_{h \bmod q}(h, q)
		\sum_{|m_{1}|\sim M_{1}}\sum_{\substack{n: |\xi - \frac{nm_{1}}{q}|\lessapprox \frac{M_{1}}{M_{3}} \\ q\nmid n }}
		&\sum_{|m_{1}'|\sim M_{1}}
		\sum_{\substack{n': |\xi - \frac{n'm_{1}'}{q}|\lessapprox \frac{M_{1}}{M_{3}}, \, q\nmid n' \\  m_{1}'n'-m_{1}n\equiv h (\!\bmod q) }}\bigg|\hat{\psi}_{1}\bigg(\bigg(\frac{n}{q}-\frac{\xi}{m_{1}}\bigg)M_{3}\bigg)\bigg|^{2}
		\\
		&\cdot \Biggl|\sum_{m_{2}\sim M_{2}}
		\sum_{\substack{1 \leq b \leq q \\ (b, q)=1}}\frac{m_{2}}{m_{1}}\hat{f}_{b}\bigg(\frac{m_{2}}{m_{1}}\xi\bigg)
		e_q({-bm_{2}n})\Biggl|^{2}
	\end{align*}
	and 
	\begin{align*}
		G_2(\xi):= \sum_{h \bmod q}(h, q)
		\sum_{|m_{1}|\sim M_{1}}\sum_{\substack{n: |\xi - \frac{nm_{1}}{q}|\lessapprox \frac{M_{1}}{M_{3}} \\ q\nmid n }}
		&\sum_{|m_{1}'|\sim M_{1}}
		\sum_{\substack{n': |\xi - \frac{n'm_{1}'}{q}|\lessapprox \frac{M_{1}}{M_{3}}, \, q\nmid n' \\  m_{1}'n'-m_{1}n\equiv h (\!\bmod q) }}\bigg|\hat{\psi}_{1}\bigg(\bigg(\frac{n'}{q}-\frac{\xi}{m_{1}'}\bigg)M_{3}\bigg)\bigg|^{2}
		\\
		&\cdot\Biggl|\sum_{m_{2}'\sim M_{2}}
		\sum_{\substack{1 \leq b' \leq q \\ (b', q)=1}}\frac{m_{2}'}{m_{1}'}\overline{\hat{f}}_{b'}\bigg(\frac{m_{2}'}{m_{1}'}\xi\bigg)
		e_q({bm_{2}'n'})\Biggl|^{2}
	\end{align*}
	Note that since $|z|=|\bar{z}|$, upon mapping $(m_1,n,m_1',n') \to (-m_1',n',-m_1,n)$, we have $G_2(\xi)= G_1(-\xi)$. By the Cauchy-Schwarz inequality again, we arrive at
	\begin{align*}
		\int\limits_{q^{-4} \leq |\xi| \leq (qT)^2}\sum_{\substack{1 \leq a \leq q \\ (a, q)=1}}|\hat{g}_{a,1}(\xi)|^{2}\: \mathrm{d} \xi &\leq \bigg(M_{3}^{2}\int_{\mathbb{R}}G_{1}(\xi)\: \mathrm{d} \xi\bigg)^{1/2}\bigg(M_{3}^{2}\int_{\mathbb{R}}G_{2}(\xi)\: \mathrm{d} \xi\bigg)^{1/2}
	\end{align*}
	so that
	\begin{align}\label{causII'}
		II' \leq M_3^2 \int G_1(\xi) \mathrm{d} \xi.
	\end{align}
	We now estimate $M_{3}^{2}\int_{\mathbb{R}}G_{1}(\xi)\: \mathrm{d} \xi$. Notice that $|\hat{\psi}|\ll 1$. Moreover, the double sum
	$\sum_{m_1'}\sum_{n'}1$ is bounded by $\lessapprox 1+M_1/M_3$.
	To see this, write $s=m_1'n'$. Then $s$ is a non-zero integer lying in a fixed
	residue class modulo $q$, and it is contained in a $\lessapprox qM_1/M_3$
	neighborhood of $q\xi$. Thus there are at most $\lessapprox 1+M_1/M_3$
	possibilities for $s$, giving the claim. It follows that $M_3^2 \int_{\mathbb{R}}G_1(\xi) \mathrm{d} \xi$ is bounded by 
	\refstepcounter{equation}\label{exprintG}
	\begin{flalign*}
		&\lessapprox
		M_{3}\int_{\mathbb{R}}
		\sum_{h \bmod q}(h,q)
		\sum_{|m_{1}|\sim M_{1}}
		\sum_{\substack{
				n:\,|\xi-\frac{nm_{1}}{q}|
				\lessapprox\frac{M_{1}}{M_{3}}\\
				q\nmid n
		}}
		(M_{1}+M_{3})
		\bigg|
		\sum_{m_{2}\sim M_{2}}
		\sum_{\substack{
				1\leq b\leq q\\
				(b,q)=1
		}}
		\frac{m_{2}}{m_{1}}
		\hat{f}_{b}\bigg(\frac{m_{2}}{m_{1}}\xi\bigg)
		e_q({-bm_{2}n})
		\bigg|^{2}
		\,\mathrm{d}\xi
		&&
		\\[0.5ex]
		&\lessapprox
		qM_{3}(M_{1}+M_{3})
		\sum_{|m_{1}|\sim M_{1}}
		\sum_{n\in\mathbb{Z}}
		\int\limits_{
			|\xi-\frac{nm_{1}}{q}|
			\lessapprox\frac{M_{1}}{M_{3}}
		}
		\bigg|
		\sum_{m_{2}\sim M_{2}}
		\sum_{\substack{
				1\leq b\leq q\\
				(b,q)=1
		}}
		\frac{m_{2}}{m_{1}}
		\hat{f}_{b}\bigg(\frac{m_{2}}{m_{1}}\xi\bigg)
		e_q({-bm_{2}n})
		\bigg|^{2}
		\,\mathrm{d}\xi,
		\hspace{0.5em}\textup{(\theequation)}
		&&
	\end{flalign*}
	since $\sum_{h \bmod q} (h,q) \lessapprox q$. By making a change of variables $\xi =\frac{m_{1}n}{q}+\frac{m_{1}}{M_{3}}\tau$, the only $m_1$ dependence in the inner double sum is a $1/m_1$ term. Then the above expression is bounded by
	\begin{align*}
		&
		\lessapprox q(M_{1}+M_{3})
		\sum_{|m_{1}|\sim M_{1}}\sum_{n \in \mathbb{Z}}
		\int_{|\tau|\lessapprox 1}
		\biggl|\sum_{m_{2}\sim M_{2}}
		\sum_{\substack{1 \leq b \leq q \\ (b, q)=1}}\frac{m_{2}}{m_{1}}\hat{f}_{b}\bigg(\frac{m_{2}n}{q}+\frac{m_{2}}{M_{3}}\tau\bigg)
		e_q({-bm_{2}n})\biggl|^{2}m_{1}\: \mathrm{d} \tau
		\\
		&\lessapprox
		q(M_{1}+M_{3})
		\sum_{n \in \mathbb{Z}}
		\int_{|\tau|\lessapprox 1}
		\biggl|\sum_{m_{2}\sim M_{2}}
		\sum_{\substack{1 \leq b \leq q \\ (b, q)=1}}m_{2}\hat{f}_{b}\bigg(\frac{m_{2}n}{q}+\frac{m_{2}}{M_{3}}\tau\bigg)
		e_q({-bm_{2}n})\biggl|^{2}\: \mathrm{d} \tau.
	\end{align*}
	Since $\sum_{b}|\hat{f}_{b}(\xi)|=O((qT)^{-200})$ unless $|\xi|\lessapprox T$, we can restrict\footnote{The selection of $M_{3}$ must satisfy the lower bound $M_{3}\gg M$ as dictated by the definition of $J(\overrightarrow{f})$ as the larger $M_3$ is, the greater the inner double sum in $J(\overrightarrow{f})$ is.} the sum over $n$ 
	to the range $|n| \lessapprox qT/M_{2}$ at the cost of a negligible error. We introduce\footnote{One might choose $\psi_{2}(x)=\psi_{0}(x/(qT)^{\eta})$, where $\psi_{0}(x)$ is a smooth bump supported on $[-2, 2]$ and identically equal to 1 for $|x|\leq 1$.} a bump function $\psi_{2}$ supported on $|x|\lessapprox 1$ so that $\psi_{2}(\frac{M_{2}n}{qT})$ majorises this summation condition and obtain
	\begin{align}
		M_{3}^{2}\int_{\mathbb{R}}G_{1}(\xi)\: \mathrm{d} \xi \lessapprox q(M_{1}+M_{3}) \varSigma_{II}+O((qT)^{-100}), \label{Sigma II equation}
	\end{align}
	where
	\begin{align*}
		\varSigma_{II} := \sum_{n \in \mathbb{Z}}\psi_{2}\bigg(\frac{M_{2}n}{qT}\bigg)
		\int \limits_{|\tau|\lessapprox 1}
		\biggl|\sum_{m_{2}\sim M_{2}}
		\sum_{\substack{1 \leq b \leq q \\ (b, q)=1}}m_{2}\hat{f}_{b}\bigg(\frac{m_{2}n}{q}+\frac{m_{2}}{M_{3}}\tau\bigg)
		e_q({-bm_{2}n})\biggl|^{2}\: \mathrm{d} \tau.
	\end{align*}
	We write out $\hat{f}_{b}$ as an integral, expand out the square and bring the summation over $n$ and integration over $\tau$ on the inside to get
	\begin{align*}
		\varSigma_{II}=\int \int
		\sum_{m_{2}\sim M_{2}}\sum_{m_{2}'\sim M_{2}}
		\sum_{\substack{1 \leq b \leq q \\ (b, q)=1}}
		\sum_{\substack{1 \leq b' \leq q \\ (b', q)=1}}
		m_{2}m_{2}'f_{b}(u)f_{b'}(u')
		\varSigma_{1}\varSigma_{2}\:\mathrm{d} u'\:\mathrm{d} u,
	\end{align*}
	where
	\[\varSigma_{1}=\int_{|\tau| \lessapprox 1} e\bigg(\bigg(\frac{m_{2}'}{M_{3}}u'-\frac{m_{2}}{M_{3}}u\bigg)\tau\bigg)\:\mathrm{d} \tau,\]
	\[\varSigma_{2}=\sum_{n \in \mathbb{Z}}\psi_{2}\bigg(\frac{M_{2}n}{qT}\bigg)
	e_q({(m_{2}'u'-m_2u+b'm_2'-bm_2)n}).\]
	Trivially we have $\varSigma_{1} \lessapprox 1$. By Poisson summation and the rapid decay of the Fourier transform $\hat{\psi_{2}}$, we have 
	\begin{align*}
		\varSigma_{2}&=\frac{qT}{M_{2}}\sum_{j \in \mathbb{Z}}\hat{\psi_{2}}\bigg(\frac{j-\frac{m_{2}'u'-m_{2}u+b'm_{2}'-bm_{2}}{q}}{M_{2}/qT}\bigg)
		\\
		&=\frac{qT}{M_{2}}\sum_{\big|j-\frac{m_{2}'u'-m_{2}u+b'm_{2}'-bm_{2}}{q}\big|\ll \frac{M_{2}}{qT}}\hat{\psi_{2}}\bigg(\frac{j-\frac{m_{2}'u'-m_{2}u+b'm_{2}'-bm_{2}}{q}}{M_{2}/qT}\bigg)+O((qT)^{-100}).
	\end{align*}
	Inserting this back into our expression for $\varSigma_{II}$, we find
	\begin{align*}
		\varSigma_{II} \lessapprox qTM_{2}
		\int
		\sum_{\substack{1 \leq b \leq q \\ (b, q)=1}}
		\sum_{\substack{1 \leq b' \leq q \\ (b', q)=1}}
		f_{b}(u)
		\sum_{m_{2}\sim M_{2}}\sum_{m_{2}'\sim M_{2}}
		\sum_{j \in \mathbb{Z}}
		\int\limits_{\big|u'-\frac{qj+m_{2}u+bm_{2}-b'm_{2}'}{m_{2}'}\big|\ll \frac{1}{T}} f_{b'}(u')
		\:\mathrm{d} u'\:\mathrm{d} u
		\\
		+O((qT)^{-50}).
	\end{align*}
	Putting this bound for $\Sigma_{II}$ into \eqref{Sigma II equation}, we get
	\begin{align*}
		M_{3}^{2}\int_{\mathbb{R}}G_{1}(\xi)\: \mathrm{d} \xi \lessapprox& \
		q^{2}TM_{2}(M_{1}+M_{3})
		\int
		\sum_{\substack{1 \leq b \leq q \\ (b, q)=1}}
		\sum_{\substack{1 \leq b' \leq q \\ (b', q)=1}}
		f_{b}(u)
		\\
		&\sum_{m_{2}\sim M_{2}}\sum_{m_{2}'\sim M_{2}}
		\sum_{j \in \mathbb{Z}}
		\int\limits_{\big|u'-\frac{qj+m_{2}u+bm_{2}-b'm_{2}'}{m_{2}'}\big|\ll \frac{1}{T}} f_{b'}(u')
		\:\mathrm{d} u'\:\mathrm{d} u+O((qT)^{-20}).
	\end{align*}
	Let $\psi(x)$ be a smooth bump function supported on $|x| \ll 1$ and define
	\[\tilde{f}_{b'}(u):=\int_{\mathbb{R}}T\psi(T(u-v))f_{b'}(v)\:\mathrm{d}v.\] 
	We can then bound the above integral over $u'$ by $\tilde{f_{b'}}(\frac{qj+m_{2}u+bm_{2}-b'm_{2}'}{m_{2}'})$. Thus, combining the bounds $M_{1}, M_{2}, M_{3} \leq M$ with the Cauchy-Schwarz inequality and observing that $\tilde{f}_{b'}(u)$ is supported on $|u|\ll 1$, we have the following bound for $M_3^2 \int_{\mathbb{R}} G_1(\xi) \mathrm{d} \xi$,
	\begin{align*}
		&\lessapprox q^{2}M^{2}
		\bigg(\int\sum_{\substack{1 \leq b \leq q \\ (b, q)=1}}
		\bigg(\sum_{\substack{1 \leq b' \leq q \\ (b', q)=1}}
		\sum_{\substack{m_{2},m_2'\sim M_{2} \\ j \in \mathbb{Z}}} \tilde{f_{b'}}\bigg(\frac{qj+m_{2}u+bm_{2}-b'm_{2}'}{m_{2}'}\bigg)\bigg)^{2} \: \mathrm{d}u\bigg)^{\frac{1}{2}} \\
		& \phantom{abcdefghi} \cdot \bigg(\int\sum_{\substack{1 \leq b \leq q \\ (b, q)=1}}
		f_{b}(u)^{2}\: \mathrm{d}u\bigg)^{\frac{1}{2}}+O((qT)^{-20})
		\\
		& \lessapprox qM^2 \bigg(\int\sum_{\substack{1 \leq b \leq q \\ (b, q)=1}}
		\bigg(\sum_{\substack{1 \leq b' \leq q \\ (b', q)=1}}
		\sum_{\substack{m_{2}, m_{2}'\sim M_{2} \\ m_{3} \equiv bm_{2}-b'm_{2}' \bmod q}}
		\tilde{f_{b'}}\bigg(\frac{m_{2}u+m_{3}}{m_{2}'}\bigg)(-bm_{2}+b'm_{2}'+m_{3}, q)\bigg)^{2} \mathrm{d}u\bigg)^{\frac{1}{2}} \\
		& \phantom{abcdefg} \cdot \bigg(\int\sum_{\substack{1 \leq b \leq q \\ (b, q)=1}}
		f_{b}(u)^{2}\: \mathrm{d}u\bigg)^{\frac{1}{2}} +  O((qT)^{-20}).
	\end{align*}
	In the final line we have introduced a sum over $m_3$ such that $m_3 = qj + bm_2 - b'm_2'$ for some integer $j$. Since our sum over $j$ is over all integers, this is the same as summing over all $m_3$ such that $m_3 \equiv bm_2 - b'm_2' \bmod q$. We also absorb a factor of $q$ to bring in the gcd factor since given this modular constraint on $m_3$, we have $q=(-bm_2+b'm_2'+m_3,q)$. Continuing, since $\tilde{f_b}$ is supported on $|u| \ll 1$, we have that $|m_3| \ll M_2$. Therefore the above is bounded by 
	\begin{align*}
		&\lessapprox qM^2\bigg(\int\sum_{\substack{1 \leq b \leq q \\ (b, q)=1}}
		\bigg(\sum_{\substack{1 \leq b' \leq q \\ (b', q)=1}}
		\sum_{\substack{m_{2}, m_2'\sim M_{2} \\ |m_3| \ll M_2}} \tilde{f_{b'}}\bigg(\frac{m_{2}u+m_{3}}{m_{2}'}\bigg)(-bm_{2}+b'm_{2}'+m_{3}, q)\bigg)^{2} \mathrm{d}u\bigg)^{\frac{1}{2}} 
		\\
		&  \phantom{abcdefg} \cdot\bigg(\int\sum_{\substack{1 \leq b \leq q \\ (b, q)=1}}
		f_{b}(u)^{2}\: \mathrm{d}u\bigg)^{\frac{1}{2}} +   O((qT)^{-20}).
		\\
		&\lessapprox \phi(q) \bigg(M^{4}\int_{\mathbb{R}}\sum_{\substack{1 \leq b \leq q \\ (b, q)=1}}f_{b}(u)^{2}\:\mathrm{d}u\bigg)^{1/2}J(\overrightarrow{\tilde{f}})^{1/2}.
	\end{align*}
	Combining these estimates with \eqref{causII'}, we conclude that
	\begin{equation}\label{conII'}
		II' \lessapprox \phi(q) \bigg(M^{4}\int_{\mathbb{R}}\sum_{\substack{1 \leq b \leq q \\ (b, q)=1}}f_{b}(u)^{2}\:\mathrm{d}u\bigg)^{1/2}J(\overrightarrow{\tilde{f}})^{1/2}.
	\end{equation}
	\subsubsection{The contribution of $II$.}\label{affTcII}
	
	The estimate for the term $II$ follows from adapting the method in the preceding subsection with minor adjustments. By the definition of $\hat{g}_{a,2}(\xi)$ (cf. \eqref{ga2 definition}), we have
	\begin{align*}
		\hat{g}_{a,2}(\xi)= M_3\sum_{|m_1| \sim M_1} \sum_{\ell: |\xi - \ell m_{1}|\lessapprox \frac{M_{1}}{M_{3}} }\hat{\psi}_{1}\bigg(\bigg(\ell - \frac{\xi}{m_{1}}\bigg)M_{3}\bigg) \sum_{\substack{m_{2}\sim M_{2}}} \sum_{\substack{1 \leq b \leq q \\ (b, q)=1}} \frac{m_{2}}{m_{1}}\hat{f}_{b}\bigg(\frac{m_{2}}{m_{1}}\xi\bigg).
	\end{align*}
	There is no dependence on $a$ in $\hat{g}_{a,2}$. Therefore
	\begin{align*}
		\sum_{\substack{1 \leq a \leq q \\ (a, q)=1}}|\hat{g}_{a,2}(\xi)|^{2}&\lessapprox 
		q
		M_{3}^{2}\sum_{\substack{|m_{1}|, |m_1'|\sim M_{1}}}\sum_{\substack{\ell: |\xi - \ell m_{1}|\lessapprox \frac{M_{1}}{M_{3}} \\ \ell': |\xi - \ell' m_{1}'|\lessapprox \frac{M_{1}}{M_{3}} }}\hat{\psi}_{1}\bigg(\bigg(\ell-\frac{\xi}{m_{1}}\bigg)M_{3}\bigg)
		\overline{\hat{\psi}}_{1}\bigg(\bigg(\ell'-\frac{\xi}{m_{1}'}\bigg)M_{3}\bigg)
		\\
		&\quad\cdot\bigg(\sum_{m_{2}\sim M_{2}}
		\sum_{\substack{1 \leq b \leq q \\ (b, q)=1}}\frac{m_{2}}{m_{1}}\hat{f}_{b}\bigg(\frac{m_{2}}{m_{1}}\xi\bigg)
		\bigg)\bigg(\sum_{m_{2}'\sim M_{2}}
		\sum_{\substack{1 \leq b' \leq q \\ (b', q)=1}}\frac{m_{2}'}{m_{1}'}\overline{\hat{f}}_{b'}\bigg(\frac{m_{2}'}{m_{1}'}\xi\bigg)\bigg).
	\end{align*}
	Next, we apply Cauchy-Schwarz to find that $\sum_{\substack{1 \leq a \leq q \\ (a, q)=1}}|\hat{g}_{a,2}(\xi)|^{2}$ is bounded by $M_3^2H_1(\xi)^{1/2}H_2(\xi)^{1/2}$, where 
	\begin{align*}
		H_1(\xi):= q \sum_{\substack{|m_1|,|m_1'|\sim M_1}} \sum_{\substack{\ell : |\xi-lm_1| \lessapprox \frac{M_1}{M_3} \\ \ell':|\xi-l'm_1'| \lessapprox \frac{M_1}{M_3}}} \bigg|\hat{\psi}_{1}\bigg(\bigg(\ell-\frac{\xi}{m_{1}}\bigg)M_{3}\bigg) \bigg|^2 \bigg|\sum_{m_2 \sim M_2}\sum_{\substack{1 \leq b \leq q \\ (b,q)=1}} \frac{m_2}{m_1} \hat{f}_b \bigg(\frac{m_2}{m_1}\xi\bigg)\bigg|^2 
	\end{align*}
	and
	\begin{align*}
		H_2(\xi):= q \sum_{\substack{|m_1|,|m_1'|\sim M_1}} \sum_{\substack{\ell : |\xi-lm_1| \lessapprox \frac{M_1}{M_3} \\ \ell':|\xi-l'm_1'| \lessapprox \frac{M_1}{M_3}}} \bigg|\hat{\psi}_{1}\bigg(\bigg(\ell'-\frac{\xi}{m_{1}'}\bigg)M_{3}\bigg)\bigg|^2 \bigg|\sum_{m_2' \sim M_2}\sum_{\substack{1 \leq b \leq q \\ (b,q)=1}} \frac{m_2'}{m_1'} \overline{\hat{f}}_{b'} \bigg(\frac{m_2'}{m_1'}\xi\bigg)\bigg|^2 .
	\end{align*}
	Note that under the change of variables $(m_1,\ell,m_1',\ell') \to (-m_1',\ell',-m_1,\ell)$ we have that $H_2(\xi)=H_1(-\xi)$. Also, we have that $|\hat{\psi}| \ll 1$ and that the double sum $\sum_{m_1'}\sum_{\ell'} 1$ is bounded by $\lessapprox 1+M_1/M_3$ so 
	\begin{align*}
		H_1(\xi) \leq q\bigg(1+ \frac{M_1}{M_3}\bigg)\sum_{\substack{|m_1|\sim M_1}} \sum_{\substack{\ell : |\xi-lm_1| \lessapprox \frac{M_1}{M_3} \\ }} \bigg|\sum_{m_2 \sim M_2}\sum_{\substack{1 \leq b \leq q \\ (b,q)=1}} \frac{m_2}{m_1} \hat{f}_b \bigg(\frac{m_2}{m_1}\xi\bigg)\bigg|^2.
	\end{align*}
	Now by the Cauchy-Schwarz inequality again, we have that 
	\begin{align*}
		\int_{(qT)^{\eta}\frac{M_{1}}{M_{3}}}^{(qT)^{2}}\sum_{\substack{1 \leq a \leq q \\ (a, q)=1}}|\hat{g}_{a,2}(\xi)|^{2}\: \mathrm{d} \xi\leq \bigg(M_{3}^{2}\int_{\mathbb{R}}H_{1}(\xi)\: \mathrm{d} \xi\bigg)^{1/2}\bigg(M_{3}^{2}\int_{\mathbb{R}}H_{2}(\xi)\: \mathrm{d} \xi\bigg)^{1/2} = M_3^2 \int_{\mathbb{R}}H_1(\xi) \mathrm{d} \xi.
	\end{align*}
	Therefore 
	\begin{align}\label{causII}
		II \leq M_3^2 \int_{\mathbb{R}} H_1(\xi) \mathrm{d}\xi.
	\end{align}
	
	We now show that $M_{3}^{2}\int_{\mathbb{R}}H_{1}(\xi)\: \mathrm{d} \xi$ can be written in a form similar to the bound \eqref{exprintG} for $M_3^2 \int_{\mathbb{R}} G_1(\xi) \mathrm{d} \xi$ from the previous section, by reintroducing $n=lq$. We have that $M_3^2 \int_{\mathbb{R}}H_1(\xi) \mathrm{d} \xi$ is bounded by 
	\begin{align*}
		&\lessapprox  qM_3(M_1+M_3)\int_{\mathbb{R}} \sum_{\substack{|m_1|\sim M_1}} \sum_{\substack{\ell : |\xi-lm_1| \lessapprox \frac{M_1}{M_3} \\ }} \bigg|\sum_{m_2 \sim M_2}\sum_{\substack{1 \leq b \leq q \\ (b,q)=1}} \frac{m_2}{m_1} \hat{f}_b \bigg(\frac{m_2}{m_1}\xi\bigg)\bigg|^2 \mathrm{d}\xi.
		\\
		&=   qM_3(M_1+M_3)\int_{\mathbb{R}} \sum_{\substack{|m_1|\sim M_1}} \sum_{\substack{n : |\xi-\frac{nm_1}{q}| \lessapprox \frac{M_1}{M_3} \\ q \mid n }} \bigg|\sum_{m_2 \sim M_2}\sum_{\substack{1 \leq b \leq q \\ (b,q)=1}} \frac{m_2}{m_1} \hat{f}_b \bigg(\frac{m_2}{m_1}\xi\bigg) e_q({-bm_{2}n})\bigg|^2 \mathrm{d}\xi
		\\
		&\lessapprox
		qM_{3}(M_{1}+M_{3})
		\sum_{|m_{1}|\sim M_{1}}\sum_{n \in \mathbb{Z}}
		\int\limits_{|\xi - \frac{nm_{1}}{q}|\lessapprox \frac{M_{1}}{M_{3}}}
		\Biggl|\sum_{m_{2}\sim M_{2}}
		\sum_{\substack{1 \leq b \leq q \\ (b, q)=1}}\frac{m_{2}}{m_{1}}\hat{f}_{b}\bigg(\frac{m_{2}}{m_{1}}\xi\bigg)
		e_q({-bm_{2}n})\Biggl|^{2} \mathrm{d} \xi.
	\end{align*}
	The final expression coincides precisely with the integral representation $M_3^2 \int_{\mathbb{R}} G_1(\xi) d\xi$ established in the previous subsection (see \eqref{exprintG}). Therefore the same approach yields the estimate:
	\begin{align*}
		M_{3}^{2}\int_{\mathbb{R}}H_{1}(\xi)\: \mathrm{d} \xi 
		\lessapprox \phi(q) \bigg(M^{4}\int_{\mathbb{R}}\sum_{\substack{1 \leq b \leq q \\ (b, q)=1}}f_{b}(u)^{2}\:\mathrm{d}u\bigg)^{1/2}J(\overrightarrow{\tilde{f}})^{1/2}.
	\end{align*}
	Thus by \eqref{causII}, we conclude that
	\begin{equation}\label{conII}
		II \lessapprox \phi(q) \bigg(M^{4}\int_{\mathbb{R}}\sum_{\substack{1 \leq b \leq q \\ (b, q)=1}}f_{b}(u)^{2}\:\mathrm{d}u\bigg)^{1/2}J(\overrightarrow{\tilde{f}})^{1/2}.
	\end{equation}
	
	\subsubsection{The contribution of $I'$.}\label{affTcI'}
	
	Recalling the definition of $\hat{g}_{a,1}(\xi)$,
	\begin{align*}
		\hat{g}_{a,1}(\xi)= M_{3}\sum_{\substack{1 \leq b \leq q \\ (b, q)=1}}&\sum_{\substack{|m_{1}|\sim M_{1} \\ m_{2}\sim M_{2}}}\frac{m_{2}}{m_{1}}\hat{f}_{b}\bigg(\frac{m_{2}}{m_{1}}\xi\bigg)\sum_{d \bmod q}\frac{(d, q)}{q}
		\notag
		\\
		&\quad\quad\quad\quad\quad\cdot \sum_{\substack{n: |\xi - \frac{nm_{1}}{q}|\lessapprox \frac{M_{1}}{M_{3}} \\ q\nmid n }}e_q({(d-am_{1}-bm_{2})}n)\hat{\psi}_{1}\bigg(\bigg(\frac{n}{q}-\frac{\xi}{m_{1}}\bigg)M_{3}\bigg)
	\end{align*}
	Since $|\hat{\psi}| \ll 1$, we have by the triangle inequality that
	\begin{align*}
		|\hat{g}_{a,1}(\xi)| \ll  M_3 \sum_{\substack{1 \leq b \leq q \\ (b,q)=1}} \sum_{\substack{n: |\xi - \frac{nm_{1}}{q}|\lessapprox \frac{M_{1}}{M_{3}} \\ q\nmid n }} \sum_{\substack{|m_1| \sim M_1 \\ m_2 \sim M_2}} \frac{m_2}{m_1} \sup_{\xi} |\hat{f}_b(\xi)| \sum_{d \bmod q} \frac{(d,q)}{q}.
	\end{align*}
	The number of summands in the sum over $n$ is $\lessapprox 1+ \frac{q}{M_3}$ so we have 
	\begin{align*}
		|\hat{g}_{a,1}(\xi)| \lessapprox M_2^2(M_3+q) \sum_{\substack{1 \leq b \leq q \\ (b,q)=1}}  \sup_{\xi} |\hat{f}_b(\xi)|\leq qM^3 \sum_{\substack{1 \leq b \leq q \\(b,q)=1}}\int f_b(u)  \: \mathrm{d} u,
	\end{align*}
	since $f_b$ is a non-negative function. We conclude that
	\begin{align}\label{conI'}
		I'= \int\limits_{|\xi|\leq q^{-4}}\sum_{\substack{1 \leq a \leq q \\ (a, q)=1}}|\hat{g}_{a,1}(\xi)|^{2}\: \mathrm{d} \xi  \lessapprox \phi(q)M^{6}\bigg(\sum_{\substack{1 \leq b \leq q \\ (b, q)=1}}\int f_{b}(u)\:\mathrm{d}u\bigg)^{2}.
	\end{align}
	\subsubsection{The contribution of $I$.}\label{affTcI}
	By the definition of $\hat{g}_{a,2}(\xi)$, we see that 
	\begin{align*}
		|\hat{g}_{a,2}(\xi)|&=M_3\bigg|\sum_{\substack{1 \leq b \leq q \\ (b, q)=1}}\sum_{\substack{|m_{1}|\sim M_{1} \\ m_{2}\sim M_{2}}} \frac{m_{2}}{m_{1}}\hat{f}_{b}\bigg(\frac{m_{2}}{m_{1}}\xi\bigg) \sum_{\ell: |\xi - \ell m_{1}|\lessapprox \frac{M_{1}}{M_{3}} }\hat{\psi}_{1}\bigg(\bigg(\ell - \frac{\xi}{m_{1}}\bigg)M_{3}\bigg)\bigg| \\
		&\leq M_3\sum_{\substack{1 \leq b \leq q \\ (b, q)=1}}\sum_{\substack{|m_{1}|\sim M_{1} \\ m_{2}\sim M_{2}}} \frac{m_{2}}{m_{1}}\bigg|\hat{f}_{b}\bigg(\frac{m_{2}}{m_{1}}\xi\bigg)\bigg| \sum_{\ell: |\xi - \ell m_{1}|\lessapprox \frac{M_{1}}{M_{3}} }\bigg|\hat{\psi}_{1}\bigg(\bigg(\ell - \frac{\xi}{m_{1}}\bigg)M_{3}\bigg)\bigg|. 
	\end{align*}
	It follows from the Cauchy-Schwarz inequality and the positivity of $f_b$ that
	\begin{align*}
		|\hat{g}_{a,2}(\xi)|^{2}
		&\lessapprox M_{1}M_3^2\sum_{|m_{1}|\sim M_{1}}\sum_{\ell: |\xi - \ell m_{1}|\lessapprox \frac{M_{1}}{M_{3}}}\bigg(\sum_{m_{2}\sim M_{2}}\sum_{\substack{1 \leq b \leq q \\ (b, q)=1}}\frac{m_{2}}{m_{1}}\bigg|\hat{f}_{b}\bigg(\frac{m_{2}}{m_{1}}\xi\bigg)\bigg|\bigg)^{2}
		\notag
		\\
		&\lessapprox M_2^4 M_3^2 \bigg(\sum_{\substack{1 \leq b \leq q \\ (b,q)=1}}  \sup_{\xi} |\hat{f}_b(\xi)| \bigg)^2 \\
		&\lessapprox M_{2}^{4}M_{3}^{2}\bigg(\sum_{\substack{1 \leq b \leq q \\ (b, q)=1}}\int f_{b}(u)\:\mathrm{d}u\bigg)^{2}.
	\end{align*} 
	Since $M_1,M_2,M_3 \leq M$, we conclude that
	\begin{equation}\label{conI}
		I=\int\limits_{|\xi|\leq (qT)^{\eta}\frac{M_{1}}{M_{3}}}\sum_{\substack{1 \leq a \leq q \\ (a, q)=1}}|\hat{g}_{a,2}(\xi)|^{2}\: \mathrm{d} \xi \lessapprox \phi(q)M^{6}\bigg(\sum_{\substack{1 \leq b \leq q \\ (b, q)=1}}\int f_{b}(u)\:\mathrm{d}u\bigg)^{2}.
	\end{equation}
	\subsubsection{Completion of the proof of Lemma \ref{IterabJf}}\label{affTcompl}
	
	Substituting the contributions of $II'$, $II$, $I'$ and $I$ (see \eqref{conII'}, \eqref{conII}, \eqref{conI'} and \eqref{conI}) into \eqref{sepJf} yields the desired result. \qed
	
	\section{\texorpdfstring{Bounds for the second and fourth moments of $R$}
		{Bounds for the second and fourth moments of R}}\label{mobounds}
	In this section, we prove two key lemmas controlling the  moments of  $R$ and $\tilde{R}$. Our analysis adapts the methodology developed in \cite[Section 8]{GM}, with necessary modifications to account for the character sums in our setting. Recall from \eqref{funcR} that $R(v, a)$ is defined as:
	\[
	R(v,a):=\sum_{(t, \chi) \in W}v^{it}\chi(a), \quad v > 0, \ a\in \mathbb{Z}/q\mathbb{Z}.
	\]
	\begin{lemma}\label{secm}
		Let $W$ be a finite set of pairs $(t, \chi)$,  where $|t| \leq T$ and for $(t, \chi)\neq (t', \chi')$ either $\chi \neq \chi'$ or $|t-t'|\geq (qT)^{\epsilon}$. For any $M_{2}>0$, we have
		\[\int\limits_{v \asymp 1}\sum_{\substack{a \bmod q \\ (a,q)=1}}|R(v, a)|^{2}\:\mathrm{d}v \ll_{\epsilon} \phi(q)|W| \quad \mbox{and} \quad \int\limits_{v \asymp 1}\sum_{\substack{a \bmod q \\ (a,q)=1}}|\tilde{R}_{M_{2}}(v, a)|^{2}\:\mathrm{d}v \lessapprox \phi(q)|W|.\]
	\end{lemma}
	\begin{proof}
		From the definition of $\tilde{R}$ (see \eqref{deftiR}) we find
		\begin{align*}
			\int\limits_{v \asymp 1}|\tilde{R}_{M_{2}}(v, a)|^{2}\:\mathrm{d}v
			= &\int\limits_{v \asymp 1}\int\limits_{u' \in [1/2,\, 2]}\frac{NM_{2}}{q}\tilde{\psi}\bigg(\frac{NM_{2}}{q}(v-u')\bigg)| R(u', a)|^{2}
			\:\mathrm{d}u'\:\mathrm{d}v
			\\
			\lessapprox & \int\limits_{ u' \in [1/2,\, 2]}|R(u', a)|^{2}\:\mathrm{d}u'.
		\end{align*}
		Thus, it suffices to prove the result for $R$. 
		Let $\psi_{1}(v)$ be a smooth bump function that  majorizes the range of integration in the integral described in the lemma,  with support restricted to $v \asymp 1$.
		\begin{align*}
			\int\limits_{v \asymp 1}\sum_{\substack{a \bmod q \\ (a,q)=1}}|R(v, a)|^{2}\:\mathrm{d}v 
			&\leq \int\sum_{\substack{a \bmod q \\ (a,q)=1}}\psi_{1}(v)|R(v, a)|^{2}\:\mathrm{d}v  
			\\
			&= \int\sum_{\substack{a \bmod q \\ (a,q)=1}}\psi_{1}(v)\bigg|\sum_{(t, \chi) \in W}v^{it}\chi(a)\bigg|^{2}\:\mathrm{d}v.
		\end{align*}
		By making the change of variables $v=e^{\tau}$ and defining $\psi_{2}(\tau):=e^{\tau}\psi_{1}(e^{\tau})$, we obtain
		\begin{align*}
			\int\limits_{v \asymp 1}\sum_{\substack{a \bmod q \\ (a,q)=1}}\psi_{1}(v)\bigg|\sum_{(t, \chi) \in W}v^{it}\chi(a)\bigg|^{2}\:\mathrm{d}v  
			&= \int_{\tau \asymp 1}\sum_{\substack{a \bmod q \\ (a,q)=1}}e^{\tau}\psi_{1}(e^{\tau})\bigg|\sum_{(t, \chi) \in W}e^{it\tau}\chi(a)\bigg|^{2}\:\mathrm{d}\tau
			\\
			&=\sum_{\substack{(t_1,\chi_1),(t_2, \chi_2)\in W }}\sum_{\substack{a \bmod q \\ (a,q)=1}}\chi_{1}(a)\overline{\chi}_{2}(a)\hat{\psi}_{2}\bigg(\frac{t_{2}-t_{1}}{2\pi}\bigg) \\
			&=\sum_{\substack{(t_1,\chi_1),(t_2,\chi_2)\in W\\ \chi_{1}=\chi_{2}}}\phi(q)\hat{\psi}_{2}\bigg(\frac{t_{2}-t_{1}}{2\pi}\bigg).
		\end{align*}
		The final equality follows from the orthogonality of characters. Observe that $\psi_{2}$ is a smooth bump function satisfying $\Vert \psi_{2}^{(j)} \Vert_{\infty}=O_{j}(1)$ for any $j \in \mathbb{N}$. Consequently, repeated integration by parts yields $|\hat{\psi}_{2}(\xi)| \ll_{j}|\xi|^{-j}$ for any $j \in \mathbb{N}$. In particular, when $|t_{1}-t_{2}| \geq (qT)^{\epsilon}$, we have $\hat{\psi}_{2}(\frac{t_{1}-t_{2}}{2\pi}) \ll_{\epsilon} (qT)^{-100}$. Thus, the terms with $t_{1} \neq t_{2}$ are negligible. Finally, the terms with $t_{1}=t_{2}$ contribute $\ll \phi(q)|W|$ to the sum. This concludes the proof. 
	\end{proof}
	
	The following lemma controls the fourth moments of $R$ and $\tilde{R}$ by the energy of $W$.
	\begin{lemma}\label{fourthm}
		Let $W$ be a finite set of pairs $(t, \chi)$,  where $|t| \leq T$ and for $(t, \chi)\neq (t', \chi')$ either $\chi \neq \chi'$ or $|t-t'|\geq (qT)^{\epsilon}$. For any $M_{2}>0$, we have
		\[\int\limits_{v \asymp 1}\sum_{\substack{a \bmod q \\ (a,q)=1}}|R(v, a)|^{4}\:\mathrm{d}v \lessapprox \phi(q)E(W) \quad \mbox{and} \quad \int\limits_{v \asymp 1}\sum_{\substack{a \bmod q \\ (a,q)=1}}|\tilde{R}_{M_{2}}(v, a)|^{4}\:\mathrm{d}v \lessapprox_{\epsilon} \phi(q)E(W).
		\]
	\end{lemma}
	\begin{proof}
		From the definition of $\tilde{R}$ (see \eqref{deftiR}) and by applying the Cauchy-Schwarz inequality, it follows that
		\begin{align*}
			\int\limits_{v \asymp 1}|\tilde{R}_{M_{2}}(v, a)|^{4}\:\mathrm{d}v
			&= \int\limits_{v \asymp 1}\bigg(\int\limits_{ u' \in [1/2,\, 2]}\frac{NM_{2}}{q}\tilde{\psi}\bigg(\frac{NM_{2}}{q}(v-u')\bigg)| R(u', a)|^{2}
			\:\mathrm{d}u' \bigg)^{2}\:\mathrm{d}v
			\\
			&\leq \bigg(\frac{NM_{2}}{q}\bigg)^{2}\int\limits_{v \asymp 1}\bigg(\int\tilde{\psi}\bigg(\frac{NM_{2}}{q}(v-u')\bigg)^{2}
			\:\mathrm{d}u'\bigg)
			\bigg(\int\limits_{ \substack{|u'-v| \leq \frac{4(qT)^{\epsilon}q}{M_{2}N} \\ u' \in [1/2, \, 2]}}| R(u', a)|^{4}
			\:\mathrm{d}u' \bigg)\:\mathrm{d}v
			\\
			&\ll \bigg(\frac{NM_{2}}{q}\bigg)^{2}\bigg(\frac{q}{NM_{2}}\bigg)(qT)^{\epsilon}\int\limits_{v \asymp 1}
			\int\limits_{ \substack{|u'-v| \leq \frac{4(qT)^{\epsilon}q}{M_{2}N} \\ u' \in [1/2, \, 2]}}|R(u', a)|^{4}
			\:\mathrm{d}u'\:\mathrm{d}v \\
			&\lessapprox
			\int\limits_{u' \in [1/2, \, 2]}|R(u', a)|^{4}
			\:\mathrm{d}u'.
		\end{align*}
		Thus, it suffices to prove the result for $R$. By making the change of variables $v=e^{\tau}$, we have
		\[\int\limits_{v \asymp 1}\sum_{\substack{a \bmod q \\ (a,q)=1}}|R(v, a)|^{4}\:\mathrm{d}v =
		\int\limits_{|\tau|\ll 1}\sum_{\substack{a \bmod q \\ (a,q)=1}}\bigg|\sum_{(t, \chi) \in W}e^{it\tau}\chi(a)\bigg|^{4}e^{\tau}\:\mathrm{d}\tau.\]
		Let $\psi(\tau)$ be a smooth bump function supported on $|\tau| \ll 1$ such that $\psi(\tau/(qT)^{\epsilon})$ majorizes the range of integration. It follows that the above expression is bounded by
		\begin{align*}
			&\leq \int\sum_{\substack{a \bmod q \\ (a,q)=1}}\psi\bigg(\frac{\tau}{(qT)^{\epsilon}}\bigg)\bigg|\sum_{(t, \chi) \in W}e^{it\tau}\chi(a)\bigg|^{4}e^{\tau}\:\mathrm{d}\tau
			\\
			&=\sum_{\substack{(t_{i}, \chi_{i})\in W \\ 1\leq i \leq 4}}\sum_{\substack{a \bmod q \\ (a,q)=1}}\int \psi\bigg(\frac{\tau}{(qT)^{\epsilon}}\bigg)e^{i\tau(t_{1}+t_{2}-t_{3}-t_{4})}\chi_{1}(a)\chi_{2}(a)\overline{\chi}_{3}(a)\overline{\chi}_{4}(a)\:\mathrm{d}\tau
			\\
			&=\sum_{\substack{(t_{i}, \chi_{i})\in W \\ 1\leq i \leq 4}}\sum_{\substack{a \bmod q \\ (a,q)=1}}(qT)^{\epsilon} \hat{\psi}\bigg(\frac{(qT)^{\epsilon}(t_{3}+t_{4}-t_{1}-t_{2})}{2\pi}\bigg)\chi_{1}(a)\chi_{2}(a)\overline{\chi}_{3}(a)\overline{\chi}_{4}(a)
			\\
			&=\sum_{\substack{(t_{i}, \chi_{i})\in W, \, 1\leq i \leq 4 \\ \chi_{1}\chi_{2}=\chi_{3}\chi_{4}}}(qT)^{\epsilon} \hat{\psi}\bigg(\frac{(qT)^{\epsilon}(t_{3}+t_{4}-t_{1}-t_{2})}{2\pi}\bigg)\phi(q).
		\end{align*}
		The last line follows by the orthogonality of characters. Due to the rapid decay of $\hat{\psi}$,  the summation can be restricted to $|t_{1}+t_{2}-t_{3}-t_{4}| \leq 1$, incurring an error of $O_{\epsilon}((qT)^{-100})$. The proof is then completed by observing that the remaining terms contribute $\lessapprox \phi(q)E(W)$.
	\end{proof}
	
	\section{Energy bound}\label{Engb}
	We recall the definition of the energy for a finite set $W$ of pairs $(t, \chi)$:
	\[E(W):=\# \{ (t_1,\chi_1), (t_2,\chi_2), (t_3,\chi_3) , (t_4,\chi_4) \in W :     |t_{1}+t_{2}-t_{3}-t_{4}| \leq 1,  \chi_{1}\chi_{2}=\chi_{3}\chi_{4}\}.\] 
	The goal of this section is to establish the following bound for the energy of $W$. Our approach is based on Heath-Brown's bound for double zeta sums.
	\begin{theorem}[Heath-Brown, \cite{HBLV}]\label{HBT}
		Let $\mathcal{T}$ be a set of pairs $(t, \chi)$,   where $\chi$ is a primitive character modulo $q$ and $t$ is a real number with $|t|\leq T$ such that for $(t, \chi)\neq (t', \chi')$ either $\chi \neq \chi'$ or $|t-t'|\geq 1$. Then
		\[\sum_{\substack{(t_1,\chi_1),(t_2,\chi_2)\in \mathcal{T}}}\bigg|\sum_{n \sim N} c_{n}\chi_{1}(n)\overline{\chi}_{2}(n)n^{i(t_{1}-t_{2})}\bigg|^{2} \lessapprox(|\mathcal{T}|N^2+|\mathcal{T}|^{2}N+|\mathcal{T}|^{\frac{5}{4}}(qT)^{\frac{1}{2}}N)\max_{n \sim N}|c_{n}|^{2}.\]
	\end{theorem}
	We aim to prove the following bound:
	\begin{proposition}\label{enb}
		Let
		\[
		D_{N}(t, \chi)=\sum_{N < n \leq 2N} a_{n}\chi(n)n^{it},
		\]
		where $\chi$ is a primitive character modulo $q$ and $|a_{n}| \leq1$. Suppose that $W$ is a finite set of pairs $(t, \chi)$,  where $|t| \leq T$ and for $(t, \chi)\neq (t', \chi')$ either $\chi \neq \chi'$ or $|t-t'|\geq (qT)^{\epsilon}$. Assume that
		\[ |D_{N}(t, \chi)| \geq N^{\sigma}/6
		\]
		for all $(t, \chi) \in W$ and $(qT)^{2/3}/2\leq N\leq qT$. Assume that either $|W|\geq (qT)^{\frac{2}{3}}$ or that $E(W)^{\frac{1}{8}}N \geq (qT)^{\frac{1}{2}}|W|^{\frac{5}{8}}$. Then we have
		\[
		E(W)\lessapprox_{\epsilon} |W|N^{4-4\sigma}+|W|^{21/8}(qT)^{1/4}N^{1-2\sigma}+|W|^3N^{1-2\sigma}.
		\]
	\end{proposition}
	\begin{remark}
		The extra assumption allows us to weaken the lower bound on $N$; it can be compared to the condition $T^{\frac{3}{4}} < N$ in Guth-Maynard's Proposition 11.1. For $T>N^{\frac{6}{5}}$, they can use subdivision so weakening this lower bound is not useful for them but is of value in this case.
		
		Ultimately, we only need an upper bound for $S_3$ and this additional
		assumption does not affect the final bound. Suppose that
		$|W|<(qT)^{2/3}$ and
		$E(W)^{1/8}N < (qT)^{1/2}|W|^{5/8}$. Then
		$(qT)^2|W|^{3/2}
		>
		qTN|W|^{1/2}E(W)^{1/2}$.
		Indeed, the second assumption gives
		$E(W)^{1/2}
		<
		(qT)^2|W|^{5/2}N^{-4}$,
		and hence
		\[qTN|W|^{1/2}E(W)^{1/2}
		<
		(qT)^3|W|^3N^{-3}
		<
		qT|W|^3
		<
		(qT)^2|W|^{3/2}.\]
		Here the first strict inequality follows from the lower bound $N \geq (qT)^{2/3}$ in the present range, while the last inequality follows from
		$|W|<(qT)^{2/3}$. In this case, our argument in Section \ref{S3bandpomp} would imply (cf. \eqref{S3p} )
		$S_3 \lessapprox (qT)^2|W|^{3/2}$.
	\end{remark}
	
	We begin by recalling a result from \cite[Lemma 11.3]{GM}, which provides an upper bound for the pointwise values of a Dirichlet polynomial in terms of an average of its values in a neighborhood of the point.
	\begin{lemma}[Dirichlet polynomials do not vary too quickly]\label{Dpdnvf} 
		We have 
		\[|D_{N}(t, \chi)|\lessapprox \int_{|u-t|\lessapprox 1}|D_{N}(u, \chi)| \:\mathrm{d}u+O((qT)^{-100}).\]
	\end{lemma}
	
	The following lemma establishes that the energy of $W$ can be bounded in terms of the discrete third moment of $R$.
	\begin{lemma}\label{econ3}
		We have 
		\[E(W)\lessapprox N^{-2\sigma}\sum_{\substack{n_1,n_2\sim N \\  (n_1n_2, q)=1}}\bigg|R\bigg(\frac{n_{1}}{n_{2}}, n_{1}n_{2}^{-1}\bigg)\bigg|^{3},\]
	\end{lemma}
	\begin{proof}
		Since $|D_N(t)|=|\sum_{n \sim N} a_{n}\chi(n)n^{it}| \geq N^{\sigma}$ for $(t, \chi) \in W$, we see that
		\[E(W) \ll N^{-2\sigma}\sum_{\substack{(t_{i}, \chi_{i})\in W, \, 1\leq i \leq 4 \\ |t_{1}+t_{2}-t_{3}-t_{4}|\leq 1 \\ \chi_{1}\chi_{2}=\chi_{3}\chi_{4}}}|D_{N}(t_{4}, \chi_{4})|^{2}.\]
		By Lemma \ref{Dpdnvf} and Cauchy-Schwarz, if $|t_{1}+t_{2}-t_{3}-t_{4}|\leq 1$ and $\chi_{1}\chi_{2}=\chi_{3}\chi_{4}$ we have
		\begin{align*}
			|D_{N}(t_{4}, \chi_{4})|^{2} &\lessapprox \int_{|u-t_{4}|\lessapprox 1}|D_{N}(u, \chi_{1}\chi_{2}\overline{\chi}_{3})|^{2} \:\mathrm{d}u+O((qT)^{-100}) 
			\\
			&\lessapprox \int_{|u-(t_{1}+t_{2}-t_{3})|\lessapprox 1}|D_{N}(u, \chi_{1}\chi_{2}\overline{\chi}_{3})|^{2} \:\mathrm{d}u+O((qT)^{-100}).
		\end{align*}
		In view of the condition satisfied by the set $W$, given $(t_{i}, \chi_{i})\in W$ for $1\leq i \leq 3$, there exists at most one choice of $(t_{4}, \chi_{4}) \in W$ such that $|t_{1}+t_{2}-t_{3}-t_{4}|\leq 1$ and $\chi_{1}\chi_{2}=\chi_{3}\chi_{4}$. Therefore, we have
		\begin{align*}
			E(W) &\lessapprox N^{-2\sigma}\sum_{\substack{(t_{i}, \chi_{i})\in W \\ 1\leq i \leq 3}}\int_{|s|\lessapprox 1}|D_{N}(t_{1}+t_{2}-t_{3}-s, \chi_{1}\chi_{2}\overline{\chi}_{3})|^{2} \:\mathrm{d}s+O((qT)^{-50})
			\\
			&=N^{-2\sigma}\sum_{\substack{n_1,n_2 \sim N \\ (n_1n_2,q)=1}}a_{n_{1}}\overline{a}_{n_{2}}\int_{|s|\lessapprox 1}R\bigg(\frac{n_{1}}{n_{2}}, n_{1}n_{2}^{-1}\bigg)^{2}R\bigg(\frac{n_{2}}{n_{1}}, n_{2}n_{1}^{-1}\bigg)\bigg(\frac{n_{2}}{n_{1}}\bigg)^{is}  \:\mathrm{d}s+O((qT)^{-50})
		\end{align*}
		Thus we have
		\begin{align*}
			E(W)\lessapprox N^{-2\sigma}\sum_{\substack{n_1,n_2 \sim N, \\ (n_1n_2, q)=1}}\bigg|R\bigg(\frac{n_{1}}{n_{2}}, n_{1}n_{2}^{-1}\bigg)\bigg|^{3},
		\end{align*}
		since $|a_n| \leq 1$.
	\end{proof}
	Next, we derive bounds for the discrete second and fourth moments of $R$, leveraging Heath-Brown's theorem (see Theorem \ref{HBT}).
	\begin{lemma}\label{secmR}
		For any $M \geq 1$, we have 
		\[ \sum_{\substack{n_1,n_2 \sim M  \\  (n_1n_2, q)=1}}\bigg|R\bigg(\frac{n_{1}}{n_{2}}, n_{1}n_{2}^{-1}\bigg)\bigg|^{2} \lessapprox |W|M^{2}+ |W|^{2}M+ |W|^{5/4}(qT)^{1/2}M.\]
	\end{lemma}
	\begin{proof}
		We have 
		\begin{align*}
			\sum_{\substack{n_1,n_2 \sim M \\ (n_1n_2, q)=1}}\bigg|R\bigg(\frac{n_{1}}{n_{2}}, n_{1}n_{2}^{-1}\bigg)\bigg|^{2}&=\sum_{\substack{n_1,n_2 \sim M \\ (n_1n_2, q)=1}}\bigg| \sum_{(t, \chi) \in W}\bigg(\frac{n_{1}}{n_{2}}\bigg)^{it}\chi(n_{1})\overline{\chi}(n_{2})\bigg|^{2}
			\\
			&=\sum_{\substack{(t_1,\chi_1),(t_2,\chi_2)\in W}}\bigg| \sum_{n\sim M}n^{i(t_{1}-t_{2})}\chi_{1}\overline{\chi}_{2}(n)\bigg|^{2}.
		\end{align*}
		Applying Theorem \ref{HBT} yields the desired result. 
	\end{proof}
	\begin{lemma}\label{fourmR}
		For any $M\geq 1$, we have 
		\[ \sum_{\substack{n_1,n_2 \sim M \\ (n_1n_2, q)=1}}\bigg|R\bigg(\frac{n_{1}}{n_{2}}, n_{1}n_{2}^{-1}\bigg)\bigg|^{4} \lessapprox E(W)M^{2}+ |W|^{4}M+ E(W)^{3/4}|W|(qT)^{1/2}M.\]
	\end{lemma}
	\begin{proof}
		We define the set
		\[U_{B, q_{1}}:=\{(u, \chi): \# \{ (t_{1}, \chi_{1}), (t_{2}, \chi_{2})\in W :     \lfloor t_{1}-t_{2} \rfloor =u,  \chi_{1}\overline{\chi}_{2}=\chi  \}\sim B, \text{cond}(\chi)=q_1\},\]
		where $\text{cond}(\chi)$ is the conductor of $\chi$. Clearly, $U_{B, q_{1}}$ is empty if $B < 1/2$ or if $B > |W|$. Therefore, by the Cauchy-Schwarz inequality and the divisor bound, we deduce that
		\begin{align*}
			|R(x, a)|^{4}&=\bigg|\sum_{\substack{(t_1, \chi_1),(t_2,\chi_2)\in W }}x^{i(t_{1}-t_{2})}\chi_{1}(a)\overline{\chi}_{2}(a)\bigg|^{2}
			\\
			&=\bigg|\sum_{q_{1}| q}\sum_{\substack{B=2^{j} \\  j\geq -1}}\sum_{(u, \chi)\in U_{B, q_{1}}}\sum_{\substack{(t_1,\chi_1), (t_2,\chi_2) \in W  \\ \lfloor t_{1}-t_{2} \rfloor =u, \, \chi_{1}\overline{\chi}_{2}=\chi}}x^{i(t_{1}-t_{2})}\chi(a)\bigg|^{2}
			\\
			&\lessapprox \sum_{q_{1}| q}\sum_{B=2^{j}\leq |W|}\bigg|\sum_{(u, \chi)\in U_{B, q_{1}}}\sum_{\substack{(t_1,\chi_1),(t_2,\chi_2) \in W \\ \lfloor t_{1}-t_{2} \rfloor =u, \, \chi_{1}\overline{\chi}_{2}=\chi}}x^{i(t_{1}-t_{2})}\chi(a)\bigg|^{2}.
		\end{align*}
		Taking $x=n_1/n_2$, $a=n_{1}n_{2}^{-1}$ and summing over $n_1,n_2 \sim M$, $(n_1n_2, q)=1$ then gives 
		\begin{align*}
			&\sum_{\substack{n_1,n_2 \sim M, \\ (n_1n_2, q)=1}}\bigg|R\bigg(\frac{n_{1}}{n_{2}}, n_{1}n_{2}^{-1}\bigg)\bigg|^{4}  \lessapprox \sup_{\substack{B \leq |W| \\ q \mid q_1}}\sum_{\substack{n_1,n_2 \sim M, \\ (n_1n_2, q)=1}} \bigg|\sum_{(u, \chi)\in U_{B,q_{1}}}\sum_{\substack{(t_1, \chi_1), (t_2,\chi_2) \in W \\ \lfloor t_{1}-t_{2} \rfloor =u, \, \chi_{1}\overline{\chi}_{2}=\chi}}\bigg(\frac{n_{1}}{n_{2}}\bigg)^{i(t_{1}-t_{2})}\chi(n_{1})\overline{\chi}(n_{2})\bigg|^{2}.
		\end{align*}
		Expanding the square, we find that the right hand side is bounded by 
		\begin{align}
			&\lessapprox  \sup_{\substack{B \leq |W| \\ q \mid q_1}}\sum_{\substack{(u' , \chi') \in U_{B, q_{1}} \\ (u'', \chi'')\in U_{B, q_{1}}}}
			\sum_{\substack{(t_{1}, \chi_{1}), (t_{3}, \chi_{3}) \in W \\ \lfloor t_{1}-t_{3} \rfloor =u', \, \chi_{1}\overline{\chi}_{3}=\chi'}}
			\sum_{\substack{(t_{2}, \chi_{2}), (t_{4}, \chi_{4}) \in W \\ \lfloor t_{2}-t_{4} \rfloor =u'', \, \chi_{2}\overline{\chi}_{4}=\chi''}}
			\sup_{|s|\ll 1}
			\bigg|\sum_{n \sim M} n^{i(u'-u''+s)}\chi'(n)\overline{\chi''}(n)\bigg|^{2} \notag
			\\
			&\lessapprox  \sup_{\substack{B \leq |W| \\ q \mid q_1}}B^{2}\sum_{\substack{(u' , \chi') \in U_{B, q_{1}} \\ (u'', \chi'')\in U_{B, q_{1}}}}
			\sup_{|s|\ll 1}
			\bigg|\sum_{n \sim M} n^{i(u'-u''+s)}\chi'(n)\overline{\chi''}(n)\bigg|^{2} \notag
			\\
			&=  \sup_{\substack{B \leq |W| \\ q \mid q_1}}B^{2}\sum_{\substack{(u' , \chi') \in U^{*}_{B, q_{1}} \\ (u'', \chi'')\in U^{*}_{B, q_{1}}}}
			\sup_{|s|\ll 1}
			\bigg|\sum_{n \sim M} n^{i(u'-u''+s)}\chi_{0}(n)\chi'(n)\overline{\chi''}(n)\bigg|^{2} \notag,
		\end{align}
		where $U^{*}_{B,q_{1}}:= \{(u, \chi^{*}): \chi^{*} \ \text{is primitive}, (u, \chi_{0}\chi^{*}) \in U_{B, q_{1}} \}$ and $\chi_0$ denotes the principal character modulo $q$.
		In the above, with Lemma \ref{Dpdnvf}, we replace the supremum with an integral and then apply Theorem \ref{HBT} which gives the upper bound
		\begin{align*}
			&\lessapprox \sup_{\substack{B \leq |W| \\ q \mid q_1}}B^{2}\int_{|t|\lessapprox 1}\sum_{\substack{(u' , \chi') \in U^{*}_{B,q_{1}} \\ (u'', \chi'')\in U^{*}_{B,q_{1}}}}
			\bigg|\sum_{n \sim M} n^{i(u'-u''+t)}\chi_{0}(n)\chi'(n)\overline{\chi''}(n)\bigg|^{2} \: \mathrm{d}t+O(B^2|U^{*}_{B,q_{1}}|^{2}(qT)^{-100}) 
			\\
			&\lessapprox \sup_{q_{1}|q}\sup_{B \leq |W|}B^{2}(M^{2}|U^{*}_{B,q_{1}}|+|U^{*}_{B,q_{1}}|^{2}M+(q_{1}T)^{1/2}|U^{*}_{B,q_{1}}|^{5/4}M ).
		\end{align*}
		Notice that $B|U^{*}_{B,q_{1}}|\leq |W|^{2}$ and $B^{2}|U^{*}_{B,q_{1}}|\leq E(W)$. Thus, we conclude that
		\[\sum_{\substack{n_1,n_2 \sim M, \\ (n_1n_2, q)=1}}\bigg|R\bigg(\frac{n_{1}}{n_{2}}, n_{1}n_{2}^{-1}\bigg)\bigg|^{4} \lessapprox   E(W)M^{2}+ |W|^{4}M+ E(W)^{3/4}|W|(qT)^{1/2}M,\]
		which completes the proof.
	\end{proof}
	
	An upper bound for $\sum_{n_{1}, n_{2}}|R(\frac{n_{1}}{n_{2}}, n_{1}n_{2}^{-1})|^{3}$ can be readily established using the Cauchy-Schwarz inequality in conjunction with the above two lemmas. However, as noted in \cite[p. 42]{GM}, a sharper estimate can be obtained by first splitting the sum according to the size of $(n_{1}, n_{2})$ and then evaluating the resulting contributions separately. To handle the sum involving terms with small $(n_{1}, n_{2})$, we rely on the following lemma, which indicates that for $v\asymp 1$, $|R(v,a)|$ is morally locally constant at scale $1/T$.
	
	\begin{lemma}\label{RbaveR}
		For $v \in [1/C, C]$ with a fixed constant $C> 1$ and $(a, q)=1$, we have
		\[|R(v,a)| \ll_{\epsilon, C} T\int_{|v'-v|\leq T^{\epsilon}/T}|R(v', a)|\:\mathrm{d}v'+O(T^{-100}).\]
	\end{lemma}
	\begin{proof}
		Recall that $W$ is a set of pairs $(t, \chi)$ where $|t|\leq T$. Let $\psi$ be a smooth bump which is $1$ on $[-1, 1]$. We then have, by the Fourier inversion formula,
		\begin{align*}
			R(v, a)&= \sum_{(t, \chi) \in W} e^{it \log v}\chi(a)=\sum_{(t, \chi) \in W} \psi\bigg(\frac{t}{T}\bigg)e^{it \log v}\chi(a) = \int_{\mathbb{R}}\hat{\psi}(\xi)\sum_{(t, \chi) \in W} e^{it(\log v+2\pi \xi /T)}\chi(a) \: \mathrm{d} \xi.
		\end{align*}
		As $\hat{\psi}$ decreases rapidly, we may restrict the integral to $|\xi| \leq \frac{T^{\epsilon}}{4\pi C}$ at the cost of an $O_{\epsilon, C}(T^{-100})$ error term. By performing the change of variables $\log v +2\pi \xi /T= \log v'$, it follows that
		\begin{align*}
			\bigg|\int\limits_{|\xi| \leq \frac{T^{\epsilon}}{4\pi C}}\hat{\psi}(\xi)\sum_{(t, \chi) \in W} e^{it(\log v+2\pi \xi /T)}\chi(a) \: \mathrm{d} \xi \bigg|&=\bigg|\int\limits_{|\log v'- \log v| \leq  \frac{T^{\epsilon}}{2CT}}\hat{\psi}\bigg(\frac{(\log v'-\log v)T}{2\pi}\bigg)R(v', a) \: \frac{T \mathrm{d}v'}{2\pi v'}\bigg| \\
			&\ll_{C} T\int_{|v'-v|\leq \frac{T^{\epsilon}}{T}}|R(v', a)| \: \mathrm{d}v'.
		\end{align*}
		In the final step, we utilized the estimate  \[|v'-v|=|v(e^{(\log v'-\log v)}-1)|\leq 2C|\log v' -\log v|,\]
		along with the bound $|\hat{\psi}|\ll 1$. Gathering these results yields the desired conclusion.
	\end{proof}
	
	\begin{lemma}[Small GCD terms] \label{sgcd}
		We have
		\[\sum_{\substack{n_{1}, n_{2} \sim N, \\ (n_1n_2, q)=1  \\ (n_{1}, n_{2}) \leq D}}\bigg|R\bigg(\frac{n_{1}}{n_{2}}, n_{1}n_{2}^{-1}\bigg)\bigg|^{3} \lessapprox_{\epsilon} (DqT+N^{2})|W|^{1/2}E(W)^{1/2}.\]
		\begin{proof}
			Let $d=(n_{1}, n_{2})$ and $n_{1}=n_{1}'d$, $n_{2}=n_{2}'d$ for some $n_{1}', n_{2}' \sim N/d$ with $(n_{1}', n_{2}')=1$. It follows from Lemma \ref{RbaveR} that, for $(a, q)=1$,
			\begin{align}
				\sum_{\substack{n_1', n_2' \sim N/d, \\ (n_1'n_2', q)=(n_1', n_2')=1  \\ n_{1}'n_{2}'^{-1} \equiv a\bmod q}}\bigg|R\bigg(\frac{n_{1}'}{n_{2}'}, n_{1}'n_{2}'^{-1}\bigg)\bigg|^{3}
				&\lessapprox
				\sum_{\substack{n_{1}, n_{2} \sim N/d, \\ (n_1'n_2', q)=(n_{1}',n_2')=1  \\ n_{1}'n_{2}'^{-1} \equiv a \bmod q}}T\int_{|v-\frac{n_{1}'}{n_{2}'}|\leq \frac{T^{\epsilon}}{T}}|R(v, a)|^{3}\:\mathrm{d}v+O(N^{2}T^{-100}) \notag
				\\
				&=T\int_{|v|\asymp 1}|R(v, a)|^{3}
				\sum_{\substack{n_{1}, n_{2} \sim N/d, \\ (n_{1}'n_2', q)=(n_{1}', n_2')=1  \\ n_{1}'n_{2}'^{-1} \equiv a \bmod q \\ |\frac{n_{1}'}{n_{2}'}-v|\leq \frac{T^{\epsilon}}{T}}}1\:\mathrm{d}v+O(N^{2}T^{-100}). \label{small GCD RHS}
			\end{align}
			If we write $n_{2}'=n_{1}'a^{-1}+q\ell'$, $n_{2}''=n_{1}''a^{-1}+q\ell''$, then  the condition $n_{1}'/n_{2}' \neq n_{1}''/n_{2}''$ implies
			\begin{align*}
				\bigg|\frac{n_{1}'}{n_{1}'a^{-1}+q\ell'}-\frac{n_{1}''}{n_{1}''a^{-1}+q\ell''}\bigg|=
				\bigg|\frac{q(n_{1}'\ell''-n_{1}''\ell')}{(n_{1}'a^{-1}+q\ell')(n_{1}''a^{-1}+q\ell'')}\bigg|\geq
				\frac{q}{n_{2}'n_{2}''} \gg qd^{2}/N^{2}.
			\end{align*}
			Hence, the fractions $n_{1}'/n_{2}'$ that appear in the inner sum are $qd^{2}/N^{2}$-separated. This separation enables us to bound the inner sum in \eqref{small GCD RHS} by $\lessapprox 1+ N^{2}/(qTd^{2})$. Summing over $d \leq D$ and $a \bmod q$ gives 
			\begin{align*}
				\sum_{\substack{n_{1}, n_{2} \sim N, \\ (n_{1}n_{2}, q)=1  \\ (n_{1}, n_{2}) \leq D}}\bigg|R\bigg(\frac{n_{1}}{n_{2}}, n_{1}n_{2}^{-1}\bigg)\bigg|^{3} =\sum_{d \leq D} \sum_{\substack{a \bmod q \\ (a, q)=1}}\sum_{\substack{n_{1}, n_{2} \sim N/d, \\ (n_{1}'n_{2}', q)=(n_1',n_2')=1  \\  n_{1}'n_{2}'^{-1} \equiv a \bmod q}}\bigg|R\bigg(\frac{n_{1}'}{n_{2}'}, n_{1}'n_{2}'^{-1}\bigg)\bigg|^{3}
			\end{align*}
			Thus by \eqref{small GCD RHS}, we have that 
			\begin{align}
				\sum_{\substack{n_{1}, n_{2} \sim N, \\ (n_{1}n_{2}, q)=1  \\ (n_{1}, n_{2}) \leq D}}\bigg|R\bigg(\frac{n_{1}}{n_{2}}, n_{1}n_{2}^{-1}\bigg)\bigg|^{3}&\lessapprox\sum_{d \leq D}\bigg(T+\frac{N^{2}}{d^{2}q}\bigg)\sum_{\substack{a \bmod q \\ (a, q)=1}}\int_{|v|\asymp 1}|R(v, a)|^{3}
				\:\mathrm{d}v+O(N^{2}T^{-100}) \notag \\
				&\lessapprox \bigg(DT+ \frac{N^2}{q}\bigg) \sum_{\substack{a \bmod q \\ (a, q)=1}}\int_{|v|\asymp 1}|R(v, a)|^{3}
				\:\mathrm{d}v + O(N^2T^{-100}). \label{small GCD intermediate bound}
			\end{align}
			We may bound the third moment by Cauchy-Schwarz and Lemmas \ref{secm} and \ref{fourthm},
			\begin{align*}
				\sum_{\substack{a \bmod q \\ (a, q)=1}}\int_{|v|\asymp 1}|R(v, a)|^{3}
				\:\mathrm{d}v &\leq \bigg(\sum_{\substack{a \bmod q \\ (a, q)=1}}\int_{|v|\asymp 1}|R(v, a)|^{2}
				\:\mathrm{d}v\bigg)^{1/2}
				\bigg(\sum_{\substack{a \bmod q \\ (a, q)=1}}\int_{|v|\asymp 1}|R(v, a)|^{4}
				\:\mathrm{d}v\bigg)^{1/2} \\
				&\lessapprox_{\epsilon} q|W|^{1/2}E(W)^{1/2}
			\end{align*}
			Putting this into \eqref{small GCD intermediate bound}, we get
			\begin{align*}
				\sum_{\substack{n_{1}, n_{2} \sim N, \\ (n_{1}n_{2}, q)=1  \\ (n_{1}, n_{2}) \leq D}}\bigg|R\bigg(\frac{n_{1}}{n_{2}}, n_{1}n_{2}^{-1}\bigg)\bigg|^{3}\lessapprox_{\epsilon}(DqT+N^{2})|W|^{1/2}E(W)^{1/2}.
			\end{align*}
			This completes the proof.
		\end{proof}	
	\end{lemma}
	To make $DqT=N^{2}$, we choose $D=N^{2}/(qT)$.  Substituting this into the lemma, we obtain: 
	\begin{equation}\label{sgcdb}
		\sum_{\substack{n_{1}, n_{2} \sim N, \\ (n_{1}n_{2}, q)=1  \\ (n_{1}, n_{2}) \leq D}}\bigg|R\bigg(\frac{n_{1}}{n_{2}}, n_{1}n_{2}^{-1}\bigg)\bigg|^{3} \lessapprox N^{2}|W|^{1/2}E(W)^{1/2}.
	\end{equation}
	
	\begin{lemma}[Large GCD terms] \label{lgcd}
		Let $D=N^{2}/(qT)$ and $N \geq (qT)^{2/3}/2$. Assume that either $|W|\geq (qT)^{\frac{2}{3}}$ or that $E(W)^{\frac{1}{8}}N \geq (qT)^{\frac{1}{2}}|W|^{\frac{5}{8}}$. Then we have 
		\[\sum_{\substack{n_{1}, n_{2} \sim N, \\ (n_{1}n_{2}, q)=1  \\ (n_{1}, n_{2}) \geq D}}\bigg|R\bigg(\frac{n_{1}}{n_{2}}, n_{1}n_{2}^{-1}\bigg)\bigg|^{3} \lessapprox N|W|^{3}+N(qT)^{1/4}|W|^{21/8}+E(W)^{1/2}|W|^{1/2}N^{2}.\]
		\begin{proof}
			As in the previous lemma, we let $d=(n_{1}, n_{2})$ and $n_{1}=n_{1}'d$, $n_{2}=n_{2}'d$. By the Cauchy-Schwarz inequality, we have
			\begin{align*}
				\sum_{\substack{n_{1}', n_{2}' \sim N/d, \\ (n_{1}'n_{2}', q)=1}}\bigg|R\bigg(\frac{n_{1}}{n_{2}}, n_{1}n_{2}^{-1}\bigg)\bigg|^{3} \leq \bigg(\sum_{\substack{n_{1}', n_{2}' \sim N/d, \\ (n_{1}'n_{2}', q)=1}}\bigg|R\bigg(\frac{n_{1}}{n_{2}}, n_{1}n_{2}^{-1}\bigg)\bigg|^{4}\bigg)^{1/2}
				\bigg(\sum_{\substack{n_{1}', n_{2}' \sim N/d, \\ (n_{1}'n_{2}', q)=1}}\bigg|R\bigg(\frac{n_{1}}{n_{2}}, n_{1}n_{2}^{-1}\bigg)\bigg|^{2}\bigg)^{1/2}.
			\end{align*} 
			Applying Lemmas \ref{secmR} and \ref{fourmR}  yields
			\begin{align*}
				\sum_{\substack{n_{1}', n_{2}' \sim N/d, \\ (n_{1}'n_{2}', q)=1}}\bigg|R\bigg(\frac{n_{1}}{n_{2}}, n_{1}n_{2}^{-1}\bigg)\bigg|^{3}
				\lessapprox& \bigg(\frac{|W|N^{2}}{d^{2}}+\frac{|W|^{2}N}{d}+\frac{|W|^{5/4}(qT)^{1/2}N}{d}\bigg)^{1/2}
				\\
				\cdot&\bigg(\frac{E(W)N^{2}}{d^{2}}+\frac{|W|^{4}N}{d}+\frac{E(W)^{3/4}|W|(qT)^{1/2}N}{d}\bigg)^{1/2}.
			\end{align*} 
			By summing over $d \geq D$ and using Cauchy-Schwarz, we arrive at
			\begin{align*}
				\sum_{\substack{n_{1}, n_{2} \sim N, \\ (n_{1}n_{2}, q)=1 \\ (n_{1}, n_{2})\geq D}}\bigg|R\bigg(\frac{n_{1}}{n_{2}}, n_{1}n_{2}^{-1}\bigg)\bigg|^{3}
				\lessapprox& \bigg(\frac{|W|N^{2}}{D}+|W|^{2}N+|W|^{5/4}(qT)^{1/2}N\bigg)^{1/2}
				\\
				\cdot&\bigg(\frac{E(W)N^{2}}{D}+|W|^{4}N+E(W)^{3/4}|W|(qT)^{1/2}N\bigg)^{1/2}.
			\end{align*} 
			Since $N \geq (qT)^{1/2}$ and $D=N^{2}/(qT)$, we have $|W|^{5/4}(qT)^{1/2}N \gg |W|qT=|W|N^{2}/D$. Consequently, the first term in the first factor can be ignored. Thus, the above expression is bounded by
			\[\lessapprox (|W|^{2}N+|W|^{5/4}(qT)^{1/2}N)^{1/2}
			(E(W)qT+|W|^{4}N+E(W)^{3/4}|W|(qT)^{1/2}N)^{1/2}.\]
			We split into cases based on whether $|W| > (qT)^{2/3}$ or not.
			
			Case $|W| \geq (qT)^{2/3}$: We observe that $|W|^{2}N \geq|W|^{5/4}(qT)^{1/2}N$ and so the first factor is dominated by $|W|^{2}N$. On the other hand, noting that
			\[N|W|^{4} \geq N|W|^{13/4}(qT)^{1/2} \geq E(W)^{3/4}|W|(qT)^{1/2}N\]
			and 
			\[N|W|^{4} \gg |W|^{3}qT \geq E(W)qT,\]
			the second factor is dominated by $N|W|^{4}$. Thus, we conclude that 
			\begin{equation}\label{lgcdcase1}
				\sum_{\substack{n_{1}, n_{2} \sim N, \\ (n_{1}n_{2}, q)=1 \\ (n_{1}, n_{2})\geq D}}\bigg|R\bigg(\frac{n_{1}}{n_{2}}, n_{1}n_{2}^{-1}\bigg)\bigg|^{3}
				\lessapprox N|W|^{3}.
			\end{equation}
			
			Case $|W| < (qT)^{2/3}$: As $|W|^{2}N \leq |W|^{5/4}(qT)^{1/2}N$, the first factor is dominated by $|W|^{5/4}(qT)^{1/2}N$. For the second factor, we observe that $E(W)^{3/4}|W|(qT)^{1/2}N > E(W)qT$. Thus, it follows that 
			\begin{align*}
				\sum_{\substack{n_{1}, n_{2} \sim N, \\ (n_{1}n_{2}, q)=1 \\ (n_{1}, n_{2})\geq D}}\bigg|R\bigg(\frac{n_{1}}{n_{2}}, n_{1}n_{2}^{-1}\bigg)\bigg|^{3}
				&\lessapprox
				(|W|^{5/4}(qT)^{1/2}N)^{1/2}
				(|W|^{4}N+E(W)^{3/4}|W|(qT)^{1/2}N)^{1/2}
				\\
				&\lessapprox
				N(qT)^{1/4}|W|^{21/8}+E(W)^{1/2}|W|^{1/2}N^{2}\bigg(\frac{(qT)^{1/2}|W|^{5/8}}{E(W)^{1/8}N}\bigg).
			\end{align*}
			Notice that, in this case, our assumption gives
			$E(W)^{1/8}N \geq (qT)^{1/2}|W|^{5/8}$. Consequently, the final term in the
			preceding inequality is
			$O(E(W)^{1/2}|W|^{1/2}N^2)$. We therefore obtain
			\begin{align}\label{lgcdcase2}
				\sum_{\substack{n_{1}, n_{2} \sim N, \\ (n_{1}n_{2}, q)=1 \\ (n_{1}, n_{2})\geq D}}\bigg|R\bigg(\frac{n_{1}}{n_{2}}, n_{1}n_{2}^{-1}\bigg)\bigg|^{3}
				\lessapprox
				N(qT)^{1/4}|W|^{21/8}+E(W)^{1/2}|W|^{1/2}N^{2}.
			\end{align}
			By combining \eqref{lgcdcase1} and \eqref{lgcdcase2}, we arrive at the desired conclusion, irrespective of the size of \(W\).
		\end{proof}
	\end{lemma}
	
	\begin{proof}[Proof of Proposition \ref{enb}]\label{profmainPro}
		First, we use Lemma \ref{econ3} to obtain
		\[E(W)\lessapprox N^{-2\sigma}\sum_{\substack{n_1,n_2 \sim N, \\ (n_1n_2, q)=1}}\bigg|R\bigg(\frac{n_{1}}{n_{2}}, n_{1}n_{2}^{-1}\bigg)\bigg|^{3}.\]
		By splitting the sum based on whether $(n_{1}, n_{2})\leq D$ nor not and applying \eqref{sgcdb} and Lemma \ref{lgcd}, we find
		\[E(W)\lessapprox N^{-2\sigma}(N|W|^{3}+N(qT)^{1/4}|W|^{21/8}+E(W)^{1/2}|W|^{1/2}N^{2}).\]
		so that
		\[ E(W) \lessapprox |W|N^{4-4\sigma}+|W|^{21/8}(qT)^{1/4}N^{1-2\sigma}+|W|^{3}N^{1-2\sigma},
		\]
		as required.
	\end{proof}
	
	\section{\texorpdfstring{$S_{3}$ bound and Proof of main proposition}
		{S3 bound and Proof of main proposition}}\label{S3bandpomp}
	In this section we establish the $S_{3}$ bound and complete the proof of Proposition \ref{Auxiliary theorem}.
	\subsection{\texorpdfstring{$S_{3}$ bound}
		{S3 bound}}\label{Pofmaipro1}
	\begin{proposition}[$S_{3}$ bound]\label{S3bound}
		Let $(qT)^{2/3}/2 \leq N \leq qT$. Then we have
		\[
		S_{3}\lessapprox (qT)^2 |W|^{\frac{3}{2}}+qT|W|N^{3-2\sigma}+qT|W|^2N^{\frac{3}{2}-\sigma}+(qT)^{\frac{9}{8}}|W|^{\frac{29}{16}}N^{\frac{3}{2}-\sigma}.
		\]
	\end{proposition}
	\begin{remark}
		The terms that set the limitations on our bounds are the first two terms, $(qT)^2|W|^{\frac{3}{2}}$ and $qT|W|N^{3-2\sigma}$. The term $qT|W|N^{3-2\sigma}$ also appears in Proposition \ref{S2B}, our bound for $S_2$. 
	\end{remark}
	\begin{proof}
		Recalling Proposition \ref{S3B1}, we have 
		\begin{align*}S_{3}\lessapprox \frac{N^{2}}{q^{2}M_{2}}\int\limits_{ v_{1} \in [1/2, 2]}\sum_{\substack{ b_1,b_2 \bmod q \\ (b_1b_2, q)=1}}| R(v_{1}, b_{1})|\sum_{|m_{i}|\sim M_{i}, \, 1\leq i \leq 3}\Bigg|\tilde{R}_{M_{2}}\bigg(\frac{m_{1}v_{1}+m_{3}}{m_{2}v_{1}}, b_{1}^{-1}b_{2}\bigg)
			\\
			\cdot\tilde{R}_{M_{2}}\bigg(\frac{m_{1}v_{1}+m_{3}}{m_{2}}, b_{2}\bigg)(b_{1}m_{1}-b_{2}m_{2}+m_{3},q)\Bigg|\:\mathrm{d}v_{1}+|W|.
		\end{align*}
		Applications of the Cauchy-Schwarz inequality and the second moment bound for $R$ (cf. Lemma \ref{secm}) give
		\[S_{3}\lessapprox \frac{N^{2}}{q^{2}M_{2}}S_{3,1}^{1/2}S_{3,2}^{1/2}+|W|,\]
		where
		\[S_{3,1}:=\int\limits_{ v \in [1/2, 2]}\sum_{\substack{b_1 \bmod q \\ (b_{1}, q)=1 }}| R(v, b_{1})|^{2}\:\mathrm{d}v \ll_{\epsilon} \phi(q)|W|,\]
		
		\begin{align*}
			S_{3,2}:=\int\limits_{ v \in [1/2, 2]}\sum_{\substack{b_1 \bmod q \\ (b_{1}, q)=1 }}\Bigg(\sum_{\substack{ b_2 \bmod q \\ (b_{2}, q)=1 }}\sum_{|m_{i}|\sim M_{i} 1\leq i \leq 3}\Bigg|\tilde{R}_{M_{2}}\bigg(\frac{m_{1}v+m_{3}}{m_{2}v}, b_{1}^{-1}b_{2}\bigg)
			\\
			\cdot\tilde{R}_{M_{2}}\bigg(\frac{m_{1}v+m_{3}}{m_{2}}, b_{2}\bigg)(b_{1}m_{1}-b_{2}m_{2}+m_{3},q)\Bigg|\Bigg)^{2}\:\mathrm{d}v.
		\end{align*}
		Using the Cauchy-Schwarz inequality again, we have that
		\begin{equation}\label{S32cs}
			S_{3, 2} \leq S_{3,3}^{1/2}S_{3,4}^{1/2},
		\end{equation}
		where
		\[
		S_{3,3}:=\int_{1/2}^{2}\sum_{\substack{b_1 \bmod q \\ (b_{1}, q)=1 }}\bigg(\sum_{\substack{b_2 \bmod q \\ (b_{2}, q)=1 }}\sum_{\substack{|m_{i}|\sim M_{i} \\ 1\leq i \leq 3}}\bigg|\tilde{R}_{M_{2}}\bigg(\frac{m_{1}v+m_{3}}{m_{2}v}, b_{1}^{-1}b_{2}\bigg)
		\bigg|^{2}(b_{1}m_{1}-b_{2}m_{2}+m_{3},q)\bigg)^{2}\:\mathrm{d}v,\]
		\[
		S_{3,4}:=\int_{1/2}^{2}\sum_{\substack{ b_1 \bmod q \\ (b_{1}, q)=1 }}\bigg(\sum_{\substack{b_2 \bmod q \\ (b_{2}, q)=1 }}\sum_{\substack{|m_{i}|\sim M_{i} \\ 1\leq i \leq 3}}\bigg|\tilde{R}_{M_{2}}\bigg(\frac{m_{1}v+m_{3}}{m_{2}}, b_{2}\bigg)
		\bigg|^{2}(b_{1}m_{1}-b_{2}m_{2}+m_{3},q)\bigg)^{2}\:\mathrm{d}v.\]
		We estimate $S_{3,3}$ and $S_{3,4}$ using Proposition \ref{bsoat}. To bound $S_{3,4}$, we define $f_{b}(v)=\psi(v)|\tilde{R}_{M_{2}}(v, -b)|^{2}$ for $(b,q) = 1$, where $\psi(v)$ is a smooth bump function supported on $v \asymp 1$ and such that $\psi(1)>0$. To bound $S_{3,3}$, we first change variables $u=v^{-1}$ and this yields that $S_{3,3}$ is bounded by
		\begin{align*}
			&\ll \int_{1/2}^{2}\sum_{\substack{b_1 \bmod q \\ (b_{1}, q)=1 }}\bigg(\sum_{\substack{b_2 \bmod q \\ (b_{2}, q)=1 }}\sum_{\substack{|m_{i}|\sim M_{i} \\ 1\leq i \leq 3}}\bigg|\tilde{R}_{M_{2}}\bigg(\frac{m_{1}+m_{3}u}{m_{2}u}, b_{1}^{-1}b_{2}\bigg)
			\bigg|^{2}(b_{1}m_{1}-b_{2}m_{2}+m_{3},q)\bigg)^{2}\mathrm{d}u
			\\
			&= \int_{1/2}^{2}\sum_{\substack{b_1 \bmod q \\ (b_{1}, q)=1 }}\bigg(\sum_{\substack{b_2 \bmod q \\ (b_{2}, q)=1 }}\sum_{\substack{|m_{i}|\sim M_{i} \\ 1\leq i \leq 3}}\bigg|\tilde{R}_{M_{2}}\bigg(\frac{m_{1}+m_{3}u}{m_{2}u}, b_{2}\bigg)
			\bigg|^{2}(b_{1}m_{1}-b_{1}b_{2}m_{2}+m_{3},q)\bigg)^{2}\mathrm{d}u.
		\end{align*}
		Note that $(b_{1}m_{1}-b_{1}b_{2}m_{2}+m_{3},q)=(m_{1}-b_{2}m_{2}+b_{1}^{-1}m_{3}, q)$ when $(b_{1}, q)=1$ and when $b_{1}$ runs through all elements of a reduced residue system, so does $b_{1}^{-1}$. Therefore we have
		\begin{align*}
			S_{3,3}&\ll  \int_{1/2}^{2}\sum_{\substack{b_1 \bmod q \\ (b_{1}, q)=1 }}\bigg(\sum_{\substack{b_2 \bmod q \\ (b_{2}, q)=1 }}\sum_{\substack{|m_{i}|\sim M_{i} \\ 1\leq i \leq 3}}\bigg|\tilde{R}_{M_{2}}\bigg(\frac{m_{1}+m_{3}u}{m_{2}u}, b_{2}\bigg)
			\bigg|^{2}(m_{1}-b_{2}m_{2}+b_{1}^{-1}m_{3},q)\bigg)^{2}\mathrm{d}u
			\\
			&=\int_{1/2}^{2}\sum_{\substack{b_1 \bmod q \\ (b_{1}, q)=1 }}\bigg(\sum_{\substack{b_2 \bmod q \\ (b_{2}, q)=1 }}\sum_{\substack{|m_{i}|\sim M_{i} \\ 1\leq i \leq 3}}\bigg|\tilde{R}_{M_{2}}\bigg(\frac{m_{1}+m_{3}u}{m_{2}u}, b_{2}\bigg)
			\bigg|^{2}(m_{1}-b_{2}m_{2}+b_{1}m_{3},q)\bigg)^{2}\mathrm{d}u.
		\end{align*}
		Then we use Proposition \ref{bsoat} with $f_{b}(u)=\psi(u)|\tilde{R}_{M_{2}}(u, -b)|^{2}$. First, we check that $f_{b}$ has the Fourier decay stated in the proposition. Recalling from \eqref{deftiR}, we have 
		\[
		|\tilde{R}_{M_{2}}(u, -b)|^{2}:= \int\limits_{ u' \in [1/2,\, 2]}\frac{NM_{2}}{q}\tilde{\psi}\bigg(\frac{NM_{2}}{q}(u-u')\bigg)| R(u', -b)|^{2}
		\:\mathrm{d}u' .
		\]
		Notice that $M_{2}\leq (qT)^{\epsilon}qT/N$. By \eqref{psi definition}, it follows that 
		\begin{align*}
			\mathcal{F}\bigg(\frac{NM_{2}}{q}\tilde{\psi}\bigg(\frac{NM_{2}}{q}u\bigg)\bigg)(\xi)=\hat{\tilde{\psi}}\bigg(\frac{q\xi}{NM_{2}}\bigg)&=2(qT)^{\epsilon}\hat{\tilde{\psi_{0}}}\bigg(\frac{2(qT)^{\epsilon}q\xi}{NM_{2}}\bigg)
			\\
			&\ll_{j}2(qT)^{\epsilon}\bigg(\frac{NM_{2}}{2(qT)^{\epsilon}q|\xi|}\bigg)^{j} \\
			&\ll_{j}(qT)^{\epsilon}\bigg(\frac{NM_{2}}{q|\xi|}\bigg)^{j}.
		\end{align*}
		Applying the convolution theorem and noting $\int |R(u, -b)|^{2}\mathbf{1}_{[1/2, 2]}(u)\:\mathrm{d}u \ll |W|^{2}$ since $|R(u,-b)|\leq|W|$, we see that 
		\[\sum_{\substack{1 \leq b\leq q \\ (b,q)=1}}|\hat{f_{b}}(\xi)| \lessapprox_{j} \phi(q)|W|^{2}\bigg(\frac{NM_{2}}{q|\xi|}\bigg)^{j}.\]
		On the other hand, since $|t| \leq T$ for $(t, \chi) \in W$, we have $|R(u', 1)|^{2} \gg |W|^{2}$ if $u'$ is a sufficiently small multiple of $1/T$ away from one. We also have $\tilde{\psi}(\frac{NM_{2}}{q}(1-u')) \gg 1$ if $u'$ is a sufficiently small multiple of $q/(NM_{2})$ away from one. Since $M_{2}\leq (qT)^{\epsilon}qT/N$ we have that $q/NM_2 \gtrapprox 1/T$. Therefore, we find from restricting $u'$ to a neighbourhood of 1 of width a small multiple of $\min\{1/T, q/(NM_{2})\}$, that
		\[
		f_{q-1}(1)=\psi(1)\int\limits\frac{NM_{2}}{q}\tilde{\psi}\bigg(\frac{NM_{2}}{q}(1-u')\bigg)| R(u', 1)|^{2}
		\:\mathrm{d}u'\gtrapprox \frac{NM_{2}}{qT}|W|^{2}.
		\]
		This implies 
		\[\sum_{b}|\hat{f_{b}}(\xi)| \lessapprox_{j} \phi(q)|W|^{2}\bigg(\frac{NM_{2}}{q|\xi|}\bigg)^{j} \lessapprox \phi(q)\frac{T}{|\xi|^{j}}\bigg(\frac{NM_{2}}{q}\bigg)^{j-1}f_{q-1}(1)\lessapprox_{j} \phi(q)\frac{T^{j}}{|\xi|^{j}}\sup_{b,u}f_{b}(u).
		\]
		Namely, we have the desired Fourier decay. Applying Proposition \ref{bsoat} then gives that $S_{3,3}$ and $S_{3,4}$ are both bounded by
		\begin{align*}
			\lessapprox \phi(q) M_{2}^{6}\bigg(\int_{\mathbb{R}}\sum_{\substack{b \bmod q \\ (b, q)=1}}|\tilde{R}_{M_{2}}(u, -b)|^{2}\psi(u)\:\mathrm{d}u\bigg)^{2}+\phi(q)^{2} M_{2}^{4}\int_{\mathbb{R}}\sum_{\substack{b \bmod q \\ (b, q)=1}}|\tilde{R}_{M_{2}}(u, -b)|^{4}\psi(u)^{2}\:\mathrm{d}u
		\end{align*} 
		This yields, by Lemmas \ref{secm} and \ref{fourthm}, 
		\begin{align*}
			S_{3,3}, S_{3,4}\lessapprox_{\epsilon} \phi(q)^{3} M_{2}^{6}|W|^{2}+
			\phi(q)^{3} M_{2}^{4}E(W).
		\end{align*} 
		The same bound holds for $S_{3,2}$ in view of \eqref{S32cs}. Gathering all these estimates together, we arrive at
		\begin{align*}
			S_{3}&\lessapprox_{\epsilon} \frac{N^{2}}{q^{2}M_{2}}(\phi(q)|W|)^{1/2}\phi(q)^{3/2}(M_{2}^{6}|W|^{2}+M_{2}^{4}E(W))^{1/2}+|W|
			\\
			&\leq N^{2}M_{2}^{2}|W|^{3/2}+N^{2}M_{2}|W|^{1/2}E(W)^{1/2}.
		\end{align*} 
		Since $M_{2}\lessapprox qT/N$, we get
		\begin{equation} \label{S3p}
			S_{3}\lessapprox (qT)^{2}|W|^{3/2}+qTN|W|^{1/2}E(W)^{1/2}.
		\end{equation}
		Finally, inserting the energy bound established in Section \ref{Engb} (see Proposition \ref{enb}) into the above inequality yields the result. 
	\end{proof}
	
	\subsection{Proof of Proposition \ref{Auxiliary theorem}}\label{Pofmaipro2}
	Finally, we combine our estimates for $S_1,S_2,S_3$ to prove Proposition \ref{Auxiliary theorem}. 
	\begin{proof}
		By Proposition \ref{W in terms of I} and \eqref{ss3}, we have
		\[
		|W| \ll_{\epsilon} N^{2-2\sigma}+N^{1-2\sigma}\bigg( \sum_{\overrightarrow{m}\in \mathbb{Z}^{3} \backslash \{\mathbf{0}\}}I_{\overrightarrow{m}} \bigg)^{1/3}=N^{2-2\sigma}+N^{1-2\sigma}( S_{1}+S_{2}+S_{3})^{1/3}.
		\]
		Substituting the bounds for $S_{1}$, $S_{2}$ and $S_{3}$ (see Proposition \ref{S_1 bound}, Proposition \ref{S2B} and  Proposition \ref{S3bound}) into the preceding expression yields
		\begin{align*}
			|W|^{3}N^{6\sigma-3} &\ll_{\epsilon} N^{3}+S_{1}+S_{2}+S_{3}
			\\
			&\lessapprox_{\epsilon} N^3 + (qT)^2|W|^{\frac{3}{2}} +qT|W|N^{3-2\sigma} + qT|W|^2 N^{\frac{3}{2}-\sigma}+(qT)^{\frac{9}{8}}|W|^{\frac{29}{16}}N^{\frac{3}{2}-\sigma}.
		\end{align*} 
		Note that the term bounding $S_2$, $qT|W|N^{3-2\sigma}$, appears in the bound for $S_3$. We rewrite the last inequality as 
		\begin{align*}
			|W| &\lessapprox_{\epsilon} N^{2-2\sigma} + (qT)^{\frac{4}{3}}N^{2-4\sigma}+(qT)^{\frac{1}{2}}N^{3-4\sigma}+qTN^{\frac{9}{2}-7\sigma}
			+(qT)^{\frac{18}{19}}N^{\frac{(72-112 \sigma)}{19}}.\label{New eqn 12.1}
		\end{align*}
		For $(qT)^{\frac{2}{3}} \leq N \leq (qT)^{\frac{5}{6}}$, the dominant term is 
		\begin{align*}
			|W| \lessapprox_{\epsilon} (qT)^{\frac{4}{3}}N^{2-4\sigma}.
		\end{align*}
		For $(qT)^{\frac{5}{6}} \leq N$, the dominant terms are
		\begin{align*}
			|W| \lessapprox_{\epsilon} N^{2-2\sigma}+(qT)^{\frac{1}{2}}N^{3-4\sigma}.
		\end{align*}
		This gives us a final large values estimate of 
		\begin{align*}
			|W| \lessapprox_{\epsilon} N^{2-2\sigma}+(qT)^{\frac{1}{2}}N^{3-4\sigma} +(qT)^{\frac{4}{3}}N^{2-4\sigma},
		\end{align*}
		completing the proof of Proposition \ref{Auxiliary theorem}. 
	\end{proof}
	\section{\texorpdfstring{Application to Dirichlet $L$-functions}
		{Application to Dirichlet L-functions}}\label{AppDL}
	The aim of this section is to establish Theorem \ref{zero density estimate} with the aid of our new large value estimate, Theorem \ref{Partial LVE}. Since Theorem \ref{zero density estimate} follows from Ingham's result \eqref{Inghaqt} when $\sigma \leq 0.7$ and Huxley's result \eqref{zeeshmh} when $\sigma \geq 0.8$, we may restrict our attention to the range $\sigma \in [0.7,0.8]$. Furthermore, we can restrict our analysis to primitive characters modulo $q$ as those non-primitive characters can be included by applying our final estimate for all factors of $q$ and summing.
	\subsection{The zero-detection method}\label{zedete}
	
	Our analysis begins with the now-standard zero-detection method (see, e.g., \cite[Chapter 12]{MontgomeryBook}). For the reader’s convenience and to maintain a self-contained exposition, we briefly outline the key ideas underlying this technique.
	
	Let $X,Y, T > 1$. Define
	\[
	M_{X}(s,\chi)= \sum_{n \leq X} \frac{\mu(n) \chi(n)}{n^{s}}.
	\]
	Then
	\[ L(s,\chi) M_X(s,\chi) = \sum_{n=1}^{\infty} c_n\chi(n) n^{-s}, \ \ \ \Re s > 1,
	\]
	where $c_n = \sum_{d | n , d \leq X} \mu(d)$. Observe that $c_1 = 1$, $c_n = 0$ for $1 < n \leq X$ and $|c_n| \ll_{\epsilon} n^{\epsilon}$.
	
	Introducing the weight $e^{-n/Y}$ and exploiting the Mellin inversion formula for $e^{-x}$, one finds, for $1/2 <\Re s < 1$,
	\begin{align}\label{zerodMI}
		e^{-1/Y} + \sum_{n > X} c_n \chi(n)n^{-s} e^{-n/Y} 
		& = \frac{1}{2\pi i} \int^{2 + i \infty}_{2 - i \infty} \Gamma(z) Y^{z} L(s + z,\chi) M_X(s+z,\chi) \mathrm{d}z 
		\notag
		\\
		& = \frac{1}{2\pi i} \int^{1/2 - \Re s + i \infty}_{1/2 - \Re s - i \infty} \Gamma(z) Y^{z} L(s + z,\chi) M_X(s+z,\chi) \mathrm{d}z 
		\notag
		\\
		& \phantom{=} {}+ L(s,\chi)M_X(s,\chi)+\varepsilon(\chi)\frac{\phi(q)}{q}M_{X}(1, \chi)Y^{1-s}\Gamma(1-s), 
	\end{align}
	where we picked up the residue at $z = 0$ and also one at $z=1-s$ if $\chi$ is a principal character while shifting the contour of integration to the line $\Re(z) = 1/2 - \Re(s)$.
	
	We consider \eqref{zerodMI} with $s=\rho= \beta + it$ and first treat the last term. As $\varepsilon(\chi)=1$ only if $\chi$ is principal, this term arises only when 
	\[L(s, \chi)= \zeta(s)\prod_{p | q}(1-p^{-s}).\]
	But if $|M_{X}(1, \chi_{0})Y^{1-\rho}\Gamma(1-\rho)| \geq 1/6$, by the exponential decay of the $\Gamma$ function on vertical lines, we see that $|t| \leq A \log qT$ for some absolute constant $A$, so this term is $\geq 1/6$ for at most $N(\sigma, A\log qT) \ll (\log qT)^{2}$ zeros.
	
	The tail of the sum $\sum_{n > Y\log^2 Y} c_n \chi(n)n^{-s} e^{-n/Y}$ is $o(1)$ as $Y\to \infty$. If $|t| \leq T$, then also the tails $|\Im z| \geq \log^{2} T$ of the integral in \eqref{zerodMI} become $o(1)$ as $T \to \infty$ if $X$ is polynomially bounded in $T$, say, in view of the exponential decay on vertical lines of the $\Gamma$ function and the trivial estimate $M_X(1/2 + iu, \chi) \ll X^{1/2}$. Therefore,  from \eqref{zerodMI} for $Y, T$ sufficiently large and $|t|\geq A\log qT$, we have either that
	\begin{equation} \label{I-ze}
		\bigg| \sum_{X< n \leq Y\log^{2}Y}c_n\chi(n) n^{-\rho}e^{-n/Y}\bigg| \gg 1, 
	\end{equation}
	or
	\begin{equation} \label{II-ze}
		\bigg| \int_{-\log ^{2} T}^{\log ^{2} T} L\bigg(\frac{1}{2}+i(t +u), \chi\bigg) M_{X}\bigg(\frac{1}{2}+i (t+u),\chi \bigg) Y^{\frac{1}{2}-\beta+iu}\Gamma\bigg(\frac{1}{2}-\beta+iu\bigg)\:\mathrm{d}u \bigg| \gg 1.
	\end{equation}
	The zeros $(\rho, \chi)$ with $\beta \geq \sigma$ and $|t| \leq T$ for which \eqref{I-ze} holds are referred to as \emph{class-I zeros} while those for which \eqref{II-ze} holds are called \emph{class-II zeros}. As a zero must belong to at least one of these classes we obtain, for $1/2 < \sigma < 1$,
	\begin{equation} \label{NOze}
		\sum_{\chi \bmod q} N(\sigma, T, \chi) \ll_{\epsilon} (|R_1|+|R_2|+1)(qT)^{\epsilon},
	\end{equation}	
	where $R_{1} = R_1(X,Y,T)$, resp. $R_2$, is the set of class-I, resp. class-II zeros and $|R_{j}|$ denotes their cardinality.
	
	For both of these classes we now consider a (saturated) subset $\tilde{R}_{j}$ of \emph{well-spaced zeros}; those are subsets of $R_j$ for which the imaginary parts of the zeros are well-spaced in the sense that 
	\begin{equation*} \label{disze}
		|t_{1} - t_{2}| \geq (qT)^{\epsilon}, 
	\end{equation*}
	for $(\rho_1, \chi_{1}) = (\beta_1 + it_1, \chi_{1})$ and $(\rho_2, \chi_{2}) = (\beta_2 + it_2, \chi_{2})$ with $\chi_{1}=\chi_{2}$ but $\rho_{1} \neq \rho_2$ belonging to $\tilde{R}_j$. Since  $N(\sigma, T+ 1, \chi) - N(\sigma, T, \chi) \ll \log qT$, as follows from e.g. \cite[Chapter 16]{DavenportBook}, one can always select a set of well-spaced zeros $\tilde{R}_j$ such that $|\tilde{R}_j| \gg |R_j| / ((qT)^\epsilon \log qT)$. Therefore, the estimate \eqref{NOze} remains valid if we replace $R_j$ by $\tilde{R}_j$.  
	
	\subsection{The contribution of the class-II zeros}
	\label{TconclassII}
	We begin by analyzing the contribution of the well-spaced class-II zeros $\tilde{R}_{2}$. If we set 
	\begin{equation*} 
		\bigg|L\bigg(\frac{1}{2}+i\gamma_r, \chi_{r}\bigg)\bigg|=\max_{-\log ^{2} T\leq u \leq \log ^{2} T}\bigg|L\bigg(\frac{1}{2}+it_{r} +iu,\chi_{r}\bigg)\bigg|,
	\end{equation*}
	where $t_r$ are the imaginary parts of the class-II zeros, then we find
	\begin{equation*} 1 \ll (qT)^{\epsilon} Y^{1/2-\sigma} \bigg|L\bigg(\frac{1}{2}+i\gamma_{r},\chi_{r}\bigg)\bigg|,\ \ \ r=1,2, ... ,|\tilde{R}_{2}|,
	\end{equation*}
	where we have set $X = (qT)^{\epsilon}$. Raising the inequality to the fourth power and exploiting the well-spacedness of the ordinates $\gamma_{r}$ (which follows from the class-II zero spacing condition), we apply the fourth moment estimate \cite[Theorem 10.3]{MontgomeryBook} to obtain 
	\begin{equation*} 
		|\tilde{R}_{2}| \ll (qT)^{\epsilon} Y^{2-4 \sigma} \sum_{r \leq |\tilde{R}_{2}|} \bigg|L\bigg(\frac{1}{2}+i\gamma_{r},\chi_r\bigg)\bigg|^{4} \ll (qT)^{1+\epsilon}Y^{2-4\sigma}.
	\end{equation*} 
	Therefore, upon choosing $Y=(qT)^{1/2}$, we obtain $|\tilde{R}_2| \lessapprox (qT)^{2(1-\sigma)}$ and this concludes the analysis of the class-II zeros.

	\subsection{The contribution of the class-I zeros}
	\label{TconclassI}
	The rest of the argument is then to bound the contribution of the class-I zeros. By a dyadic subdivision of the interval $(X, Y \log^2 Y]$, one can find $X \leq N < Y\log^2 Y$ such that 
	\begin{equation} \label{proze}
		\bigg| \sum_{N< n \leq 2N}c_n\chi(n) n^{-\rho}e^{-n/Y} \bigg| \gg \frac{1}{\log Y}
	\end{equation}
	for at least  $|\tilde{R}_1|\log 2/\log (Y\log^2 Y)$ zeros $\rho$ of $\tilde{R}_1$. The elements of $\tilde{R}_1$ that additionally satisfy \eqref{proze} are called \emph{representative well-spaced zeros} and this subset will be denoted as $R$. We remark that \eqref{NOze} remains valid upon replacing $|R_1|$ by $|R|$. If we now set $a_n = \varpi c_n$ for a sufficiently small $\varpi \asymp N^{-\epsilon}$ such that $|a_n| \leq 1$, we find
	\begin{equation*} 
		\bigg| \sum_{N< n \leq 2N}a_n\chi(n) n^{-it} \bigg| \gtrapprox N^{\sigma}.
	\end{equation*}
	for $(t, \chi) \in R$.
	Let
	\[
	D_{N}(t, \chi)=\sum_{N < n \leq 2N} a_{n}\chi(n)n^{it}.
	\]
	Thus it suffices to bound $|W|$ in the range
	$(qT)^{\epsilon} \leq N \lessapprox (qT)^{1/2}$, where $W$ is a finite set of
	pairs $(t,\chi)$ satisfying $|t|\leq T$ and
	$|D_N(t,\chi)|\gtrapprox N^\sigma$ with $\sigma\in[0.7,0.8]$.
	Moreover, for any two distinct pairs $(t,\chi),(t',\chi')\in W$, either
	$\chi\neq\chi'$ or $|t-t'|\geq(qT)^\epsilon$.
	Our analysis below will rely on the new large value estimate and the classical mean value theorem to powers of $D_N$, depending on the size of $N$. 
	
	Let
	\begin{align*}
		v:=\frac{5}{3+5\sigma}, \quad w_{\text{max}} := \frac{4}{3(1+\sigma)}.
	\end{align*}
	We choose $a$ so that $v \leq a \leq w_{\text{max}}$. 
	Let 
	\begin{align*}
		b := \min\bigg(\frac{15a(1-\sigma)}{18-20\sigma},1\bigg). 
	\end{align*} 
	These quantities $a,b$ will be helpful in splitting the possible lengths $N$ of $D_N$ into cases to apply either Theorem \ref{Partial LVE} or the mean value theorem to. We prove some useful inequalities for $a,b$, namely that $\frac{2}{3}< a< \frac{5}{6} < b \leq 1 $ and $2-2\sigma \leq 1+b(1-2\sigma) \leq 3a(1-\sigma)$ for all $\sigma \in [0.7,0.8]$, for any $a$ such that $v \leq a \leq w_{\text{max}}$.
	
	Note that $v,w_{\max}$ are decreasing functions of $\sigma$. At $\sigma = 0.8, v = \frac{5}{7} > \frac{2}{3}$ and at $\sigma = 0.7$, $w_{\max} = \frac{40}{51} < \frac{5}{6}$. On the other hand, for $\sigma \in [0.7,0.8]$, we have $\frac{15a(1-\sigma)}{18-20\sigma} \geq \frac{15v(1-\sigma)}{18-20\sigma}$ and the right hand side is an increasing function of $\sigma$; at $\sigma =0.7$ the right hand side has minimum value $\frac{45}{52} > \frac{5}{6}$. Thus for any such choice of $a$, $\frac{2}{3}<a<\frac{5}{6}<b\leq 1$ when $\sigma \in [0.7,0.8]$.
	
	Next, $2-2\sigma \leq 1+b(1-2\sigma)$ is true as $b \leq 1$. To show that $1+b(1-2\sigma) \leq 3a(1-\sigma)$, note that $1+b(1-2\sigma)$ is a decreasing function of $a$ while $3a(1-\sigma)$ is an increasing function of $a$. Therefore to show that $1+b(1-2\sigma) \leq 3a(1-\sigma)$, it suffices to consider the case $a=v$. 
	
	When $b=1$, the inequality $2-2\sigma < 3v(1-\sigma)$ is equivalent to $2(3+5\sigma) < 15$, true for $\sigma \leq 0.8$. When $b = \frac{15v(1-\sigma)}{18-20\sigma}$ the inequality $1+b(1-2\sigma) < 3v(1-\sigma)$ is equivalent to $\frac{250\sigma^2-375\sigma+141}{15(13-10\sigma)(1-\sigma)} > 0$ upon rearranging. Since $\sigma\in[0.7,0.8]$, it suffices to show that the numerator is positive.
	The quadratic $250\sigma^2-375\sigma+141$ has discriminant $-375$ and positive
	leading coefficient and is therefore positive for all real $\sigma$. Hence
	$1+b(1-2\sigma)<3a(1-\sigma)$
	for all choices of $a\geq v$.
	
	Since $N \geq (qT)^{\epsilon}$, we can choose $k \ll_{\epsilon} 1$ so that 
	\begin{align*}
		(qT)^{a} \leq N^k \leq (qT)^{\frac{3a}{2}}.
	\end{align*}
	If $N \leq (qT)^{a/2}$ then this is clearly possible, if $N \geq (qT)^{a/2}$ it is also possible with $k=2$ as $N \lessapprox(qT)^{\frac{1}{2}}$ and $a > \frac{2}{3}$.
	
	We split into cases now based on the size of $N^k$,
	\begin{enumerate}
		\item[1] $(qT)^b \leq  N^k \leq (qT)^{\frac{3a}{2}} $,
		\item[2] $(qT)^{\frac{5}{6}} \leq N^k \leq (qT)^b$,
		\item[3] $(qT)^{a} \leq N^k \leq (qT)^{\frac{5}{6}}$.
	\end{enumerate}
	In all cases $(qT)^a \leq N^k$ and $a \geq 2/3$ so we may apply Theorem \ref{Partial LVE} in any case. 
	
	\textbf{Case 1:  $(qT)^b \leq N^k \leq (qT)^{\frac{3a}{2}}$}
	
	By the mean value theorem applied to $D_N^k$,
	\begin{align*}
		|W| \lessapprox_{\epsilon} N^{2k(1-\sigma)}+qTN^{k(1-2\sigma)} \leq (qT)^{3a(1-\sigma)}+(qT)^{1+b(1-2\sigma)} \ll (qT)^{3a(1-\sigma)},
	\end{align*}
	since $1+b(1-2\sigma) \leq 3a(1-\sigma)$. 
	
	\textbf{Case 2: $(qT)^{\frac{5}{6}} \leq N^k \leq (qT)^b$}
	
	We apply Theorem \ref{Partial LVE} to $D_N^k$ with $q_1=q$, valid as $(qT)^{2/3} \leq N^k$. Since $qT<(N^k)^{\frac{6}{5}}$, we have the bound
	\[
	|W| \lessapprox_{\epsilon} (qT)^{\frac{1}{2}}N^{k(3-4\sigma)}+N^{k(2-2\sigma)}.
	\]
	Note that $(qT)^{\frac{1}{2}}N^{k(3-4\sigma)} \leq N^{k(\frac{18-20\sigma}{5})}$ and $N^{k(2-2\sigma)} \leq N^{k(\frac{18-20\sigma}{5})}$. Thus in this range, 
	\begin{align*}
		|W| \lessapprox_{\epsilon} N^{k(\frac{18-20\sigma}{5})} \leq (qT)^{b ({\frac{18-20\sigma}{5}})} \leq (qT)^{3a(1-\sigma)},
	\end{align*}
	by our choice of $b$.
	
	Thus we only have Case 3 to consider, when $(qT)^a \leq N^k \leq (qT)^{\frac{5}{6}}$, which we analyse in the following. Let $q_1 \mid q$. Applying Theorem \ref{Partial LVE} to $ D_N^k$, we have that
	\begin{align*}
		|W| \lessapprox_{\epsilon} N^{k(2-2\sigma)}+ qq_1^{-\frac{1}{2}}T^{\frac{1}{2}}N^{k(3-4\sigma)}+ q q_1^{\frac{1}{3}}TN^{k(2-4\sigma)} + qTN^{k(\frac{12-20\sigma}{5})}.
	\end{align*}
	Unless we have an optimal bound of $N^{k(2-2\sigma)}$, the dominant term is decided by the size of $q_1$. If $q_1 > N^{\frac{6k}{5}}$ then the dominating term is $qq_1^{\frac{1}{3}}TN^{k(2-4\sigma)}$. If $\frac{N^{\frac{6k}{5}}}{T} < q_1 < N^{\frac{6k}{5}}$ then the dominating term is $qT(N^k)^{\frac{12-20\sigma}{5}}$. Lastly, if $q_1 < \frac{N^{\frac{6k}{5}}}{T},$ then the term $qq_1^{-\frac{1}{2}}T^{\frac{1}{2}}N^{k(3-4\sigma)}$ dominates.
	
	We now prove the following bound for $|W|$:
	\begin{lemma}\label{p subdiv large values estimate}
		Suppose that $W$ is a finite set of pairs $(t, \chi)$,  where $|t| \leq T$ and for $(t, \chi)\neq (t', \chi')$ either $\chi \neq \chi'$ or $|t-t'|\geq (qT)^{\epsilon}$. Assume that $|D_N(t,\chi)|>N^{\sigma}$ for $(t,\chi)\in W$ and that $\sigma \in [0.7,0.8]$. If $N \lessapprox (qT)^{1/2}$ then for a divisor $q_1 \mid q$, we have that
		\begin{align*}
			|W| \lessapprox_{\epsilon} (qT)^{\frac{15(1-\sigma)}{3+5\sigma}}+(q_1^{\frac{1}{3}}qT)^{\frac{3(1-\sigma)}{1+\sigma}}+(qT(q_1T)^{-\frac{1}{2}})^{\frac{3(1-\sigma)}{\sigma}}+(q_1T)^{-\frac{1}{2}}(qT)^{\frac{21-20\sigma}{6}}. \label{Partial ZDE bound}
		\end{align*}
		If $q$ is $T$-smooth then we have that
		\begin{align*}
			|W| \lessapprox_{\epsilon} (qT)^{\frac{15(1-\sigma)}{3+5\sigma}}.
		\end{align*}
	\end{lemma}
	\begin{proof}
		From the above, we may assume that $(qT)^a \leq N^k \leq (qT)^{\frac{5}{6}}$ for some positive integer $k \ll 1$. Let 
		\begin{align*}
			w_q:= \frac{\log(q^{4/3}T)}{(1+\sigma)\log{qT}} \quad  w_{q_1}:= \frac{\log(q_1^{\frac{1}{3}}qT)}{(1+\sigma)\log{qT}} \quad  z_{q_1}:=\frac{\log(qT(q_1T)^{-\frac{1}{2}})}{\sigma \log{qT}}.
		\end{align*}
		First we deal with the case where $q$ has a divisor $q_1$. We split the proof into four cases, according to the size of $q_1$:
		\begin{enumerate}
			\item $(qT)^v \leq q_1^{\frac{5}{6}}$,
			\item $q_1^{\frac{5}{6}} \leq (qT)^{v} \leq (q_1T)^{\frac{5}{6}}$,
			\item $(q^{\frac{3-\sigma}{3}}T)^{\frac{5}{3(1+\sigma)}}<(q_1T)^{\frac{5}{6}} \leq (qT)^v$
			\item $ (q_1T)^{\frac{5}{6}}<(q^{\frac{3-\sigma}{3}}T)^{\frac{5}{3(1+\sigma)}}$  and $(q_1T)^{\frac{5}{6}} \leq (qT)^v$.
		\end{enumerate}
		
		\noindent \textbf{Case 1: $(qT)^v \leq q_1^{\frac{5}{6}}$:}
		
		In this case, let $a= w_{q_1}$. Then since $q_1^{\frac{5}{6}} \geq (qT)^v$, we have
		\begin{align*}
			(qT)^a = (q_1^{\frac{1}{3}}qT)^{\frac{1}{1+\sigma}} \leq (q_1^{\frac{1}{3}+\frac{5}{6v}})^{\frac{1}{1+\sigma}} = q_1^{\frac{5}{6}}.
		\end{align*}
		Also, we have that $v \leq a=w_{q_1} \leq w_{\text{max}}$.
		
		Then if $(qT)^a \leq N^k \leq q_1^{\frac{5}{6}}$, by Theorem \ref{Partial LVE} applied to $D_N^k$, we have that
		\begin{align*}
			|W| \lessapprox_{\epsilon} q_1^{\frac{1}{3}}qTN^{k(2-4\sigma)} = q_1^{\frac{1}{3}}(qT)^{1+a(2-4\sigma)} = (qT)^{3a(1-\sigma)},
		\end{align*}
		by the choice of $a$.
		
		For $q_1^{\frac{5}{6}} \leq N^k \leq (q_1T)^{\frac{5}{6}}$, we have that $\frac{N^{\frac{6k}{5}}}{T} \leq q_1 \leq N^{\frac{6k}{5}}$. Thus by Theorem \ref{Partial LVE} applied to $D_N^k$, we have that 
		\begin{align*}
			|W| \lessapprox_{\epsilon} qTN^{k(\frac{12-20\sigma}{5})} < qT \cdot q_1^{\frac{12-20\sigma}{6}}.
		\end{align*}
		The inequality $qT \cdot q_1^{\frac{12-20\sigma}{6}}< q_1^{\frac{1}{3}}(qT)^{1+a(2-4\sigma)}$ is equivalent to $(qT)^a \leq q_1^{\frac{5}{6}}$ for $\sigma \in [0.7,0.8]$, which we have.
		
		Finally, if $(q_1T)^{\frac{5}{6}} \leq N^k \leq (qT)^{\frac{5}{6}}$ then $q_1 T < N^{\frac{6k}{5}}$ so by Theorem \ref{Partial LVE} applied to $D_N^k$, we have that
		\begin{align*}
			|W| &\lessapprox qT (q_1T)^{-\frac{1}{2}}N^{k(3-4\sigma)} \leq qT(q_1T)^{\frac{12-20\sigma}{6}}+(q_1T)^{-\frac{1}{2}}(qT)^{\frac{21-20\sigma}{6}} \\
			&\leq (qT)^{3a(1-\sigma)}+(q_1T)^{-\frac{1}{2}}(qT)^{\frac{21-20\sigma}{6}}.
		\end{align*}
		The second term may be larger for $\sigma \leq \frac{3}{4}$.
		
		\noindent \textbf{Case 2: $q_1^{\frac{5}{6}} \leq (qT)^v \leq (q_1T)^{\frac{5}{6}}$:}
		
		In this case, let $a=v$ so $v \leq a < w_{\text{max}}$. If $(qT)^a \leq N^k \leq (q_1T)^{\frac{5}{6}}$ then by Theorem \ref{Partial LVE} applied to $D_N^k$, we have that 
		\begin{align*}
			|W| \lessapprox_{\epsilon} qTN^{k(\frac{12-20\sigma}{5})} \leq (qT)^{1+a(\frac{12-20\sigma}{5})}  = (qT)^{3a(1-\sigma)},
		\end{align*}
		by the choice of $a$.
		
		If $(q_1T)^{\frac{5}{6}} \leq N^k \leq (qT)^{\frac{5}{6}}$ then again by Theorem \ref{Partial LVE},
		\begin{align*}
			|W| \lessapprox_{\epsilon} qT (q_1T)^{-\frac{1}{2}}N^{k(3-4\sigma)} &\leq qT(q_1T)^{\frac{12-20\sigma}{6}}+(q_1T)^{-\frac{1}{2}}(qT)^{\frac{21-20\sigma}{6}} \\
			&\leq (qT)^{3a(1-\sigma)} +(q_1T)^{-\frac{1}{2}}(qT)^{\frac{21-20\sigma}{6}},
		\end{align*}
		where $qT(q_1T)^{\frac{12-20\sigma}{6}} \leq (qT)^{3a(1-\sigma)}$ since $q_1T \geq (q_1T)^{\frac{6v}{5}}$. 
		
		\noindent \textbf{Case 3: $(q^{\frac{3-\sigma}{3}}T)^{\frac{5}{3(1+\sigma)}}<(q_1T)^{\frac{5}{6}} \leq (qT)^v$:}
		
		For some values of $q,T$ this inequality may not ever hold but we suppose it does. In this case, we choose $a=z_{q_1}$. Since $(q^{\frac{3-\sigma}{3}}T)^{\frac{5}{3(1+\sigma)}} < (q_1T)^{\frac{5}{6}}$, we have that $z_{q_1} \leq w_q$ and since $(q_1T)^{\frac{5}{6}} \leq (qT)^v$, we have that $z_{q_1} \geq v$. Then $ v \leq a = z_{q_1} \leq w_{q} \leq w_{\text{max}}$.
		
		If $(qT)^a \leq N^k \leq (qT)^{\frac{5}{6}}$ then by Theorem \ref{Partial LVE} applied to $D_N^k$, we have that 
		\begin{align*}
			|W| \lessapprox_{\epsilon} qT(q_1T)^{-\frac{1}{2}}N^{k(3-4\sigma)} &\lessapprox (qT)^{1+a(3-4\sigma)}(q_1T)^{-\frac{1}{2}} + (q_1T)^{-\frac{1}{2}}(qT)^{\frac{21-20\sigma}{6}} \\
			&=(qT)^{3a(1-\sigma)}+(q_1T)^{-\frac{1}{2}}(qT)^{\frac{21-20\sigma}{6}},
		\end{align*}
		by our choice of $a$. 
		
		\noindent \textbf{Case 4: $(q_1T)^{\frac{5}{6}}<(q^{\frac{3-\sigma}{3}}T)^{\frac{5}{3(1+\sigma)}}$} and $(q_1T)^{\frac{5}{6}} \leq (qT)^v$\textbf{:}
		
		Taking $q_1=q$ in Theorem \ref{Partial LVE} yields 
		\begin{equation*}
			|W| \lessapprox_{\epsilon} (qT)^{\frac{15(1-\sigma)}{3+5\sigma}} + (q^{4/3}T)^{\frac{3(1-\sigma)}{1+\sigma}}.
		\end{equation*}
		This is because $v \leq 5/6$ so we would end up in Case 1 or 2 above.
		Observe that the first term on the right hand side coincides with the corresponding term in the desired lemma. Moreover, the second term, namely $(qT)^{3w_q(1-\sigma)}$, is dominated by the third term in the lemma, $(qT)^{3z_{q_{1}}(1-\sigma)}$, since $z_{q_{1}} \geq w_q$ in this case. This completes the discussion of this case, proving the first part of the Lemma.
		
		Now suppose that $q$ is $T$-smooth. In this case, we have a divisor $q_1 \mid q$ such that $q_1^{\frac{5}{6}} < (qT)^v \leq (q_1T)^{\frac{5}{6}}$. If $q^{\frac{5}{6}} < (qT)^{v}$ then since $v<5/6$, we have $(qT)^v < (qT)^{\frac{5}{6}}$. Otherwise, we take $q_1$ to be the largest divisor of $q$ such that $q_1^{\frac{5}{6}} <(qT)^v$. Then since $q$ is $T$-smooth, we have that $(qT)^v {\leq} (q_1T)^{\frac{5}{6}}$.
		
		With this $q_1$, we follow Case 2 above, choosing $a=v$ so that for $(qT)^a \leq N^k \leq (q_1T)^{\frac{5}{6}}$, we have $|W| \lessapprox_{\epsilon} (qT)^{3a(1-\sigma)}$. 
		
		Since $q$ is $T$-smooth, we can pick $q_2>q_1$ so that $q_2 \leq q_1T$. Then in the range $q_2^{\frac{5}{6}} \leq N^k \leq (q_2T)^{\frac{5}{6}}$, we have that $|W| \lessapprox_{\epsilon} (qT)^{3a(1-\sigma)}$ too. We can choose divisors to cover the length of the whole interval $(qT)^a \leq N^k \leq (qT)^{\frac{5}{6}}$ in this way so we conclude that $|W| \lessapprox_{\epsilon} (qT)^{3a(1-\sigma)}$.
	\end{proof}
	\begin{remark}
		The cases in the above may only hold for certain $\sigma \in [0.7,0.8]$. When $q_1=q$, we get the bound $|W| \lessapprox_{\epsilon} (q^{\frac{4}{3}}T)^{\frac{3(1-\sigma)}{1+\sigma}}+(qT)^{\frac{15(1-\sigma)}{3+5\sigma}}$, which is valid for any $q$. As seen in Case 4, this does better than subdivision with a divisor that is too small.
		
		We also may do better with extra divisors of appropriate size, given that $T$ is sufficiently large, similar to the case when $q$ is $T$-smooth. When $\sigma \geq 3/4$, the term $(q_1T)^{-\frac{1}{2}}(qT)^{\frac{21-20\sigma}{6}}$ can be dropped. 
	\end{remark}
	\subsection{Conclusion}\label{conclu}
	Returning to the zero density estimate, combining the above bound for $|W|$ (i.e., the estimation of the class-I zeros)
	with the estimate for the class-II zeros in Subsection \ref{TconclassII}, we
	deduce our general zero density estimate from \eqref{NOze} that, for any divisor $q_1\mid q$,
	\begin{align*}
		\sum_{\chi \bmod q} N(\sigma,T,\chi)
		&\lessapprox
		(q_1^{1/3}qT)^{\frac{3(1-\sigma)}{1+\sigma}}
		+
		\bigl(qT(q_1T)^{-1/2}\bigr)^{\frac{3(1-\sigma)}{\sigma}}+
		(q_1T)^{-\frac{1}{2}}(qT)^{\frac{21-20\sigma}{6}}
		+
		(qT)^{\frac{15(1-\sigma)}{3+5\sigma}}.
	\end{align*}
	The second part of the theorem states that when $q_1 \mid q$ and $q_1 \geq \sqrt{q}$, we have that
	\begin{align}
		\sum_{\chi \bmod{q}} N(\sigma,T,\chi) \lessapprox (q_1^{\frac{1}{3}}q^2T^2)^{1-\sigma}+(q^3T^{\frac{9}{4}}q_1^{-\frac{3}{4}})^{1-\sigma}+ (qT)^{B(1-\sigma)} + (qT)^{\frac{30}{13}(1-\sigma)}, \label{Ingham combined ZDE}
	\end{align}
	where $\frac{1}{2} < \sigma < 1$ and $B=\frac{37+3\beta-\sqrt{9\beta^2+222\beta-71}}{12}$ with $\beta = \frac{\log{q_1T}}{\log{qT}} \geq 1/2$.
	
	This may be deduced by the general zero density estimate with
	\eqref{Inghaqt}. Writing each term in the general zero density estimate in the
	form $(qT)^{C(\sigma)(1-\sigma)}$, we find that in each case $C(\sigma)$ is
	decreasing, whereas the exponent $3/(2-\sigma)$ in Ingham's estimate
	\eqref{Inghaqt} is increasing in $\sigma$. Thus, after combining these two
	bounds, the resulting estimate for $\sum_{\chi}N(\sigma,T,\chi)$ is
	$(qT)^{C(\sigma_0)(1-\sigma_0)}$, where $\sigma_0$ is determined by
	$C(\sigma_0)=3/(2-\sigma_0)$. As an example, we derive the first term in
	\eqref{Ingham combined ZDE}.
	We have
	\begin{align*}
		(q_1^{1/3}qT)^{\frac{3(1-\sigma)}{1+\sigma}}
		&=(qT)^{{\frac{3}{1+\sigma}\frac{\log(q_1^{1/3}qT)}{\log(qT)}}(1-\sigma)}.
	\end{align*}
	Thus the coefficient to compare with Ingham is
	\[
	\frac{3}{1+\sigma}\frac{\log(q_1^{1/3}qT)}{\log(qT)},
	\]
	which is decreasing in $\sigma$. The intersection is given by
	\begin{align*}
		\frac{3}{1+\sigma}\frac{\log(q_1^{1/3}qT)}{\log(qT)}
		=\frac{3}{2-\sigma}.
	\end{align*}
	Upon rearranging, we find that
	\[
	\frac{3}{2-\sigma}
	=1+\frac{\log(q_1^{1/3}qT)}{\log(qT)}.
	\]
	Therefore the combined contribution of the first term with Ingham is
	\begin{align*}
		(qT)^{\big(1+\frac{\log(q_1^{1/3}qT)}{\log(qT)}\big)(1-\sigma)}
		=(q_1^{1/3}q^2T^2)^{1-\sigma}.
	\end{align*}
	The other terms in \eqref{Ingham combined ZDE} are obtained in the same way. For the last part of Theorem \ref{zero density estimate}, we use the last part of Lemma \ref{p subdiv large values estimate}.
	\begin{remark}
		The worst case for our zero density estimate is when $T=1$ and there are no factors $q_1$ that are of an appropriate size; they may be too small (or too large). Notably, this happens for prime moduli $q$. In this case, the critical situation is at $\sigma = 5/7$ with $N=q^{7/9}$ and $|W|=q^{2/3}$. This is a case where our lack of subdivision causes us to do worse as the largest term is $q^{4/3}N^{2-4\sigma} = q^{2/3}$, which comes from our $S_3$ bound, $S_3 \lessapprox (qT)^2|W|^{3/2}$. This term corresponds to square root cancellation in the $R$ functions but no cancellation in the sums over $m$. To improve this term, we would need to find some joint cancellation in $R$ and the $m$ sums, which seems quite difficult.
	\end{remark}

\end{document}